\RequirePackage{rotating}
\documentclass[smallextended,envcountsect,envcountsame]{svjour3}
\usepackage[margin=1.0in]{geometry} 
\usepackage{algorithm}
\usepackage{comment}
\usepackage{graphicx}
\usepackage{mathtools}
\usepackage{mathrsfs}
\usepackage{fix-cm}
\usepackage{multirow}
\usepackage{amsmath}
\usepackage{amssymb}
\usepackage{latexsym}
\usepackage{dsfont}
\usepackage{xcolor}
\usepackage{cite}
\usepackage{dsfont}
\usepackage{enumitem}
\usepackage{todonotes}
\usepackage{hyperref}
\usepackage{booktabs}
\hypersetup{
colorlinks=true,
linkcolor=blue,
citecolor=darkgreen,
urlcolor=blue}
\definecolor{darkgreen}{RGB}{0,150,0}
\definecolor{C0}{RGB}{31,119,180}
\definecolor{C1}{RGB}{255, 127, 14}
\usepackage{algpseudocode}

\newtheorem{experiment}{Experiment}
\newtheorem{assumption}{Assumption}
\let\oldexperiment\experiment
\renewcommand{\experiment}{\oldexperiment\normalfont}

\def\tto{\rightrightarrows}

\def\hat{\widehat}

\def\R{\mathbb{R}}
\def\N{\mathbb{N}}

\def\dom{\mbox{\rm dom}\,}

\def\dist{\mbox{\rm dist}\,}

\def\tilde{\widetilde}

\def\N{\mathbb{N}}

\def\tt{\tilde t}

\newcounter{count}

\DeclareMathOperator{\1}{\mathds{1}}

\newcommand{\Rex}{\R\cup\{+\infty\}}
\let\epsilon\varepsilon
\DeclareMathAlphabet{\mathpzc}{OT1}{pzc}{m}{it}

\usepackage{framed}
\usepackage{mdframed}
\usepackage{pgfplots}
\pgfplotsset{compat = newest}
\usepackage{caption}
\usepackage{subcaption}

 \DeclareMathOperator{\prox}{prox}

 \DeclareMathOperator*{\argmin}{argmin}

    \newcommand{\rr}{{{r}}}
    \newcommand{\rA}{\ell}
    \newcommand{\rB}{{\hat{r}}}
    \newcommand{\rC}{{{r}_{1}}}
    \newcommand{\rD}{{{r}_{2}}}
    
    \newcommand{\lA}{\delta}

\newcommand{\expo}{\varrho}

\graphicspath{{../Figures/}}
\usepackage{bm}

\newcommand{\Intf}[1]{\mathbf{E}_{#1}}
\DeclareMathOperator{\vPhi}{\mathrm{\Phi}}

\newenvironment{proofof}[1]
{\par\noindent{\itshape Proof of Theorem~\ref{#1}.}\ }
{\hfill$\square$\par}

\newenvironment{proofoflemma}[1]
{\par\noindent{\itshape Proof of Lemma~\ref{#1}.}\ }
{\hfill$\square$\par}

\newcommand{\KLname}{Kurdyka--{\L}ojasiewicz~}
\newcommand{\KLshort}{KL}

\usepackage{lineno}

\begin{document}
\titlerunning{Proximal subgradient method for nonconvex stochastic optimization}
\title{A proximal subgradient method for nonconvex stochastic optimization under the \KLname condition
\thanks{Research of the first three authors  was supported by  Centro de Modelamiento Matem\'atico (CMM), ACE210010 and FB210005, BASAL
funds for center of excellence from ANID-Chile. The third author was supported by  ANID-Chile grant: Fondecyt Regular 1240335,  Fondecyt Regular 1240120, Fondecyt Regular 1261728. The fourth author was  partially supported by grant PID2022-136399NB-C21 funded by
ERDF/EU and by MICIU/AEI/10.13039/501100011033.}}
\subtitle{}
\author{\mbox{Felipe Atenas} \and \mbox{Alejandro Jofr\'e} \and \mbox{Pedro P\'erez-Aros}  \and  \mbox{David Torregrosa-Bel\'en}}
\institute{
Felipe Atenas \at Centro de Modelamiento Matem\'atico (CNRS IRL2807),
Universidad de Chile, Santiago, Chile\\
\email{fatenas@cmm.uchile.cl}
\and
Alejandro Jofré and Pedro Pérez-Aros \at Departamento de Ingenier\'ia Matem\'atica and Centro de Modelamiento Matem\'atico (CNRS IRL2807), Universidad de Chile, Santiago, Chile\\
\email{ajofre@uchile.cl, pperez@dim.uchile.cl}   \and
David Torregrosa-Bel\'en \at  Department of Mathematics, University of Alicante, Alicante, Spain\\
\email{david.torregrosa@ua.es}}

\date{\today}
\maketitle

\begin{abstract}
This work introduces a proximal stochastic subgradient method for minimizing the sum of an expected cost, whose integrand is potentially nonsmooth and nonconvex, and a lower semicontinuous, prox-bounded function. We target a broad class of integrands obeying a nonsmooth, localized variant of the descent lemma in the decision variable, a structural assumption that simultaneously covers smooth losses with Lipschitz gradient and differences of such losses with convex functions. At each iteration the expected cost is replaced by a sample average that is progressively refined, and the proximal-subgradient stepsize is selected by an Armijo-type line search enforcing a sufficient-decrease property up to stochastic errors induced by the sample-based approximation. This framework accommodates substantially more general problem formulations than existing methods, in particular, it requires neither (weak) convexity of the regularizer nor a uniform bound on the variance of the stochastic oracle, and our analysis yields convergence guarantees that are new even in the smooth setting. Specifically, we establish almost sure convergence of the sequence of function values and stationarity of every accumulation point of the trajectories under the relaxed requirement that the sample-size sequence be merely nondecreasing and unbounded, with no prescribed growth rate. Leveraging the \KLname (\KLshort) property, we further upgrade this subsequential guarantee to convergence of the whole trajectory to a single stationary point. Finally, for exponential-type \KLshort~desingularizing functions and polynomially growing sample sizes, we derive explicit polynomial convergence rates, up to a logarithmic factor,  for both the function values and the iterates.

\end{abstract}\vspace*{-0.05in}

 \vspace*{-0.05in}

\vspace*{-0.2in}

\section{Introduction}

A large class of optimization problems under uncertainty are modeled as minimizing an expected value written as an integral with respect to the underlying probability distribution. In many
applications arising in machine learning, statistics, and operations research, this expectation cannot be evaluated exactly, either because the
distribution is not available in closed form or because the resulting integral is
computationally intractable in high dimensions. To overcome this issue,  stochastic algorithms replace exact first-order information by sample-based
approximations, leading to methods that must balance computational effort,
statistical accuracy, and convergence guarantees; see, for instance,
\cite{NemirovskiJuditskyLanShapiro2009,BottouCurtisNocedal2018,
GhadimiLan2013,MR3459195,MR3439803}.

In this paper, we consider the following structured nonconvex stochastic
programming problem
\begin{equation}\label{ProblemP}\tag{$\mathcal{P}$}
    \min_{x\in \mathbb{R}^d}
    \vPhi(x):=\mathbb{E}_{\mu}\big[\varphi(\xi,x)\big]+\psi(x),
\end{equation}
where the expected cost is defined on a probability space
$(\Xi,\mathcal{A},\mu)$ by
\[
    \mathbb{E}_{\mu}\big[\varphi(\xi,x)\big]
    :=\int_{\Xi}\varphi(\xi,x)\,d\mu(\xi).
\]
Here, $\varphi:\Xi\times\mathbb{R}^d\to\mathbb{R}$ is assumed to be an
upper-$C^2$ normal integrand with respect to the decision variable in the sense
of Definition~\ref{def_UpperC2NormalIntegrand}. Roughly speaking, for all
$\xi\in\Xi$, the mapping $x\mapsto \varphi(\xi,x)$ may be nonsmooth and
nonconvex, satisfying a local generalized version of the so-called \emph{descent lemma}, 
suitable for
variational analysis. The term
$\psi:\mathbb{R}^d\to\Rex$ is assumed to be proper, lower
semicontinuous, and prox-bounded. This allows us to encode constraints,
nonsmooth regularization terms, and convex/nonconvex penalties. Models of the form
\eqref{ProblemP} emerge naturally in stochastic problems where nonsmooth losses and regularizers appear, such as optimal quantization, sparse and partial-label learning,
matrix factorization and statistical estimation, among others; see, e.g.,
\cite{MairalBachPonceSapiro2010,FanLi2001,Zhang2010MCP,LohWainwright2015,
MetelTakeda2021,JMLR:v12:cour11a,Cheng_Wang_Feng_Zhang_An_2023}.

 To tackle problem \eqref{ProblemP}, we propose a method of proximal subgradient type, in which stochastic first-order information is estimated by average approximations. A sufficient decrease condition of the sample-based estimations tracks the progress of the method, though it does not yield a descent method in the usual sense, but rather a method of approximate descent up to vanishing stochastic errors. The method is presented in Algorithm~\ref{alg:1}. 

 \paragraph{\textbf{Related work}} Stochastic optimization methods have been extensively studied in recent years
 (see, e.g., \cite{BottouCurtisNocedal2018,MR4436019,vanAckooij_Oliveira2025methods,bach2024learning}), motivated by their relevance in modern applied mathematics, and especially by
their prominent role in machine learning and related scientific fields. A substantial part of this literature has focused on convex and
strongly convex optimization, where significant advances have been achieved
through several complementary approaches, including refined stepsize rules
(see, e.g.,~\cite{NemirovskiJuditskyLanShapiro2009}), variance reduction techniques (see, e.g., \cite{MR3935081}), and central limit theorems
describing the asymptotic behavior of stochastic algorithms (see, e.g., \cite{MR4902793} and references therein).

Despite the tremendous progress made in convex stochastic programming, the
nonconvex and nonsmooth setting has received comparatively less attention. This
is due   
to the intrinsic difficulties
arising from the interaction of nonconvexity, nonsmoothness, and stochasticity.
Indeed, deterministic techniques do not transfer directly to stochastic
frameworks, where  
first-order information is typically
available through oracles only. Moreover, in nonconvex optimization, convergence
of function values does not necessarily characterize convergence to the optimal
value, and stationary points need not be global, or even local, minimizers.
Nevertheless, several important contributions have been made in this direction. For smooth and composite nonconvex stochastic programming, complexity guarantees
and accelerated variants have been obtained in
\cite{GhadimiLan2013,MR3459195,MR3439803}. More recent developments include
variance-reduced, recursive-gradient, mini-batch, and hybrid stochastic methods
for finite-sum and expected value composite problems; see, e.g.,
\cite{PhamNguyenPhanTranDinh2020,TranDinhPhamPhanNguyen2022,Li2022}. These
methods typically assume that the expected value term is smooth, or that the
nonsmooth component has a convex structure that can be handled by a proximal
operator. For nonsmooth expected value objectives, proximal stochastic
subgradient and stochastic prox-linear methods have been proposed for weakly
convex functions; see, e.g.,
\cite{DuchiRuan2018,MR3902455,MR3982682}. Stochastic
difference-of-convex algorithms have also been investigated in
\cite{LeThiHuynhPhamDinhLuu2022}.

\paragraph{\textbf{Sample-based stochastic optimization methods}} A classical approach to tackle the minimization of expected costs is the sample average approximation (SAA) method, which
replaces the expected value
function by an empirical average over a finite sample. This leads to a deterministic  optimization problem  solved by deterministic
optimization techniques. SAA methods have been extensively studied from the
viewpoints of consistency, asymptotic convergence, stability of optimal
solutions, and finite-sample error estimates; see, among others,
\cite{KingRockafellar1993,Shapiro1993,KleywegtShapiroHomemDeMello2002,
HomemDeMello2008}. This approach is especially useful when the sampled problem
preserves exploitable deterministic structure. However, solving a sequence of
large deterministic problems can become expensive when high accuracy requires
large sample sizes or when each sampled problem is itself nonsmooth and nonconvex. A different paradigm is provided by stochastic approximation (SA) methods. This line
of research goes back to the seminal works 
\cite{RobbinsMonro1951, KieferWolfowitz1952}, 
and has become one of the central algorithmic
frameworks for stochastic optimization. Unlike SAA, these methods are able to address the original stochastic problem by calling   stochastic oracles for approximating first-order information of the expected cost; see, e.g.,~\cite{MR1167814,NemirovskiJuditskyLanShapiro2009,MR2921104}. 

 Regarding  stepsize selection, we can mention two approaches. In the absence of dynamically increasing sample sizes,  stochastic approximation schemes require 
rapidly diminishing stepsizes to ensure convergence; see, e.g.,~\cite{MR3439803,NemirovskiJuditskyLanShapiro2009,geiersbach2021stochastic}. More precisely, in these works, the sequence of stepsizes $(\gamma_k)_{k\in\N}$ satisfies $\sum_{k=0}^{\infty}\gamma_k=+\infty$ and $\sum_{k=0}^{\infty}\gamma_k^2<+\infty$. Such conditions are standard in the convergence analysis of (proximal) stochastic gradient methods, both in convex and nonconvex settings. As an alternative, the stepsizes can be defined to adapt to the progress of the iteration process via backtracking line search procedures using function and (sub)gradient estimates; see, e.g., \cite{paquette2020stochastic,nguyen2025stochastic}.

\paragraph{\textbf{\KLname condition in nonconvex optimization}} The 
analysis of optimization algorithms dealing with nonconvex problems such as~\eqref{ProblemP} is significantly different from the convex case. In the nonconvex scenario,  the existence of multiple local minima and saddle points complicates establishing global convergence guarantees for the sequence of iterates. A fundamental tool to overcome this challenge of deterministic nonconvex methods is the
\KLname (\KLshort) inequality; 
see, e.g.,
\cite{absil2005convergence,AttouchBolte2009,AttouchBolteRedontSoubeyran2010,
AttouchBolteSvaiter2013,BolteSabachTeboulle2014}. The \KLshort~framework enables improving subsequential convergence results to convergence of the entire sequence of iterates to a single stationary point, as well as characterizing local convergence rates. 
Recall that the satisfaction of a sufficient decrease condition is pivotal  for the  convergence
principle of the seminal work \cite{AttouchBolteSvaiter2013}. This condition might be compromised in the stochastic setting, due to the presence of errors induced by stochastic approximations, preventing the straightforward extrapolation of \KLshort-based techniques for analyzing stochastic methods.

\paragraph{\textbf{Our contribution}} The purpose of this work is to 
analyze a new proximal stochastic
subgradient (PSS)  algorithm for solving~\eqref{ProblemP}. At every iteration, PSS constructs a surrogate model of~\eqref{ProblemP} where the expected cost is replaced by a sample average approximation refined along iterations.  Specifically, given an increasing sequence of sample sizes $(n_k)_{k\in\N}$, where $k\in\N$  represents the algorithm's current iteration, PSS draws a collection $\xi_{n_{k-1}+1},\xi_{n_{k-1}+2},\ldots,\xi_{n_k}$ of independent and identically distributed (i.i.d.) random variables to construct an \emph{empirical average}  approximation of the integral in \eqref{ProblemP}, that is, for all $x \in \R^d$, 
\begin{equation*}
\mathbb{E}_{\mu}\big[\varphi(\xi,x)\big] \approx \frac{1}{n_k}\sum_{j=1}^{n_k}  \varphi(\xi_j,x). 
\end{equation*}	
In this manner, PSS generates an iterative sequence of the form
\[
\left\{
\begin{aligned}
v^k &\in\frac{1}{n_k}\sum_{j=1}^{n_k}  \partial\varphi(\xi_j,x), \\ 
x^{k+1} & \in\prox_{\gamma\psi} (x^k - \gamma v^k),
\end{aligned}
\right.
\]thus combining stochastic (Clarke) subgradient information of $\varphi$ taken from the empirical approximation with a proximal evaluation of $\psi$. The stepsize $\gamma>0$ for this proximal subgradient update is computed by an Armijo-type  line search~\cite{MR191071} to ensure a sufficient decrease 
of the sample model of the objective function $\vPhi$ in \eqref{ProblemP}. For more details, we refer to Algorithm~\ref{alg:1}.  
The sample errors lead to a nonmonotone stochastic scheme with respect to the original objective $\vPhi$, hindering the direct  application of standard techniques for the convergence analysis of descent methods, and particularly the \KLshort~framework. 
Nevertheless, we show that our method can be viewed as a descent scheme perturbed by stochastic errors that vanish along the
iteration process, as long as the sequence of iterates remains bounded. In this regard, and in contrast to standard approaches, we do not impose any bounded variance condition, but rather assume local Lipschitzianity of $x\mapsto \varphi(\xi,\cdot)$; see Remark~\ref{remark:bv}. This allows us to combine a stopping time based analysis together with new results on uniform error bounds for average approximations on compact sets; see Section~\ref{sect:aae}.

 From the perspective of the proximal operation, our analysis only relies on the ability to compute at least one element of the proximal mapping of $\psi$, without requiring an additional geometric structure such as (weak) convexity, different from most existing works on stochastic approximation schemes with proximal steps; see, e.g., \cite{MR3982682,zhang2022stochastic,pougkakiotis2023zeroth,fatkhullin2025stochastic,jia2025first}. Thus, our setting naturally covers the class of proper, lower semicontinuous, prox-bounded functions for which such a proximal point is available. From a numerical standpoint, 
this allows us to cover many relevant classes of constraints and nonconvex regularizers; see, e.g.,~\cite{FanLi2001,Zhang2010MCP,LohWainwright2015,MetelTakeda2021, bohm2021variable, atenas2025understanding}.

Our main results can be summarized as follows:
\begin{itemize}
\item[$\bullet$]  In Theorem~\ref{The01}, we prove that almost surely every accumulation point of bounded sequences generated by PSS is a stationary point of~\eqref{ProblemP}.   
We do not impose any assumption on the sequence of sample sizes $(n_k)_{k\in\N}$ besides being nondecreasing and unbounded. To the best of our knowledge, this feature remains novel even within the smooth  setting; see Remark~\ref{r:The01} for details.
\item[$\bullet$]  In Theorem~\ref{Main_PKL_conv}, we upgrade the previous subsequential guarantee to global convergence of the whole sequence of iterates generated by PSS to a single stationary point of~\eqref{ProblemP}. This trajectory-level result is obtained for  non-differentiable objectives by leveraging the \KLshort~condition (see~\eqref{PLK_cond}), provided the approximation error is controlled at a summable rate.

\item[$\bullet$]  In Theorem~\ref{th:rates}, we establish explicit local convergence rates when the desingularizing function of the \KLshort~condition belongs to the exponential family (see Remark~\ref{remark_kindoftheta}) and the sample sizes grow at least polynomially. We show that both the function values and the iterates converge at a rate of the order $\mathcal{O}\bigl(k^{-p}\ln(k)\bigr)$, where the exponent $p>0$ is given in closed form as a function of the \KLshort~exponent $\beta$ and of the polynomial decay rate $\gamma$ of the inverse sample-size sequence $(n_k^{-1})_{k\in\mathbb{N}}$; see Figure~\ref{fig:rates_fill} for a visualization of the attainable exponents. The underlying recursion estimate (Lemma~\ref{lemma:sequences}), which tracks a contraction perturbed by polynomially decaying errors, is of independent interest and, to our knowledge, has not previously appeared in this form.

\end{itemize}

In particular, our results
show that deterministic \KLshort-type ideas can still be used in a stochastic
nonmonotone setting, provided that the sample errors are controlled with
sufficient accuracy alongside iterations.

\paragraph{\textbf{Organization of the paper}}  In Section~\ref{s:prelim}, we introduce the notation and concepts used throughout this work and show some preliminary, technical results. In particular, we introduce the notion of upper-$C^2$ normal integrands (Definition~\ref{def_UpperC2NormalIntegrand}), pivotal for our analysis. Section~\ref{sec:3} presents our proposed  PSS method (Algorithm~\ref{alg:1}) to solve problem \eqref{ProblemP} and its analysis of almost surely subsequential convergence. In Section~\ref{s:PLK}, under the assumptions of the \KLshort~condition, we  deduce global convergence results and estimates of the rate of convergence of the PSS method. In Section~\ref{s:numerics} we illustrate the proposed method in numerical experiments on two tasks: optimal quantization and partial-label linear regression. Section~\ref{s:conclusion} concludes with some final remarks and future research directions.

\section{Preliminaries of variational and stochastic analysis} \label{s:prelim}

\subsection{Elements from 
nonsmooth analysis}
For a function $f: \R^d \to \Rex$, Lipschitz continuous around $x \in \R^d$, the  \emph{Clarke subdifferential} is defined as
\begin{align*}
	 {\partial} f(x) = \{v\in\R^d :   \langle v , d\rangle \leq  f^\circ (x;d) \textrm{ for all }  d\in\R^d\},
\end{align*}
where  $	f^\circ(x;d)$ stands for the  Clarke subderivative at $x$ in the direction $d\in\R^d$, which  is defined as
 \begin{align*}
	 f^\circ (x;d) = \limsup_{\substack{y\to x\\ t\to 0^+}}\left(  \frac{f(y+td)-f(y)}{t} \right).
\end{align*}A function $f : \mathbb{R}^d \to \Rex$ is said to be \emph{prox-bounded} if there exists some $\gamma > 0$ such that $x\mapsto f(x) + \frac{1}{2\gamma} \|x  \|^2$ is bounded from below. The supremum of the set of all such   $\gamma>0$ is called the \emph{threshold of prox-boundedness} for  $f$, and is denoted by $\gamma^f$. Given a proper, lower semicontinuous (lsc) function $f : \mathbb{R}^d \to \Rex$ and a constant $\gamma \in{( 0, +\infty]}$, the \emph{proximal mapping} is the multifunction $\prox_{\gamma f} : \mathbb{R}^d \rightrightarrows \mathbb{R}^d$ defined as the solution set of the optimization problem
\[
\prox_{\gamma f}(x) := \argmin_{u \in \mathbb{R}^d} \left\{ f(u) + \frac{1}{2\gamma} \|u - x\|^2 \right\},
\]
with the convention $\frac{1}{2\gamma} =0$ for $\gamma =+\infty$. If $f$ is prox-bounded, then  the set $\prox_{\gamma f}(x)$ is nonempty for every  $\gamma\in{(0, \gamma^f)}$. We say that $\bar x$ is a stationary point of \eqref{ProblemP} if there exists $\bar v \in \partial_x  \mathbb{E}_{\mu}\big[\varphi(\xi,\bar x)\big]$ and $\bar \gamma >0$ such that 
\begin{align}\label{def_stpoint}
	\bar x \in \prox_{\bar \gamma \psi }(\bar x - \bar \gamma \bar v). 
\end{align}

The following lemma 
shows a type of upper-semicontinuity property of the proximal mapping for prox-bounded functions. Its proof follows the same arguments in~\cite[Lemma~2.3]{perezaros2025randomizedblockproximalmethod}, and it is included for completeness.
 
 \begin{lemma}\label{upper_cont_prox}
 	Let $\psi: \mathbb{R}^d \to \Rex$ be a proper, lsc and prox-bounded function with threshold $\gamma^{\psi}$.  Consider  a point $\bar x \in \dom \psi$,  a bounded sequence    $(v^k)_{k \in  \N }$ in $\R^d$,   and, for all $k \in \N$,  $\hat{x}^{k}   \in \prox_{\gamma_k \psi} \left(x^{k} - \gamma_k v^k \right)$ for some sequence $x^k \to \bar x$ and $\gamma_k \to 0^+$ with $\gamma_k\in{(0,\gamma^{\psi})}$. Then $\hat{x}^k \to \bar x$.
 \end{lemma}

\begin{proof}
    Let $\gamma\in{(0,\gamma^\psi)}$. Since $\psi$ is prox-bounded and $(x^k)_{k\in\N}$ is bounded, using Young's inequality 
    we have that  there exists  $\alpha \in \R$ such that
 	\begin{align*}
 		\psi(x) \geq -\frac{1}{2\gamma} \| x-    x^k \|^2  + \alpha, \quad \text{ for all } x \in \mathbb{R}^d.
 	\end{align*}
 	Now, using the definition of the proximal mapping we get that 
    \begin{equation*}
 	\begin{aligned}
 		\psi(\hat x^k)   +  &  \frac{1}{2\gamma_k}\|\hat x^k -x^k\|^2   +  \langle  v^k,  \hat x^k -x^k\rangle  + \frac{\gamma_k}{2}\| v^k \|^2   \leq  \psi( \bar x ) + \frac{1}{2\gamma_k} \|\bar{x}- x^k\|^2 + \langle v^k, \bar{x}-x^k\rangle  + \frac{\gamma_k}{2}\|v^k\|^2.
 	\end{aligned}
    \end{equation*}
 	Observe that $\langle  v^k, \hat x^k -x^k\rangle \geq -  \frac{1}{2}\| \hat x^k -x^k\|^2  -  \frac{1}{2}\| v^k \|^2$. Hence, by using the two above inequalities one gets
    \begin{equation*}
 	\begin{aligned}
 		\biggl( \frac{1}{2\gamma_k}- & \frac{1}{2\gamma}- \frac{1}{2}\biggr)  \|\hat x^k -x^k \|^2  \leq 	\psi(\bar x )+\frac{1}{2\gamma_k}\|\bar{x}-x^k\|^2 + \langle v^k,\bar{x}-x^k\rangle- \alpha  +  \frac{1}{2}\| v^k \|^2.
 	\end{aligned}
    \end{equation*}
 	It suffices to multiply the above by $\gamma_k$ and pass to the limit as $k\to\infty$  to conclude that  $ \|\hat x^k -x^k\| \to 0$, and consequently $\hat x^k  \to \bar x$.
\end{proof}

\subsection{Elements from measure theory}

In what follows,  $(\Xi,  \mathcal{A}, \mu)$ will be a measure space, and we  use the notation $L^p_\mu$ for denoting the space of all $p$-integrable functions. Moreover, we  consider a (sample) probability space $(\Omega, \mathcal{F}, \mathbb{P})$. 

Consider a $\sigma$-algebra $\mathcal{G} \subseteq \mathcal{F}$ and an integrable function $f: \Omega \to \mathbb{R}$. We denote by $\mathbb{E}[f \mid \mathcal{G}]$ the conditional expectation of $f$ with respect to $\mathcal{G}$. Given a measurable set $F \in \mathcal{F}$, the indicator of the set $F$ is given by:
\[
\1_{F}(\omega) = 
\begin{cases} 
	1, & \text{if } \omega \in F, \\
	0, & \text{if } \omega \notin F.
\end{cases}
\] A set-valued mapping $M \colon \Xi \tto \mathbb{R}^d$ is said to be {measurable} if $M^{-1}(U) \in \mathcal{A}$ for every open set $U \subset \mathbb{R}^d$, where  
$M^{-1}(U) := \{ t \in \Xi \mid M(t) \cap U \neq \emptyset \}$.
Additionally, a function $\varphi \colon \Xi \times \mathbb{R}^d \to \Rex$ is called a \emph{normal integrand} if the multifunction $\xi \in \Xi \mapsto \operatorname{epi} \varphi_\xi$ is measurable with closed values, where the function $\varphi_\xi:\R^d\to\Rex$ is given by\begin{equation}\label{eq:notation-sub-xi}
    \varphi_\xi(x) := \varphi(\xi,x), \qquad \text{for all } x\in\R^d.
\end{equation} Let $(\Omega,\mathcal F,\mathbb P)$ be a probability space and let
$(\mathcal F_n)_{n\in\N}$ be a filtration, namely, a family of sub-$\sigma$-algebras such that
$\mathcal F_0 \subseteq \mathcal F_1 \subseteq \mathcal F_2 \subseteq \cdots \subseteq \mathcal F $. A random variable $\tau:\Omega\to \mathbb N\cup\{+\infty\}$ is called a
\emph{stopping time} with respect to $(\mathcal F_n)_{n\in\N}$ if for every
$n\in\N$,
$\{\tau \le n\}\in \mathcal F_n$. Equivalently, $\tau$ is a stopping time if and only if $\{\tau = n\}\in \mathcal F_n, \text{~for all } n\geq 0$.

Following standard notation in probability theory, we omit the dependence on 
$\omega\in\Omega$ when writing random variables defined on the probability space 
$(\Omega,\mathcal F,\mathbb P)$, in order to avoid overloading the notation. 
Whenever needed, this dependence will be made explicit to prevent ambiguity.

The following result, known as the \emph{freezing lemma}, provides a formula for computing conditional expectations of compositions involving integrands and independent random variables, extending the classical conditional-expectation formula under independence; see, e.g., \cite[Lemma~10.10.2]{bogachev}. The proof follows from standard approximation techniques based on simple functions and can be found in \cite[Lemma~4.11]{MR4703976}.
\begin{lemma}\label{lemma_measurability}
	Let $\varphi: \Xi \times \mathbb{R}^d \to \R$ be a measurable mapping. Let $\mathcal{G}\subseteq \mathcal{F}$ be a $\sigma$-algebra on $\Omega$, and let $ \xi: \Omega \to \Xi$ be independent of $\mathcal{G}$ and $y: \Omega \to \mathbb{R}^d$ a $\mathcal{G}$-measurable function. Then, for every $\omega\in\Omega$,\begin{equation*}
		\mathbb{E}\left[ \varphi( \xi( \cdot ), y(\cdot)  ) \, | \,\mathcal{G}  \right](\omega) = \int_{\Omega } \varphi( \xi(\omega_1 ) , y(\omega))\, d\mathbb{P}(\omega_1), \quad \text{a.s.},
	\end{equation*}
	provided that  $\omega \mapsto \varphi(\xi(\omega), y(\omega))$ is integrable.   
\end{lemma}

Next, let us establish a  notion of local Lipschitz continuity for normal integrands.
    \begin{definition}[locally Lipschitz normal integrand]\label{defin_locally_Lips} A normal integrand   $\varphi: \Xi \times\mathbb{R}^d \to \R^m$ is said to be   locally Lipschitz around $\bar x$ with square integrable modulus  if $\varphi(\cdot,\bar x)  \in L^2_{\mu}$ and there exist  $\kappa \in L^2_{\mu}$ and  a neighborhood $V$ of $\bar x$ such that\begin{align*}
 	\| \varphi_\xi( y) - \varphi_\xi(z)\| \leq \kappa(\xi) \| y- z\|, \text{ for all }y,z \in V, \, \xi\in\Xi.
 	\end{align*}
    \end{definition}

 The following   proposition   shows that the local Lipschitz continuity stated above is equivalent to what appears to be a stronger condition: a uniform Lipschitz condition over bounded sets. More precisely, we have the following result, whose proof follows from a classical compactness argument.
 	
 \begin{proposition}\label{Prop_eq_Lipsc_cont_Unif}
 	Let us consider an integrand $\varphi: \Xi \times  \mathbb{R}^d \to \R$ and let $U  \subseteq \R^d$ be a convex 
    compact set. Then the following are equivalent:
 	\begin{enumerate}[label=(\alph*)]
 		\item for every $\bar x \in U$, the integrand    $\varphi$  is locally Lipschitz around $\bar x$ with square integrable modulus.
 		\item there exist $\kappa \in L^2_{\mu}$ and a neighborhood $V$ of $U$ such that 
 			\begin{align}\label{eq_Lipsc_cont_Unif}
 		\varphi(\cdot, y )  \in L^2_{\mu} \text{ and }	| \varphi_\xi (y) - \varphi_\xi( z)| \leq \kappa(\xi) \| y- z\|, \text{ for all }y,z \in V {, \, \xi\in\Xi}.
 		\end{align}
 	\end{enumerate}
 	\end{proposition}

%
%
 \subsection{Upper-$C^2$ normal integrands}

 We now introduce the definition of  upper-$C^2$ integrands 
 anticipated in~\eqref{ProblemP}, and that will constitute a fundamental structural assumption for this problem throughout the manuscript.

\begin{definition}[upper-$C^2$ normal integrand]\label{def_UpperC2NormalIntegrand}
	Let $\varphi :\Xi \times \R^d \to \Rex$  be a normal integrand. We say that $\varphi$ is upper-$C^2$ around $\bar x$ provided that   $\varphi$  is locally Lipschitz around $\bar x$ with square integrable modulus  and there exist a neighborhood $V$ of $\bar x$ and (a nonnegative) $\rA \in L^1_{\mu}$ such that  for all $\xi \in \Xi$,  $x\in V$ and   $v \in \partial \varphi_\xi(x)$, it holds
	\begin{align}\label{integral_equpperdesc}
		\varphi_\xi(y) \leq \varphi_\xi( x) + \langle v , y-x\rangle + \rA(\xi) \|y -x\|^2,\quad\text{for all } y \in V.
	\end{align}
 Furthermore, we simply say that $\varphi$ is an upper-$C^2$  normal integrand provided that $\varphi$  is upper-$C^2$ around every point $\bar x \in \R^d$.
	\end{definition}

\begin{remark}
    Definition~\ref{def_UpperC2NormalIntegrand} extends the notion of upper-$C^2$ functions for integrands. The classical definition of upper-$C^2$ functions, namely, 
    the minimum over a family of twice-continuously differentiable functions (see, e.g., \cite{MR1491362}) 
    is equivalent, as established in \cite[Proposition 2.2]{aragonartacho2023boosted} (see also \cite[Theorem 5.1]{MR1363364}), to \eqref{integral_equpperdesc} when $\Xi$ is a singleton. This equivalence (and further characterizations) can be extended to general sets $\Xi$,  this is the object of the working paper ``A note on upper and lower-$C^2$ normal integrands''. Nonetheless, we adopt this alternative definition because it highlights that this class of functions satisfies a nonsmooth version of the descent lemma, a property well-recognized in smooth optimization (see, e.g., \cite[Lemma A.11]{MR3289054}). Notably, it is well-known that  smooth functions with Lipschitz continuous gradients satisfy an inequality of the kind of \eqref{integral_equpperdesc}. Furthermore, it has been shown that the class of upper-$C^2$ functions is broader, encompassing structured functions formed as the difference between a smooth function with a Lipschitz continuous gradient and a (potentially nonsmooth) convex function (see, e.g., \cite{aragonartacho2023boosted,aragonartacho2025nonmonotonesubgradientmethodsbased,kanzow2026nonmonotonedescentmethodoptimization} for further details). 
\end{remark}

Similarly to Proposition \ref{Prop_eq_Lipsc_cont_Unif}, using compactness arguments, we can demonstrate that the integrable function $\rA$ in \eqref{integral_equpperdesc} can be chosen uniformly on bounded sets.

\begin{proposition}\label{upperC2unif}
	Let $\varphi$ be an upper-$C^2$ normal integrand. Then for every  convex compact set $B$ there exists $\rA \in L^1_\mu$ such that for all $x\in B$ there exists a neighborhood $V$ of $x$ such that \eqref{integral_equpperdesc} holds. \end{proposition}

\subsection{Average approximation errors}\label{sect:aae}

We now introduce the notation for the expected cost in problem~\eqref{ProblemP} and its sample-based approximations that we will use throughout this paper. 
For ease of notation, we let $\Intf{\varphi} : \mathbb{R}^d \to \mathbb{R}$ be the expected value function  given by
\begin{align*}
\Intf{\varphi}(x)
:= \mathbb{E}_{\mu}\left[ \varphi(\xi,x) \right]
= \int_{\Xi} \varphi(\xi,x)\, d\mu(\xi), \mbox{~for all~} x \in \R^d.
\end{align*} Since the exact evaluation of $\Intf{\varphi}$ may be unavailable or computationally expensive, we approximate it by empirical averages constructed from samples. More precisely, for $k \in \N$ representing an iteration index,  the \emph{empirical average} of $\Intf{\varphi}$ at a point $x\in\mathbb{R}^d$ is given by
\begin{equation}\label{defEk}
\Intf{\varphi}^{k}(x) :=	\frac{1}{n_k}\sum_{j=1}^{n_k}  \varphi(\xi_j,x), \text{ for all } x\in\mathbb{R}^d,
\end{equation}and where $n_k\in\N$ is the sample size, and for $j=1,\dots,n_k$, the random variables $\xi_j$ are i.i.d. 
In this manner, the sample \emph{average approximation} of $\vPhi$ in \eqref{ProblemP} at iteration $k$ is given by
\begin{equation}\label{defPhik}
		\vPhi_{k} (x):=   \Intf{\varphi}^k(x) + \psi(x), \text{ for all } x\in\mathbb{R}^d.
\end{equation}	

In the context of the approximation in \eqref{defEk}-\eqref{defPhik}, we make the following standing assumption on the sample $(\xi_i)_{i\geq1}$ used throughout the paper.

\begin{assumption}[Sampling framework]\label{Assumption00}
Let $(\Omega,\mathcal F,\mathbb P)$ be a probability space, and let
$\xi_i:\Omega\to\Xi$, $i\geq 1$, be a sequence of i.i.d. random variables with common law $\mu$, and a divergent and nondecreasing  sequence of sample
sizes $(n_k)_{k\in\N}\subseteq\mathbb N\backslash\{0\}$. 
\end{assumption}

For notational convenience, we omit the symbol $\omega  \in \Omega$ for the functions~$\xi_i$. In particular,  using the notation already  introduced in \eqref{eq:notation-sub-xi} for a normal integrand $\varphi:\Xi\times \R^d\to \R$, in what follows we shall write
\begin{align*}
\varphi_{\xi_i}(x) :=  \varphi(\xi_i(\omega), x).
\end{align*}

To establish convergence of the PSS method, we need to control the error associated with replacing the expected value with the empirical average in \eqref{defPhik} during the iteration process. The following lemma furnishes a nonasymptotic bound for the uniform error of the average approximation generated by the sampling.

\begin{lemma}\label{lemma_logn}
Let $(\xi_i)_{i\geq 1}$ be a sampling with sample size $(n_k)_{k \in \N} $ satisfying Assumption~\ref{Assumption00}, and let $F: \Xi \times V \to \R^m$ be a measurable mapping with $V \subset \R^d$ a closed and bounded set. Suppose that there exist $\kappa \in L^2_{\mu}$ and  $\bar x \in V$ such that
 		\begin{align}\label{eq_Lipsc_cont_2}
 	 F(\cdot,\bar x)  \in L^2_{\mu} \text{ and }	\| F(\xi,  y) - F(\xi,  z)\| \leq \kappa(\xi) \| y- z\|, \text{ for all }y,z \in V , \, \xi\in\Xi.
 	\end{align} Then, the following hold.
    \begin{enumerate}[label=(\roman*), ref=(\roman*)] 
        \item \label{lemma_logn-i} There exists $a>0$ such that 
    \begin{align}\label{ineq_01lemma_logn}
	\mathbb{E}\left(  \sup_{x \in V }  \left\|	 \int_{ \Xi } F(\xi ,x ) d \mu(\xi)   -   \frac{1}{k} \sum_{i=1}^k F( \xi_i,x)   \right\| \right) \leq \frac{a }{\sqrt{k}}, \quad  \text{for  all  $k \geq 1$}.
    \end{align}
    \item \label{lemma_logn-ii}  In particular, let $\varphi: \Xi \times \R^d \to \R$ be a normal integrand  which is  locally Lipschitz around any point $  x \in \R^d$ with square integrable modulus. Then, for all $\rr >0$, there exists a constant $a_\rr$ such that  
\begin{align}\label{ineq_02lemma_logn}
	\mathbb{E}\left(  \sup_{\|x\| \leq \rr }  \left|	\Intf{\varphi}(x) -   \Intf{\varphi}^k  (x)	  \right| \right) \leq \frac{a_{\rr} }{\sqrt{n_k}}, \quad  \text{for  all  $k\in \N$}.
\end{align}
\item \label{lemma_logn-iii} There exists a measurable set $\hat\Omega $ with $\mathbb{P}(\Omega \backslash \hat \Omega) =0$ such that for all $\omega \in \hat\Omega$ and every sequence   $(y^k)_{ k \in  \N } $ in $\mathbb{R}^d$ with $y^k \to y$ we have that 
$\Intf{\varphi}^k  (y^k) \to \Intf{\varphi}(y) $.
    \end{enumerate}

\end{lemma}
 
\begin{proof}
The first assertion follows directly from \cite[Theorem~2.1 and
Proposition~2.6]{MR4726003}, applied to the family of real-valued functions
\begin{equation}\label{def_G}
    G(\xi,x^\ast,x):=\langle x^\ast,F(\xi,x)\rangle,
    \qquad
    \xi\in \Xi,\quad (x^\ast,x)\in \mathbb{B}_1[0]\times V .
\end{equation}
Indeed, this application yields precisely estimate~\eqref{ineq_01lemma_logn}. Moreover, since $\varphi$ is Lipschitz continuous with respect to $x$ on the compact set
$\mathbb{B}_{\rr }[0]$, with a square-integrable Lipschitz modulus,
the same argument applies to $\varphi$. Therefore,
\eqref{ineq_02lemma_logn} follows from \ref{lemma_logn-i}. Finally, the pointwise convergence follows from
\cite[Theorem~9.60]{MR4362585}.

\end{proof}
The following lemma establishes an almost sure convergence rate for the
uniform gap between the empirical expectation and its exact counterpart.
Its proof relies on the technique of \emph{normalized convergence}; see
\cite{MR1099305,MR3129765} and the references therein for further details.
\begin{lemma}\label{lemmarateconver}
	Under the assumptions of Lemma~\ref{lemma_logn}, we have a.s. that
	\begin{align}\label{eq00_lemmarateconver}
	   \sup_{x \in V }  \left\|	 \int_{ \Xi } F(\xi ,x ) d \mu(\xi)   -   \frac{1}{k} \sum_{i=1}^k F( \xi_i,x)   \right\| =\mathcal{O}\left( \frac{\ln(k)}{\sqrt{k}}  \right).
	\end{align}

	\end{lemma}
   \begin{proof}
A closer inspection of the proof of \cite[Theorem~2.2]{MR4726003} (see the first inequality in \cite[p.~1283]{MR4726003}, equation (5.8) therein, and the first equation in \cite[p.~1284]{MR4726003}) applied to
the function $G$ defined in \eqref{def_G}, shows that there exist constants
$a,b>0$ and $k_0\in\mathbb{N}$ such that, for every $k\geq k_0$ and
every $\epsilon>b/\sqrt{k}$,
\begin{align}\label{eq01_lemmarateconver}
    \mathbb{P}\Bigg(
    \left\{
        \sup_{x\in V}
        \left\|
            \int_{\Xi} F(\xi,x)\,d\mu(\xi)
            -
            \frac{1}{k}\sum_{i=1}^k F(\xi_i,x)
        \right\|
        >\epsilon
    \right\}
    \cap B_k
    \Bigg)
    \leq \exp\big(-a\sqrt{k}\,\epsilon\big),
\end{align}
where $B_k:=
    \left\{
        \frac{1}{k}\sum_{j=1}^k \kappa(\xi_j)^2
        \leq
        2\mathbb{E}\big[\kappa(\xi)^2\big]
    \right\}$,
and $\kappa$ denotes the Lipschitz modulus appearing in \eqref{eq_Lipsc_cont_2}. Now, without loss of generality we may assume that $\mathbb{E}\big[\kappa(\xi)^2\big]>0$. By the strong law of large numbers (see, e.g., \cite[Theorem~8.3.5]{MR982264}), $\frac{1}{k}\sum_{j=1}^k \kappa(\xi_j)^2
    \longrightarrow \mathbb{E}\big[\kappa(\xi)^2\big]$, a.s. Hence, by Egorov's theorem \cite[Theorem 2.2.1]{bogachev}, for every $\varsigma>0$ there exists a measurable set
$\Omega_\varsigma\subseteq \Omega$ such that $\mathbb{P}(\Omega\setminus \Omega_\varsigma)\leq \varsigma$, and the convergence above is uniform on $\Omega_\varsigma$. Consequently, there exists
an integer $k_1\geq k_0$ such that $\Omega_\varsigma\subseteq B_k$ for all $k\geq k_1$. Now fix $\delta>a^{-1}$, and choose $k_2\geq k_1$ such that $\delta\ln(k)>b$ for all $k\geq k_2$. For such $k$, we may apply \eqref{eq01_lemmarateconver} with $\epsilon_k:=\delta\frac{\ln(k)}{\sqrt{k}}$. Define
\[
    A_k^\varsigma:=
    \left\{\omega \in \Omega : 
        \frac{\sqrt{k}}{\ln(k)}
        \sup_{x\in V}
        \left\|
            \int_{\Xi} F(\xi,x)\,d\mu(\xi)
            -
            \frac{1}{k}\sum_{i=1}^k F(\xi_i,x)
        \right\|
        >\delta
    \right\}
    \cap \Omega_\varsigma .
\]
Since $\Omega_\varsigma\subseteq B_k$ for every $k\geq k_2$, we have
\[
    A_k^\varsigma
    \subseteq
    \left\{
        \sup_{x\in V}
        \left\|
            \int_{\Xi} F(\xi,x)\,d\mu(\xi)
            -
            \frac{1}{k}\sum_{i=1}^k F(\xi_i,x)
        \right\|
        >\epsilon_k
    \right\}
    \cap B_k .
\]
Therefore, by \eqref{eq01_lemmarateconver},
\begin{align*}
    \mathbb{P}(A_k^\varsigma)
    \leq
    \exp\big(-a\sqrt{k}\epsilon_k\big)
    =
    \exp\big(-a\delta\ln(k)\big)
    =
    \frac{1}{k^{a\delta}},
    \qquad \text{for all } k\geq k_2 .
\end{align*}
Since $a\delta>1$, the  Borel--Cantelli lemma \cite[Exercise 1.12.89 (i)]{bogachev} yields
\[
    \sum_{k=1}^{\infty}\mathbb{P}(A_k^\varsigma)<+\infty  \; \Longrightarrow \;
\mathbb{P}\left(\limsup_{k\to\infty} A_k^\varsigma\right)=0.
\]
Consequently, for almost every $\omega\in\Omega_\varsigma$, there exists
$\hat{k}(\omega)\in\mathbb{N}$ such that, for every $k\geq \hat{k}(\omega)$,
\[
    \sup_{x\in V}
    \left\|
        \int_{\Xi} F(\xi,x)\,d\mu(\xi)
        -
        \frac{1}{k}\sum_{i=1}^k F(\xi_i,x)
    \right\|
    \leq
    \delta\frac{\ln(k)}{\sqrt{k}}.
\]
Since $\varsigma>0$ was arbitrary, we conclude that
\[
    \sup_{x\in V}
    \left\|
        \int_{\Xi} F(\xi,x)\,d\mu(\xi)
        -
        \frac{1}{k}\sum_{i=1}^k F(\xi_i,x)
    \right\|
    =
    \mathcal{O}\left(\frac{\ln(k)}{\sqrt{k}}\right),
    \qquad \text{a.s.}
\]
This proves \eqref{eq00_lemmarateconver}.
\end{proof}

We  establish in the next lemma  the  convergence of the Clarke sample-based subgradients. Formally, we have the following result.

\begin{lemma}\label{LemmaClarke}
Let $(\xi_i)_{i\geq1}$ be a sampling with sample size $(n_k)_{k \in \N} $ satisfying Assumption~\ref{Assumption00}, and let $\varphi: \Xi \times \mathbb{R}^d \to \Rex$ be an upper-$C^2$ normal integrand. Then there exists a measurable set $\hat{\Omega} \subset \Omega$ with $\mathbb{P}(\Omega \setminus \hat{\Omega}) = 0$ such that,  for all $\omega \in \hat{\Omega}$, every sequence $x^k \to x$ and every $v^k \in \partial \Intf{\varphi}^k(x^k)$, all the accumulation points of $(v^k)_{k\in\N}$ belong to $\partial \Intf{\varphi}(x)$.
\end{lemma}
\begin{proof}
For all  $x, h\in\R^d$, $k \in \N$, define     $  
    f(x,h):=\big(\Intf{\varphi}\big)^\circ(x;h)$
     and $f^k(x,h):=\big(\Intf{\varphi}^k\big)^\circ(x;h)$. By the Clarke interchange formula for subdifferentiation under the integral \cite[Theorem 2.7.2]{MR1058436},
together with the Clarke regularity of $-\varphi_\xi$ for every $x\in\R^d$, we
have
\begin{equation*}
    \partial \Intf{\varphi}(x)
    =
    \int_{\Xi}\partial\varphi_\xi(x)\,d\mu(\xi),
    \qquad \qquad
    f(x,h) 
    =
    \int_{\Xi}\varphi_\xi^\circ(x;h)\,d\mu(\xi),
\end{equation*}
for all $h\in\R^d$. Similarly, by the sum rule for the Clarke
subdifferential and the Clarke regularity of each function $\varphi_\xi$ \cite[Proposition 2.3.3 and Corollary 3]{MR1058436}, we
obtain 
\begin{equation}\label{eqsubd02}
    \partial \Intf{\varphi}^k(x)
    =
    \frac{1}{n_k}\sum_{j=1}^{n_k}\partial\varphi_{\xi_j}(x),
    \qquad \qquad
    f^k(x,h)
    =
    \frac{1}{n_k}\sum_{j=1}^{n_k}\varphi_{\xi_j}^\circ(x;h),
\end{equation}
for all $x,h\in\R^d$ and all $k\in\mathbb N$. 
Therefore, by \cite[Theorem 2.3]{MR1363357}, there exists a measurable set
$\hat\Omega\subset\Omega$, with $\mathbb P(\Omega\setminus\hat\Omega)=0$, such
that, for every $\omega\in\hat\Omega$, the sequence $(f^k)_{k\in\mathbb N}$
hypographically converges to $f$ on $\R^d\times\R^d$.

Next, fix $\omega\in\hat\Omega$. Let $x^{\nu_k}\to x$ and $v^{\nu_k}\to v$, with
    $v^{\nu_k}\in\partial \Intf{\varphi}^{\nu_k}(x^{\nu_k})$ for all $k\in\mathbb N$.
We prove that $v\in\partial \Intf{\varphi}(x)$. Indeed, for every $h\in\R^d$,
the definition of the Clarke subdifferential gives
\[
    \langle v^{\nu_k},h\rangle
    \leq
    \big(\Intf{\varphi}^{\nu_k}\big)^\circ(x^{\nu_k};h)
    =
    f^{\nu_k}(x^{\nu_k},h),
    \qquad \text{for all } k\in\mathbb N.
\]
Since $(x^{\nu_k},h)\to(x,h)$ and $f^k$ hypographically converges to $f$, \cite[Eq. 7(9)]{MR1491362} implies
\[
    \langle v,h\rangle
    =
    \lim_{k\to\infty}\langle v^{\nu_k},h\rangle
    \leq
    \limsup_{k\to\infty} f^{\nu_k}(x^{\nu_k},h)
    \leq
    f(x,h)
    =
    \big(\Intf{\varphi}\big)^\circ(x;h).
\]
Since $h\in\R^d$ was arbitrary, it follows from the characterization of the
Clarke subdifferential that
   $ v\in\partial \Intf{\varphi}(x)$, 
which proves the desired inclusion.
\end{proof}

\subsection{Technical lemmas on real sequences}

To conclude this section, we state two technical lemmas involving sequences that appear in the subsequent convergence analysis. We defer the proof of these results to Appendix~\ref{s:appendix}. 

The first technical lemma shows the convergence of a series involving the sequence of sample sizes used to approximate the expected value in \eqref{ProblemP}, which is instrumental for proving subsequential convergence in Theorem~\ref{The01}. The proof of the following result can be found in Appendix~\ref{proof-l:bound_sample01}.
\begin{lemma}\label{l:bound_sample01} Let $(n_k)_{k \in \N}$ be a divergent, nondecreasing sequence of natural numbers. Then
     \begin{align*}
	\sum_{k=1}^\infty	\frac{n_{k+1} -n_{k}}{\sqrt{n_k} n_{k+1}}  <+\infty.
\end{align*}  
 \end{lemma}



The second lemma of this section characterizes the speed of convergence of a sequence satisfying a recursion up to errors with polynomial decay. Therein, the error term $w_k$ naturally determines the rate of convergence. Its proof can be found in Appendix~\ref{appendix:sequences}. To the best of our knowledge, this result has not been shown in the literature. Typically, recursions of the type \eqref{eq:recursion-lemma} below appear with no error terms, see e.g. \cite[Lemma 1]{aragon2018accelerating} and \cite[Lemma 2]{bento2025convergence}. We refer to~\cite{MR64365} and \cite[Lemmas 4 \& 5, Chap.2, p.45]{PolyakBook} for recursion formulas with decaying errors. 

    \begin{lemma} \label{lemma:sequences}
        Let $(\alpha_k)_{k \in \N}$ be a nonnegative sequence converging to zero, and $(w_k)_{k \in \N}$ be a nonnegative sequence such that for some $p_1,p_2 >0$, \[w_k = \mathcal{O}\bigl(k^{-p_2} \ln^{p_1}(k)\bigr) .\] Suppose that there exist $c >0$ and $q \in (0,2)$, such that
        \begin{equation}\label{eq:recursion-lemma}
            \alpha_{k+1}^{q} \leq c(\alpha_k - \alpha_{k+1}) + w_k,\quad  \mbox{~for all~} k \in \N.      \end{equation} Then the following hold.
            \begin{enumerate}[label=(\roman*)]
                \item[(i)] If $q\in(0,1)$, then  $\alpha_k = o\bigl(k^{-p_2} \ln^{p_1}(k)\bigr) $;
                \item[(ii)] If $q=1$, then $\alpha_k = \mathcal{O}\bigl(k^{-p_2} \ln^{p_1}(k)\bigr) $;
                \item[(iii)] If  $q \in (1, 2)$, $p_1 =q$ and $p_2 > 1$, then $\alpha_k=  \mathcal{ O  } \left(     k^{-\min\{ p_2, \frac{q}{q-1} \} +1} \ln(k) \right)   $.
            \end{enumerate}
    \end{lemma}

\section{Proximal stochastic   subgradient algorithm}\label{sec:3}\vspace*{-0.1in}

In this section,  we  formulate and analyze a PSS  algorithm for solving the nonconvex problem \eqref{ProblemP}, 
corresponding to Algorithm~\ref{alg:1}. Throughout this section and thereafter, we have the following blanket assumptions. 

\begin{assumption}[On problem~\eqref{ProblemP}]\label{Assumption01}
    Let $\vPhi: \mathbb{R}^d \to \Rex$ be the objective function of problem~\eqref{ProblemP}. We assume the following:
    \begin{enumerate}[label=(\roman*),ref=(\roman*)]
        \item \label{Assumption01-i} $\psi: \mathbb{R}^d \to \Rex$ is a proper, lsc, prox-bounded function;
        \item \label{Assumption01-ii} $\varphi: \Xi \times \mathbb{R}^d \to \R$ is an upper-$C^2$ normal integrand;
        \item The function $\xi \mapsto \inf_{x\in \R^d} \left( \varphi(\xi, x) +\psi(x) \right)$ is integrable.
    \end{enumerate}
\end{assumption}

\begin{remark}\label{remark:bv} A few comments on Assumption~\ref{Assumption01} are in order.
  \begin{enumerate}[label=(\roman*)]
     \item It is common in the literature to impose smoothness assumptions on the integrand 
$\varphi$ and require the stochastic gradient to be unbiased, that is, $\mathbb{E}\big[\nabla_x\varphi(\xi_i,x)\big]
    =
    \nabla \Intf{\varphi}(x)$ 
for a suitable class of points $x\in\mathbb{R}^d$. However, this property is inapplicable in our setting as the integrand is nonsmooth. Instead, one can show that any measurable selection $v_i^k\in \partial_x \varphi(\xi_i,x)$ satisfies $\mathbb{E}[v_i^k]\in \partial \Intf{\varphi}(x)$. This condition is closely related to the one used in \cite{MR3902455}, where the authors consider 
a measurable selection of subgradients 
$g:\Xi\times\mathbb{R}^d\to\mathbb{R}^d$ such that $\mathbb{E}_\mu[g(\xi,x)]\in \partial \Intf{\varphi}(x)$ for all  $x\in U$, where $U$ is an open set containing the constraint set. Furthermore, the latter work imposes a global second-moment bound on the stochastic subgradient 
oracle $g:\Xi\times\mathbb{R}^d\to\mathbb{R}^d$, namely, $\sup_{x\in U} \mathbb{E}_\mu \bigl[\|g(\xi,x)\|^2\bigr] <+\infty$. In contrast, since we assume only local Lipschitz continuity of $\varphi$; see 
Definition~\ref{defin_locally_Lips}; such a global moment bound may fail in our framework.

\item Let us emphasize that, in the definition of an upper-$C^2$ normal integrand, we 
implicitly assume that the integrand is locally Lipschitz in the sense of 
Definition~\ref{defin_locally_Lips}. This assumption is closely related to the 
standard uniform bounded-variance condition
\[
    \sup_{x\in \R^d}
    \mathbb{E}_\mu\left[
        \left\|
        \nabla_x \varphi_\xi(x)
        -
        \nabla \Intf{\varphi}(x)
        \right\|^2
    \right] <+\infty,
\]
or to its corresponding version on a prescribed subset of $\R^d$; see, e.g.,
\cite{MR3439803,MR3459195} and the references therein. More recently, this condition has been weakened by allowing multiplicative 
noise~\cite{MR3935081}. Particularly,  the stochastic oracle is allowed to have a variance that grows with the 
distance between points, rather than being uniformly bounded on the whole set. Moreover, when $\varphi$ is convex with respect to the decision variable 
$x\in\R^d$, the multiplicative-noise assumption 
implies local Lipschitz continuity of the 
integrand in the sense of Definition~\ref{defin_locally_Lips} provided that $C$ is an open convex set and that there exists $\hat{x}\in\R^d$ such that $\varphi(\cdot,\hat{x})\in L^2_\mu$. Indeed, for convex normal integrands, suitable boundedness of the values is 
equivalent to Lipschitz continuity of the integrand; see, e.g.,
\cite[Theorem~2]{MR4261271}.

  \end{enumerate}
\end{remark}

 The  method proposed in Algorithm~\ref{alg:1} approximates the expected value in \eqref{ProblemP} by means of an empirical average that is calculated in each iteration by adding new terms to the sum coming from new samples of the underlying random variable of the model. Then, a proximal subgradient step is performed on the approximation, where the stepsize is defined via line search. Observe that, in principle, the line search can restart the initial candidate stepsize in each (outer) iteration, thus allowing a nonmonotone sequence of stepsizes. In addition, note that only one stochastic subgradient is computed every time the line search is conducted, in contrast to other proposed line searches for stochastic methods that require new gradient computations at every trial; see, e.g.,~\cite{paquette2020stochastic,nguyen2025stochastic}.

  \begin{algorithm}[H]
  	\caption{  Proximal Stochastic Subgradient (PSS) algorithm for problem~\eqref{ProblemP}}\label{alg:1}
  	\begin{algorithmic}[1]
  		\Require{$x^0 \in\dom\psi $, $\rho\in{(0,1)}$, $\sigma>0$, $\gamma_{\min} >0$ and sample size $ (n_k)_{k \in \N} $ with $0<n_k \leq  n_{k+1}$ and $n_k \to \infty $ as $k \to \infty$. Set $n_{-1}:=0$ and $k:=0$.}
  		\State{If $n_k > n_{k-1}$, draw  $\xi_j\sim \mu$ i.i.d. for $j=n_{k-1}+1, \ldots,  n_{k}$.}
  		\State{\label{step:subgradient}Choose $v_i^{k}\in\partial \varphi_{\xi_i}( x^k)$ for $i=1, \dots, n_k$.}
  		\State{Define $v^k := \frac{1}{n_k}\sum_{i=1}^{n_k} v_i^k$  and set $\mathrm{\vPhi}_k$ using \eqref{defPhik}.}\Comment{Subgradient estimate}
  		\State{Take   $\overline{\gamma}_k \in [\gamma_{\min}, \gamma^\psi [  $. Set $\gamma_k:=\overline{\gamma}_k$ . }
  		\State{Compute \Comment{Proximal line search}
  			\begin{equation}\label{eq:algx}
  				\begin{aligned}
  					\hat{x}^{k} & \in \prox_{\gamma_k \psi} \left(x^k -\gamma_k v^k\right) .
  				\end{aligned}
  			\end{equation}
  		}
  			\State{ \label{step:ls}\textbf{while}   $\mathrm{\vPhi}_k(\hat{x}^k) >\vPhi_k(x^{k}) -\sigma  \|   \hat{x}^{k}- x^{k}\|^2 $\Comment{Sample-based descent test}
  			 
  			\textbf{do }{ $\gamma_k\leftarrow \rho \gamma_k$ and recompute $\hat{x}_k$ according to \eqref{eq:algx}.}}

  		\State{\label{step:newiterate} Set $x^{k+1}:=\hat{x}^{k}$,  update $k \leftarrow k+1$,      and go back to Step~1.}
  	\end{algorithmic}
  \end{algorithm}

\subsection{Well-posedness of the iterative scheme in Algorithm~\ref{alg:1}}

We start the analysis of Algorithm~\ref{alg:1} by showing it is well-defined, which accounts for proving that the line search terminates after finitely many trials in each iteration. We first show a technical proposition from  which the desired result will follow.
 
\begin{proposition}\label{LS-terminates}
Let $\varphi$ and $\psi$ be functions satisfying Assumption~\ref{Assumption01}\ref{Assumption01-i}--\ref{Assumption01-ii}, and suppose that Assumption~\ref{Assumption00} holds for the sampling $\xi_i$, with $i \geq 1$. Let $\bar x\in\dom\psi$ and  $\sigma >0$. Then, there exists a measurable set $\hat{\Omega}$ with $\mathbb{P}(\Omega \setminus \hat{\Omega}) = 0$ such that, for all $\omega \in \hat{\Omega}$,  there exist a neighborhood $W$ of $\bar x$ and a constant $\vartheta>0$ such that for all $x\in W$,   $k \in \N $,  and   $v\in \partial \Intf{\varphi}^k(x)$, one has 
\[
\mathrm{\vPhi}_{k}(w) \leq \vPhi_{k}(x) -\sigma  \| w- x\|^2,
\]
provided that  
\begin{equation}\label{eq:lswd-w}
w \in \prox_{\gamma \psi} \left(x -\gamma v\right), \quad \text{with } \gamma\in{(0,\vartheta).}
\end{equation}

\end{proposition}

\begin{proof}
First, by Proposition~\ref{Prop_eq_Lipsc_cont_Unif} and~Proposition~\ref{upperC2unif},  for every $\rr \in \N$, we can consider $\kappa_\rr\in L^2_\mu$ and $\rA_\rr \in L^1_\mu$ such that~\eqref{eq_Lipsc_cont_Unif} and~\eqref{integral_equpperdesc} hold on $\mathbb{B}_{\rr}[0]$. Now, by the strong law of large numbers, we can assume that there exists a  measurable set $\Omega_\rr$ with $\mathbb{P}(\Omega_\rr) =1$ such that  for all $\omega  \in  \Omega_\rr$
$$ \varsigma_\rr := \sup_{k \in \N  }  \frac{1}{n_k}\sum_{i=1}^{n_k} \rA_\rr(\xi_i) <+\infty \quad  \text{and} \quad  \tilde{\varsigma}_\rr:=\sup_{k \in \N  }\frac{1}{n_k}\sum_{i=1}^{n_k}\kappa_\rr(\xi_i) < +\infty.$$ Now, we set $\hat{\Omega}:= \bigcap_{\rr \in \N} \Omega_\rr$, a set of full measure.  Fix $\omega \in \hat{\Omega} $,  $\bar x \in \dom \psi$ and $\sigma>0$, and take $\rr_0 \in \N $ such that $\|\bar x\| \leq \rr_0-1$. Let $V$ be a neighborhood of $\bar x$ such that  $V \subset \mathbb{B}_{\rr_0}[0]$. 
 In addition, for any $x \in V$ and $v \in \partial \Intf{\varphi}^k(x)$, in view of \eqref{eqsubd02} and \cite[Proposition 2.1.2]{MR1058436}, there exist $v_{i} \in \partial\varphi_{\xi_i}(x)$, $j = 1, \dots, n_k$, such that \[\|v\| \leq \frac{1}{n_k}\sum_{i=1}^{n_k}\|v_{i}\| \leq \frac{1}{n_k}\sum_{i=1}^{n_k}\kappa_{\rr_0}(\xi_i) \leq \tilde{\varsigma}_{\rr_0}.\] Hence, the set $\{v\in\partial \Intf{\varphi}^k(x) : x \in V, \:k \in \N\} $ is bounded. Then, by Lemma~\ref{upper_cont_prox}, we can consider a neighborhood $W \subseteq V$ of $\bar x$, a positive constant $\vartheta_0$ so that any vector $w$ given by~\eqref{eq:lswd-w} with $x \in W$ and $\gamma \in{(0,\vartheta_0)}$ belongs to $V$. By resorting again to~\eqref{integral_equpperdesc}, we have that, for all $k \in\N$ and all  $\xi_i$, $i=1,\ldots,n_k$, the following inequality holds
\[
\varphi_{\xi_i}(w)- \varphi_{\xi_i}(x) - \rA_{\rr_0}(\xi_i)\|w-x\|^2 \leq \langle v_i,w -x\rangle.
\]
By summing the above inequalities for $i=1,\ldots,n_k$ and dividing by $n_k$, we get
\begin{equation}\label{eq:lswd-upper}
\Intf{\varphi}^k (w)- \Intf{\varphi}^k(x) -  \frac{1}{n_k}\sum_{i=1}^{n_k} \rA_{\rr_0}(\xi_i)\|w-x\|^2 \leq \langle v,w -x\rangle, \quad \text{for all } k\in\N.
\end{equation} By~\eqref{eq:lswd-w}, the definition of the proximity operator and rearranging terms we get
\begin{equation}\label{eq:lswd-prox}
\psi(w) + \langle v,w-x\rangle + \frac{1}{2\gamma} \|w-x\|^2 \leq \psi(x).
\end{equation}
Summing~\eqref{eq:lswd-upper} and \eqref{eq:lswd-prox} yields
\begin{equation*}
\vPhi_k (w) +  \left( \frac{1}{2\gamma} - \frac{1}{n_k}\sum_{i=1}^{n_k} \rA_{\rr_0}(\xi_i) \right) \|w-x\|^2 \leq \vPhi_k(x), \quad \text{for all } k\in\N.
\end{equation*}
Therefore, it suffices to take $\vartheta:= \min\{\vartheta_0, \bigl(2(\sigma + \varsigma_{\rr_0})\bigr)^{-1}\}$ to obtain the desired result for all $k\in\N$.
\end{proof}

A direct consequence of Proposition~\ref{LS-terminates} is that the line search procedure in Algorithm~\ref{alg:1} is well-defined, as shown in the next result.
\begin{corollary}[Finite termination of line search]
    In Algorithm~\ref{alg:1}, under Assumption~\ref{Assumption01}\ref{Assumption01-i}--\ref{Assumption01-ii}, for each iteration $k \in \N$, the line search procedure described in Step~\ref{step:ls} terminates after a finite number of (inner) steps.
\end{corollary}

\begin{proof}
   Fix $k \in \N$. Apply Proposition~\ref{LS-terminates} with $\bar x= x = x^k$ and $v = v^k \in \partial \Intf{\varphi}^k(x^k)$, where the inclusion holds from \eqref{eqsubd02}. Then, after finitely many trials of the line search, $\gamma_k \in (0,\vartheta)$, thus the line search terminates with $\mathrm{\vPhi}_k(\hat{x}^k) \leq \vPhi_k(x^{k}) -\sigma  \|   \hat{x}^{k}- x^{k}\|^2 $.
\end{proof}

\subsection{Subsequential convergence analysis of PSS}

We now proceed to show subsequential convergence of the sequence generated by Algorithm~\ref{alg:1}, for which we need the following assumption.

\begin{assumption}[On Algorithm~\ref{alg:1}]\label{Assumption02}   
Let $(n_k)_{k\in\mathbb N}\subseteq\mathbb N{\setminus}\{0\}$ be a nondecreasing sequence of
sample sizes such that $n_k\to+\infty$, and let $(x^k)_{k\in\mathbb N}$ be
the sequence generated by Algorithm~\ref{alg:1}. We consider a filtration
$\mathbb F:=(\mathcal F_k)_{k\in\N}$, representing all the information used
by Algorithm~\ref{alg:1} up to iteration $k$, such that $\xi_j$ is
$\mathcal F_k$-measurable for every $j \leq n_k$, and $\xi_j$ is
independent of $\mathcal F_k$ whenever $j>n_k$.  
 
\end{assumption}

\begin{remark}[Measurability of the sequence generated by Algorithm \ref{alg:1}] 
	It is important to note that there is no guarantee that the sequence produced by Algorithm \ref{alg:1} will be a sequence of  measurable functions, even more adapted to the filtration $\mathbb{F}$ in Assumption~\ref{Assumption00}. This issue arises primarily from the fact that the first order information considered in the algorithm comes from set-valued mappings (Step~\ref{step:subgradient}), or from the proximal subproblem~\eqref{eq:algx} possibly admitting multiple solutions. A common approach to overcome this issue is to assume the existence of a selector for the subdifferential and for the proximal mapping in \eqref{eq:algx}. Under these assumptions, the sequence $(x^{k})_{k\in\N}$ becomes an adapted process with respect to the filtration, which is sufficient for analyzing the asymptotic behavior of the sequence. However, we prefer to state this condition as an assumption rather than impose restrictions on the potential applications of Algorithm \ref{alg:1}.
\end{remark}

We first state the main theorem of this section.  Then, in Section~\ref{s:lemmas-The01}, we present some technical lemmas instrumental for the proof of the theorem, which will be developed in Section~\ref{s:proof-The01}. 

\begin{theorem}[Subsequential convergence]\label{The01}
Let $\vPhi:\mathbb{R}^d\to\Rex$ denote the objective function of problem~\eqref{ProblemP}, and assume that Assumption~\ref{Assumption01} holds. Given an initial point $x^0\in\dom \vPhi$, consider the sequence $(x^k)_{k\in\mathbb{N}}$ generated by Algorithm~\ref{alg:1} under Assumptions~\ref{Assumption00} and~\ref{Assumption02}. Then, there exists a measurable set $\hat{\Omega}$ with $\mathbb{P}(\Omega \setminus \hat{\Omega}) = 0$, such that for all $\omega \in \hat{\Omega}$, if the sequence   $(x^k)_{k \in \mathbb{N}}$ is bounded, then  both $\bigl(\vPhi_k(x^k)\bigr)_{k\in\N}$ and $\bigl(\vPhi(x^k)\bigr)_{k\in\N}$ converge to the same limit,  
\begin{equation}\label{sumsquare_as}
	\sum_{k \in \mathbb{N}} \|x^{k+1} - x^k\|^2 < +\infty ,
\end{equation} 
and the following assertions hold:

 \begin{enumerate}[label=(\roman*),ref=(\roman*)]
\item\label{itema} $\inf_{k \in \N} \gamma_k>0 $; 
 
\item\label{item2} if $x^{k_j}\to \bar{x}$ as $j\to\infty$, then $\bar{x}$ is a stationary  point of problem \eqref{ProblemP} in the sense of \eqref{def_stpoint}, and $\vPhi(\bar x) = \lim\limits_{k \to \infty }  \vPhi_k (x^{k})= \lim\limits_{k \to \infty }  \vPhi (x^{k})$;
\item\label{item3} the set of accumulation points of $(x^{k})_{k\in\N}$ is nonempty, closed, and connected;
\item\label{item4}  if $(x^{k})_{k\in\N }$ has an isolated accumulation point $\bar{x}$, then the entire sequence $(x^{k})_{k\in\N }$  converges to the stationary point $\bar{x}$ as $k\to\infty$.
 \end{enumerate}
\end{theorem}

\subsubsection{Technical results for Theorem~\ref{The01}} \label{s:lemmas-The01} 

We start the analysis by defining a stopping time related to the boundedness of the sequence $(x^{k})_{k\in\N}$.  Its proof is direct by definition.

\begin{lemma}\label{l-subseq:tech-1}
Suppose that Assumption~\ref{Assumption02} holds. Given $\rr>0$ such that
$\|x^0\|\leq \rr$, define
\[
    \tau_{\rr}:=\inf\big\{k\in\mathbb N:\|x^{k+1}\|>\rr\big\},
\]
with the convention $\inf\emptyset=+\infty$, and consider the stopped
process
\[
    x_{\rr}^k:=x^{k\wedge \tau_{\rr}},
    \qquad 
    k\wedge \tau_{\rr}:=\min\{k,\tau_{\rr}\},
    \qquad \text{for every } k\in\mathbb N.
\]
Then $\tau_{\rr}$ is a stopping time with respect to the filtration
$(\mathcal F_k)_{k\in\N}$ in Assumption~\ref{Assumption02}. Furthermore,
$x_{\rr}^k$ is $\mathcal F_{k-1}$-measurable for every $k\in\mathbb N\backslash\{0\}$, and
the stopped process $(x_{\rr}^k)_{k\in\mathbb N}$ satisfies
\[
    \tau_{\rr}\geq p
    \quad \Longrightarrow \quad
    \|x^p\|\leq \rr,
    \qquad \text{for every } p\in\mathbb N.
\]
In particular,
\[
    \|x_{\rr}^k\|\leq \rr,
    \qquad \text{for every } k\in\mathbb N.
\]
\end{lemma}


In the next lemma, we show that the stopped sequence defined in Lemma~\ref{l-subseq:tech-1} satisfies a sufficient decrease condition up to errors.

\begin{lemma} \label{l-subseq:tech-2}
     Suppose that the assumptions of Theorem~\ref{The01} hold. In the setting of Lemma~\ref{l-subseq:tech-1} define the sequence of  random variables
\begin{align*}
	\vPhi_k^{\rr} := \vPhi_{  k \wedge \tau_{\rr}}(x^{k \wedge \tau_{\rr}}), \quad \text{for all } k\in\N.
\end{align*}  Then, for all $k\in\N$, it holds that
\begin{align}\label{eqclaim2}
		\sigma  \| x^{k+1}_{\rr} - x^{k}_{\rr}\|^2 \leq \vPhi_k^{\rr} - \vPhi_{k+1}^{\rr}  + W_k^{\rr} , 
\end{align}
where the random variable $W^{\rr}_k$     is defined by  \begin{align}\label{defWk} 
	W^{\rr}_k := \left(\Intf{\varphi}^{k+1}(x^{k+1}) - \Intf{\varphi}^{k}(x^{k+1}) \right) \1_{\{  \tau_{\rr}\geq k+1 \}}.
\end{align}

\end{lemma}

\begin{proof} From the definition of $x^{k+1}$ in Step~\ref{step:newiterate} of Algorithm~\ref{alg:1} it follows that 
	\begin{equation} \label{eq:descent}
	    \vPhi_k(x^{k+1}) \leq  \vPhi_k(x^k) -  \sigma   \| x^{k+1} - x^k \|^2, \quad  \text{for all } k\in\N.
	\end{equation}
	By~\eqref{defEk}-\eqref{defPhik}, then
	\begin{align}\label{01ineq_est}
		\begin{aligned}
			\sigma \| x^{k+1} - x^k \|^2 &\leq \vPhi_k(x^k) - \vPhi_k(x^{k+1}) = \Delta \vPhi_k + \Intf{\varphi}^{k+1}(x^{k+1}) - \Intf{\varphi}^{k}(x^{k+1}), 
		\end{aligned}
	\end{align}
	where 
	\[
	\Delta \vPhi_k := \vPhi_k(x^k) - \vPhi_{k+1}(x^{k+1}), \quad  \text{for all } k\in\N.
	\] Given a fixed $k \in \N$, note that 
\begin{equation}\label{eqvaluesXa}
	\| x^{k+1}_{\rr} - x^{k}_{\rr}\| = \begin{cases}
		0, &\text{ if  }  \tau_{\rr} \leq k, \\
			\| x^{k+1} - x^{k}\|, &\text{ if  }  \tau_{\rr} \geq k+1,
	\end{cases}
\end{equation}
and 
\begin{equation}\label{eqvaluesXa2}
	\vPhi_k^{\rr} = \begin{cases}
\vPhi_{\tau_{\rr}}(x^{\tau_{\rr}}), 	 &\text{ if  }  \tau_{\rr} \leq k, \\
		\vPhi_k(x^k),  &\text{ if  }  \tau_{\rr} \geq k+1.
	\end{cases}
\end{equation}
 Hence, on the one hand,  using \eqref{01ineq_est} and \eqref{eqvaluesXa} on  $\{ \tau_{\rr} \geq k+1\} $,  we have  that
 \begin{align*}
 		\sigma  \| x^{k+1}_{\rr} - x^{k}_{\rr}\|^2 \leq  \Delta \vPhi_k + \Intf{\varphi}^{k+1}(x^{k+1}) - \Intf{\varphi}^{k}(x^{k+1}) = \vPhi_k^{\rr} - \vPhi_{k+1}^{\rr}  + W_k^{\rr}.
 \end{align*}
On the other hand, on $\{ \tau_{\rr} \leq k\} $  we simply have  
 \begin{align*}
	\sigma  \| x^{k+1}_{\rr} - x^{k}_{\rr}\|^2=0 \quad \text{and} \quad  \vPhi_k^{\rr} - \vPhi_{k+1}^{\rr}  + W_k^{\rr}=0,
\end{align*} 
so we conclude that \eqref{eqclaim2} holds.
\end{proof}

The estimate shown in Lemma~\ref{l-subseq:tech-2} is not a descent condition in the traditional sense due to the random variables  $(W^{\rr}_k)_{k\in\N}$, defined in \eqref{defWk}, appearing on the right-hand side of \eqref{eqclaim2}. Nevertheless, as shown in the next result, (the conditional expectation of) this sequence is vanishing. 

\begin{lemma} \label{l-subseq:tech-3}
   Let us assume the setting of Lemma~\ref{l-subseq:tech-2}. Then, for all $k\in\N$ and $r>0$ such that $\|x^0\|\leq r$, the random variable $W_k^{\rr}$ given in \eqref{defWk} is integrable. Furthermore, let  $Z_k^{\rr}:= \mathbb{E}\bigl[  W_k^{\rr} | \mathcal{F}_k \bigr]$. Then 
$ \sum_{k \in \N } \mathbb{E}\bigl[ |Z_k^{\rr}| \bigr]< +\infty$.  
\end{lemma}
\begin{proof}
    Let us consider an arbitrary $r>0$ such that $\|x^0\|\leq r$. From Lemma~\ref{l-subseq:tech-1}, $\| x_{\rr}^k\| \leq \rr$ for all $k\in \N$. Furthermore,  by Proposition \ref{Prop_eq_Lipsc_cont_Unif}, there exist  square integrable functions $\kappa$ and $\lA$  such that for all $\xi \in \Xi$, 
\begin{align}\label{inq}
	|  \varphi_\xi ( u) | \leq \kappa(\xi)\| u\| + \lA(\xi), \quad \text{for all } u\in \mathbb{B}_{\rr}[0].
\end{align}
Therefore, for all  $k\in\N$,
\begin{align*}
	|W_k^{\rr}|  \leq     \frac{1}{n_{k+1}}  \sum_{i=1}^{n_{k+1}}  \left( \kappa(\xi_i)\rr + \lA(\xi_i)  \right) +   \frac{1}{n_{k}} \sum_{i=1}^{n_k}\left(  \kappa(\xi_i)\rr  + \lA(\xi_i) \right).
\end{align*}      
 Hence,  $\mathbb{E}\bigl[	|W_k^{\rr}| \bigr] \leq 2 \mathbb{E} [\kappa ] \rr + 2\mathbb{E} [ \lA ]$. Moreover, observe that 
\[
W_k^{\rr}= \left(\frac{1}{n_{k+1}} \sum_{j=n_k + 1}^{n_{k+1}} \varphi_{\xi_j}(x^{k+1}_{\rr}) + \frac{n_k - n_{k+1}}{n_{k+1}} \Intf{\varphi}^{k}(x^{k+1}_{\rr})\right)  \1_{\{  \tau_{\rr}\geq k+1 \}}.
\]
Now,  using that  $\{  \tau_{\rr}\geq k+1 \} \in \mathcal{F}_k$, because $\tau_{\rr}$ is a stopping time (Lemma~\ref{l-subseq:tech-1}), we get that \begin{equation*}
    Z_k^{\rr} =  \mathbb{E}\left[\frac{1}{n_{k+1}} \sum_{j=n_k + 1}^{n_{k+1}} \varphi_{\xi_j}( x^{k+1}_{\rr})  + \frac{n_k - n_{k+1}}{n_{k+1}} \Intf{\varphi}^{k}(x^{k+1}_{\rr}) \; \Biggl| \; \mathcal{F}_k \right] \1_{\{  \tau_{\rr}\geq k+1 \}}.
\end{equation*}
Recall from Lemma~\ref{l-subseq:tech-1} that the function $x^{k+1}_{\rr}$ is $\mathcal{F}_{k}$-measurable, and by \eqref{inq} the function $\omega \mapsto\varphi_{\xi_j(\omega)}(x_{\rr}^{k+1}(\omega)) $ is integrable,  so using Lemma~\ref{lemma_measurability} we get that 
\[
\mathbb{E}\left[ \varphi_{\xi_j}(x_{\rr}^{k+1}) \mid \mathcal{F}_k \right](\omega) = \Intf{\varphi}(x_{\rr}^{k+1}(\omega))   \text{ a.s.,} 
\]
for all $j = n_k + 1, \dots, n_{k+1}$. This shows that 
\begin{align*}
	Z^{\rr}_k= \frac{n_{k+1} - n_k}{n_{k+1}} \left( \Intf{\varphi}(x_{\rr}^{k+1}) - \Intf{\varphi}^{k}(x_{\rr}^{k+1}) \right)  \1_{\{  \tau_{\rr}\geq k+1 \}}  \text{ a.s.}
\end{align*}   Since $  \| x^k_{\rr} \| \leq \rr$ for all $k \in \N$, it follows from Lemma~\ref{lemma_logn}\ref{lemma_logn-ii} that there exists a constant ${a}_{\rr} > 0$, independent of $k$, such that 
\begin{equation*}
	\mathbb{E}\bigl[|Z^\rr_k|\bigr] \leq a_{\rr} \frac{n_{k+1} - n_k}{n_{k+1}} \frac{1}{\sqrt{n_k}}.
\end{equation*}
By Lemma~\ref{l:bound_sample01}, this yields  $\sum_{k\in \N} \mathbb{E}\big[|Z_k^\rr|\bigr]  < +\infty$, concluding the proof.
\end{proof}

We now prove that the sequence generated by Algorithm~\ref{alg:1} has square-summable increments. In addition, we show that the sequence of value approximations 
along the iterates is 
convergent. 
This is stated in the following lemma.

\begin{lemma} \label{l-subseq:tech-4}
   Suppose the assumptions of Lemma~\ref{l-subseq:tech-3} hold. Then, $(\vPhi^{\rr}_k)_{k \in \N }$    converges a.s. 
     and   	\begin{align}\label{eqclaim5} 		 \sum_{k =0}^\infty  \| x_{\rr}^{k+1} - x_{\rr}^{k}\|^2 < +\infty \text{~a.s}. 	
     \end{align}
\end{lemma}
\begin{proof}
Let $\kappa$ and $\lA$ be the functions in~\eqref{inq} (which exist in view of Proposition \ref{Prop_eq_Lipsc_cont_Unif}). Let us define
\[
s_k := \frac{1}{n_k}\sum_{i=1}^{n_k}\bigl( \rr \kappa(\xi_i)+\lA(\xi_i) \bigr) - \mathfrak{i}_{\psi}^{\rr} \quad \text{and}
\quad 
s_k^{\rr} := s_{k\wedge \tau_{\rr}}, \qquad \text{for all~} k \in \N,
\]
where $\mathfrak{i}_{\psi}^{\rr} :=  \inf_{ x\in \mathbb{B}_{\rr }[0] } \psi(x)$. It is clear that $s_k^{ \rr }$ is (square) integrable. Moreover, by   \eqref{eqclaim2} and  \eqref{inq}, 
\begin{align*}
 - s_{k+1}^{ \rr }   \leq \vPhi_{k+1}^{\rr} \leq	\vPhi_0^{\rr}   + \sum_{i=0}^k W_i^{\rr},
\end{align*}
which particularly implies that $\vPhi_{k+1}^{\rr}$ is integrable in view of Lemma~\ref{l-subseq:tech-3}.  Next, we introduce the process
\[
Y_k^{\rr}
:= \mathbb{E}\!\left[\, s_{k+1}^{\rr} - s_k^{\rr} \,\middle|\, \mathcal{F}_k \right], \qquad \mbox{for all~} k\in\N.
\] 
Notice that \eqref{eqclaim2} yields,  for all $k \in \N$,
\begin{align}\label{inequality_supermartingale}
\mathbb{E} \bigl[ 	\vPhi_{k+1}^\rr + s_{k+1}^{\rr} \mid \mathcal{F}_{k}\bigr] \leq 	\vPhi_{k}^\rr + s_{k}^{\rr}-  \sigma  \| x^{k+1}_{\rr} - x^{k}_{\rr}\|^2  +  |Y_k^\rr| + |Z_k^\rr|.
\end{align} A direct computation yields
\[
Y_k^{\rr} = \frac{n_{k+1}-n_k}{n_{k+1}} \left( \mathbb{E} \bigr[ \rr \kappa(\xi) + \lA(\xi)  \bigr] - \frac{1}{n_k} \sum_{i=1}^{n_k} \rr \kappa(\xi_i) + \lA(\xi_i) \right).
\]
Hence, by Lemma~\ref{lemma_logn}, we conclude the existence of a constant $b_{\rr}>0$ such that  
\begin{align*}
	\mathbb{E}\bigl[|Y^\rr_k|\bigr] \leq b_{\rr} \frac{n_{k+1} - n_k}{n_{k+1}} \frac{1}{\sqrt{n_k}},
\end{align*}
where $b_{\rr}$ is the variance of the square integrable function $\xi \mapsto \rr \kappa(\xi) + \lA(\xi)$.
In this manner, Lemma~\ref{l:bound_sample01} implies  $\sum_{k\in \N} \mathbb{E}\big[|Y_k^\rr|\bigr]  < +\infty$. Therefore, by  Lemma~\ref{l-subseq:tech-3},  the nonnegativity of the sequence $	\bigl(\vPhi_{k}^\rr + s_{k}^{\rr}\bigr)_{k\in\N}$, and \eqref{inequality_supermartingale}, 
it follows from the well-known Robbins–Siegmund supermartingale convergence theorem (see \cite[Theorem 1]{MR343355}) that $	\bigl(\vPhi_{k}^\rr + s_{k}^{\rr}\bigr)_{k\in\N}$ converges  and \eqref{eqclaim5} holds.
Moreover, by the  strong law of large numbers 
we have, a.s., that
\begin{align*}
\lim_{k \to \infty }s_k^{\rr} = \begin{cases}
    \mathbb{E} (\rr   \kappa + \lA) - \mathfrak{i}_{\psi}^{\rr}, & \text{ if } \tau_\rr =+\infty,\\
    s_{\tau_\rr} - \mathfrak{i}_{\psi}^{\rr}, & \text{ if } \tau_\rr <+\infty.\\
\end{cases}  
\end{align*}
 Consequently, $\bigl(\vPhi_k^\rr\bigr)_{k\in\N}$ converges a.s.

\end{proof}

\subsubsection{Proof of Theorem~\ref{The01}} \label{s:proof-The01} 

We are now in a position to prove each assertion of Theorem~\ref{The01}.

\medskip

\begin{proofof}{The01}  We consider  a measurable set  $\hat{\Omega}$ of full measure such that for all $\omega \in \hat{\Omega}$, the conclusions of Proposition~\ref{LS-terminates} and  the following conditions hold:  for all $\rr \in \N $ with $\| x^0\|\leq\rr$,  
 \begin{enumerate}[label=\alph*$_\omega$), ref=\alph*$_\omega$),leftmargin=2.0em,
  itemsep=1em]
	\item 
    from Lemma~\ref{l-subseq:tech-4},
	\begin{align}\label{inqalmostsurely}
		\sum_{k \in \mathbb{N}} \|x_r^{k+1}(\omega) - x_r^k(\omega)\|^2 < +\infty
	\end{align}
and $\vPhi_k^\rr$ converges to some random variable;
    
	\item\label{comega} by Lemma~\ref{lemmarateconver} we get that 
\begin{align*}
	\left|	\Intf{\varphi}(x_{\rr}^k) -   \Intf{\varphi}^k  (x_{\rr}^k)	  \right|  \to 0;
\end{align*} 
	\item \label{bomega} $\frac{1}{n_k} \sum_{j=1}^{n_k} \kappa_\rr(\xi_j(\omega)) \to \mathbb{E}\left[\kappa_\rr\right]$, where $\kappa_\rr$ is chosen according to \eqref{eq_Lipsc_cont_Unif} in Proposition \ref{Prop_eq_Lipsc_cont_Unif}  for the bounded set $\mathbb{B}_{\rr}[0]$ with $\rr \in \N$. 
	
	
\end{enumerate}
Let $ \omega \in \hat{\Omega}$ such that $(x^k(\omega))_{k \in \N}$ is bounded. Let us consider $\rr \geq \sup_{  k\in \N} \| x^k(\omega)\| +1$. So, due to the equations~\eqref{eqvaluesXa} and~\eqref{eqvaluesXa2}, we get that $\vPhi_k(x^k) = \vPhi_k^{\rr}   $  and $\| x^{k+1} - x^{k}\| = \| x_{\rr}^{k+1} - x_{\rr}^{k}\|$ for all $k \in \N$. Therefore, $\bigl(\vPhi_k(x^k) \bigr)_{k\in\N} $ converges  and \eqref{sumsquare_as} holds due to \eqref{inqalmostsurely}.

 Moreover, by \ref{comega},   \eqref{eqvaluesXa}  and the fact  that $\tau_{\rr}=+\infty$, we get that   	$$ \left|	\Intf{\varphi}(x^k) -   \Intf{\varphi}^k  (x^k)	  \right|  \to 0,$$ and since $\vPhi(x^k) - \vPhi_k(x^k)    =  	\Intf{\varphi}(x^k) -   \Intf{\varphi}^k  (x^k)$,
we deduce that $\bigl( \vPhi(x^k)\bigr)_{k\in\N}$ converges a.s. to the same limit as $\bigl(\vPhi_k(x^k) \bigr)_{k\in\N} $.

Now, let us continue with item~\ref{itema}, showing that $\inf_{k\in \N} \gamma_k(\omega)>0$. We proceed by way of contradiction,  assuming that there exists a subsequence $(\gamma_{\nu_i})_{i \in \mathbb{N}}$ of $(\gamma_k)_{k \in \mathbb{N}}$ such that $\gamma_{\nu_i} \to 0^+$ as $i \to \infty$. Since $(x^k)_{k\in\N}$ is bounded, we may assume that  $x^{\nu_i} \to \bar{x}$. Take $\lambda_{\nu_i} := \rho^{-1}\gamma_{\nu_i}$ 
and $w^{\nu_i} \in \prox_{\lambda_{\nu_i} \psi} \left(x^{\nu_i} - \lambda_{\nu_i} v^{\nu_i}\right)$, where $v^{\nu_i} = \frac{1}{n_{\nu_i}}\sum_{j=1}^{n_{\nu_i}}v^{\nu_i}_j$ and $v^{\nu_i}_j \in \partial \varphi_{\xi_i}(x^{\nu_i})$ for all $j=1,\dots,n_{\nu_i}$. In particular, $\lambda_{\nu_i}$ is the stepsize of one trial prior to passing the test in Step~\ref{step:ls} of Algorithm~\ref{alg:1} in iteration $\nu_i$. Hence, 
for all $i \in \N$, \[\vPhi_{\nu_i}(w^{\nu_i}) > \vPhi_{\nu_i}(x^{\nu_i}) - \sigma \|w^{\nu_i} - x^{\nu_i}\|^2.\] Furthermore, since $\gamma_{\nu_i} \to 0^+$, then for all $i_0 \in \N$ sufficiently large, in view of  Proposition~\ref{LS-terminates},  applied to $\bar x$, $x =x^{\nu_{i_0}}$ and $w=w^{\nu_{i_0}}$, we  get that
\[\mathrm{\vPhi}_{k}(w^{\nu_{i_0}}) \leq \vPhi_{k}(x^{\nu_{i_0}}) -\sigma  \| w^{\nu_{i_0}}- x^{\nu_{i_0}}\|^2, \quad \text{for all } k\in\mathbb{N},\] 
yielding a contradiction. Hence item~\ref{itema} holds true.

To prove item~\ref{item2}, let $x^{k_j}(\omega) \to \bar{x}$ as $j\to\infty$. We first prove that $\bar{x}$ is a stationary point of \eqref{ProblemP}. 
First, from item~\ref{itema}, we can assume that $\gamma_{\nu_i} \to \bar{\gamma} \in (0,+\infty]$. By  \eqref{inqalmostsurely}, it follows that $ {x}^{k_j+1} \to \bar{x}$. Moreover, by the definition of $\rr$, we conclude that $\bar{x} \in \mathbb{B}_{\rr'}(0)$, where $\rr' :=\rr
 - \frac{1}{2}$. 
Since $\varphi_{\xi_i}$ is Lipschitz on $\mathbb{B}_\rr(0)$ with modulus $\kappa_r(\xi_i)$, we have 
\[
\|v^{k_j}\| \leq \frac{1}{n_{k_j}} \sum_{i=1}^{n_{k_j}} \kappa_r(\xi_i), \text{ whenever } x^{k_j} \in \mathbb{B}_{\rr}(0),
\]
which shows that $v^{k_j}$ is bounded in view of \ref{bomega}. Therefore, up to a subsequence and without relabeling, we may assume that $v^{k_j} \to \bar{v}$. Now, by 
Lemma \ref{LemmaClarke}, we have that  $\bar{v} \in \partial \Intf{\varphi}(\bar{x})$. In addition, by \eqref{eq:algx} and the definition of the proximal operator, for all $u \in \mathbb{R}^d$, we have
\begin{equation}\label{ineqProx}
\begin{aligned}
	\psi( {x}^{k_j+1}) + \langle v^{k_j}, {x}^{k_j+1}  \rangle \,+ \,& \frac{1}{2\gamma_{k_j}} \|{x}^{k_j+1} - x^{k_j}\|^2 \leq \psi(u) + \langle v^{k_j}, u \rangle + \frac{1}{2\gamma_{k_j}} \|u - x^{k_j}\|^2.
\end{aligned}
\end{equation}
Now, using the above inequality with $u= \bar x$, we get that $\limsup_{j \to \infty} \psi(  x^{k_j+1}) \leq \psi(\bar x)$, and from the lower semicontinuity of $\psi$ we obtain $\psi(\bar x) \leq \liminf_{ j\to  \infty }   \psi(  x^{k_j +1}) $, which allows to conclude that $\lim_{j \to \infty} \psi( x^{k_j +1}) = \psi(\bar x)$. 
Note that Lemma~\ref{lemma_logn}\ref{lemma_logn-iii} particularly implies
\begin{align}\label{eqlimit}
	\vPhi_{k_j+1} ({x}^{k_j+1}) \to \vPhi (\bar x).
\end{align} In this manner, fixing $u \in \R^d$ in the right-hand side of \eqref{ineqProx}, letting $j \to \infty $ and recalling that $\gamma_{k_j} \to \bar \gamma $, we deduce that \begin{align*}
	\psi(\bar x )  
	\leq \psi(u) + \langle \bar v , u-\bar x \rangle + \frac{1}{2 \bar \gamma } \|u - \bar x\|^2,
\end{align*}
with the convention $\frac{1}{2 \bar \gamma } =0$ for $\bar \gamma =\infty$. Therefore, $\bar{x}$ is a stationary point of problem \eqref{ProblemP} in the sense of \eqref{def_stpoint}. This proves the first claim in item~\ref{item2}. For the second claim, let us show that $\vPhi(\bar x) = \lim_{ k \to \infty } \vPhi_k (x^{k})= \lim_{ j \to \infty  } \vPhi_{k_j} (x^{k_j})$. We already know from the beginning of this proof that $ \lim_{ k \to \infty } \vPhi_k (x^{k})  = \lim_{k\to\infty} \vPhi(x^k)$ exists,  and thus combined with \eqref{eqlimit} yields \[\vPhi (\bar x) = \lim_{ j \to \infty } \vPhi_{k_j+1} ({x}^{k_j+1}) = \lim_{ k \to \infty } \vPhi_{k+1}(x^{k+1}) = \lim_{ k \to \infty } \vPhi_k(x^k).  \]

To conclude the proof, we now proceed to show that 
items~\ref{item3} and~\ref{item4} hold.  First, by \eqref{inqalmostsurely}, it follows that $ \|x^{k+1} - x^k\| \to  0$ as $k \to \infty$ a.s. Hence, the sequence $(x^k)_{k\in\N}$ satisfies the  \emph{Ostrowski
condition}, which implies directly our last claim 
(see, e.g., \cite[Theorem~8.3.9 and Proposition~8.3.10]{MR1955649}).

\end{proofof}

\begin{remark} \label{r:The01}

A few comments about Theorem~\ref{The01} are in order.

\begin{enumerate}[label=(\roman*)]
    \item It is worth mentioning that if  one requires  that (a.s.)  $\| x^k \| \leq \rr  $ for all $k \in\N $ for some $\rr >0$, then it is not difficult to show that
	$\tau_{\rr} =+\infty$ a.s.,  for every $\rr >  \sup_{  k\in \N} \| x^k \|$. Therefore, taking expectation in \eqref{inequality_supermartingale} we get that  
	\begin{align*}
		\sum_{k \in \mathbb{N}} \mathbb{E}\|x^{k+1} - x^k\|^2 < +\infty.
	\end{align*}
	and 
     that  $ \bigl(\vPhi_k(x^k) \bigr)_{k \in \N } $ converges to an integrable random variable.

    \item \label{r:stabilize} Due to Theorem~\ref{The01}\ref{itema}, if we construct the sequence of stepsizes $(\gamma_k)_{k\in\N}$ to be nonincreasing (i.e. setting $\bar{\gamma}_k = \gamma_{k-1}$ for all $k\in\N$  in Step~4 of Algorithm~\ref{alg:1}), then this sequence needs to stabilize. Indeed, by way of  contradiction, suppose that the sequence of stepsizes does not stabilize. Then there are infinitely many iterations where the inner loop of the line search shrinks the stepsize at least a factor $\rho\in (0,1)$, contradicting the fact that $\inf_{k \in\N} \gamma_k >0$. In other words, there exists an index $\bar{k}$ such that at iteration $k \geq \bar{k}$ the line search procedure terminates upon the first trial.

\item A distinctive feature of Theorem~\ref{The01} is that it does not impose any prescribed growth rate on the sample sizes~\(n_k\); it is sufficient to assume that~\((n_k)_{k\in\N}\) is a nondecreasing sequence satisfying~\(n_k\to\infty\). To the best of our knowledge, this level of flexibility is novel. Many classical and recent frameworks in the literature mandate 
pre-scheduled growth conditions on the sample sizes to establish subsequential convergence; see, e.g.,~\cite{geiersbach2021stochastic,MR4486508,LeThiHuynhPhamDinhLuu2022,MR3935081,MR4902793}. Specifically, within the smooth, convex setting, the reduced-variable stochastic algorithm in~\cite{MR3935081} requires a sample-size sequence growing at a rate of at least~$\mathcal{O}(k^3\ln k)$. In nonconvex scenarios, the mini-batch sizes in~\cite{MR3439803} are assumed to scale at least linearly, and no trajectory-level convergence guarantees are obtained. Central limit theorems and asymptotic properties studied in~\cite{MR4902793} similarly rely on polynomial sample-size growth schedules of the form~$n_k^{-1} = \mathcal{O}(k^\nu)$ with~$\nu \ge 1$ to enforce noise stabilization. For nonconvex settings, even more demanding polynomial configurations are typically assumed; for instance, the stochastic difference-of-convex framework in~\cite{LeThiHuynhPhamDinhLuu2022} requires~$n_k^{-1} = \mathcal{O}(k^\nu)$ with a power of at least~$\nu \ge 2$, while related nonconvex schemes~\cite{MR4486508} obtain convergence guarantees with 
strictly increasing polynomial or geometric sample paths. In the infinite-dimensional Hilbert space setting, the convergence result in~\cite{geiersbach2021stochastic} needs the summability condition~$\sum_{k\in\N} 1/n_k < +\infty$.  


\end{enumerate}

\end{remark}
	
\section{Convergence under the \KLname condition} \label{s:PLK}

    From the previous section, we see that the PSS is not a descent method, although the descent-type condition \eqref{eq:descent} is satisfied. In light of this, we now turn our attention to the convergence analysis of the (random) sequence produced by Algorithm~\ref{alg:1} under the well-known \KLshort~property~\cite{MR1644089,lojasiewicz1965EnsemblesS,POLYAK1963864}. Firstly, let us recall this property. 
	
A function $\vPhi\colon \mathbb{R}^d \to \Rex$ is said to satisfy the \KLshort~property at $\bar{x}$ if there exist a constant $\eta>0$ and a continuous, concave function $\theta:[0,\eta]\to[0,+\infty[$, called the \emph{desingularizing function}, such that $\theta(0)=0$, $\theta$ is continuously differentiable on $]0,\eta[$ with $\theta'>0$, and the inequality 
\begin{equation}\label{PLK_cond}
	\theta'\big(\vPhi(x)-\vPhi(\bar{x})\big)\mathrm{dist}\!\big(0;\partial \vPhi(x)\big) \ge 1
\end{equation} holds for all $x\in\mathbb{B}_\eta(\bar{x})$ satisfying $\vPhi(\bar{x})<\vPhi(x)<\vPhi(\bar{x})+\eta$.  Here, $\mathrm{dist}(\,\cdot\,;C)$ denotes the distance to the set $C$. It is known  that the \KLshort~property holds at critical points of semi-algebraic and functions definable in o-minimal structures~\cite{bolte2007lojasiewicz,bolte2006nonsmooth,bolte2007clarke}. For functions defined in expected value form as the objective function in \eqref{ProblemP}, it is clear that the \KLshort~property holds when $\varphi_\xi$ is, for instance, an analytic function for all $\xi$, and $\psi$ is semi-algebraic.

Next, we impose an additional structural condition on the desingularizing function $\theta$ that has proven useful for establishing convergence rates; see, e.g.,~\cite{MR4605214,khanh2024fundamental}. Its validity for a broad class of desingularizing functions is discussed in detail in \cite[Remark G.1]{khanh2024fundamental}.

\begin{assumption}[Quasi-additivity type property]\label{assu_QAP}
	Let $\theta$ be the desingularizing function in the \KLshort~condition \eqref{PLK_cond} for which there exists a constant $C \in {(0,1]}$ such that for all $t_1, t_2 \in {(0,\eta)}$ with $t_1 + t_2 < \eta$, we have
	\begin{align}\label{des_g_inq}
		C \, [\theta'(t_1 + t_2)]^{-1}
		\;\leq\;
		[\theta'(t_1)]^{-1} + [\theta'(t_2)]^{-1}.
	\end{align}
 
\end{assumption}

\begin{remark}[Exponential \KLshort~condition]\label{remark_kindoftheta}
Consider the exponential version of the \KLshort~condition, where the desingularizing function is given by $\theta(t)=M t^{1-\beta}$ with $M$ a positive constant and $\beta \in {(0,1)}$. Since the map $t \mapsto t^\beta$ is concave on $[0,+\infty[$, one obtains the inequality
\begin{equation*}
(t_1+t_2)^\beta \le t_1^\beta + t_2^\beta
\qquad \forall\, t_1,t_2 \ge 0.
\end{equation*}
Therefore, the corresponding desingularizing function $\theta(t)=M t^{1-\beta}$ satisfies Assumption~\ref{assu_QAP} with constant $C=1$ (see the discussion after \cite[eq. (2.3)]{MR4605214}). We refer to \cite{li2018calculus} for a treatment of the exponential version of the \KLshort~property.
\end{remark}

In what follows, we say that a  normal integrand $\varphi$ is $C^{1,1}$ 
around a point $\bar x $ provided that $\varphi_\xi$ is differentiable for all $\xi \in \Xi$, $\nabla_x \varphi (\cdot, \bar{x})  \in L^2_\mu$,  and  there exist $\kappa \in L^2_{\mu}$ and $\epsilon>0$ such that
	\begin{align*} 
	 	\| \nabla_x \varphi(\xi,  y) -  \nabla_x \varphi(\xi,  z) \| \leq \kappa(\xi) \| y- z\|, \quad \forall \, y,z \in  \mathbb{B}_{\epsilon}[\bar x], \,  \forall \,\xi \in \Xi.
	\end{align*} Moreover, we say that  $\varphi$ is $C^{1,1}$ around a nonempty set $S$, provided that it is  $C^{1,1}$ around every $\bar x \in S$.  When $\varphi$ is differentiable at $x\in\mathbb{R}^d$,  it is clear that the gradient of its empirical average \eqref{defEk}  is given by
\begin{equation*}
    \nabla_x\Intf{\varphi}^{k}(x) = \frac{1}{n_k}\sum_{j=1}^{n_k}  \nabla_x\varphi_{\xi_j}(x).
\end{equation*}

\subsection{Global convergence of PSS}
The following analysis of the global convergence of the sequence generated by Algorithm~\ref{alg:1} extends the arguments introduced in the seminal paper~\cite{AttouchBolteSvaiter2013} to the stochastic framework considered here. We establish a global convergence result for Algorithm~\ref{alg:1} under suitable assumptions on the sample sizes, which allow us to control the approximation error induced by the use of empirical averages. Recall that, due to this error, our scheme does not belong to the class of descent methods, thereby preventing a direct application of the classical \KLshort~convergence framework; see, e.g.,~\cite{AttouchBolteSvaiter2013,bento2025convergence} and the references therein. We prove that, provided the approximation error vanishes at an appropriate rate, the use of sample averages does not alter the anticipated convergence properties of the generated sequence. More precisely, we obtain the following global convergence result.

 	\begin{theorem}[Global convergence]\label{Main_PKL_conv}
		Suppose that the assumptions of Theorem~\ref{The01} hold. Let $\mathcal{S}$  be the set of stationary points of $\vPhi$. Suppose   that  $\varphi$ is $C^{1,1}$ around $\mathcal
		{S}$.   Furthermore,  let us assume that the sequence of sample sizes $(n_k)_{ k \in \N}$ satisfies 
		\begin{align}\label{eq:n_k-series-condition}
			\sum_{k =1}^\infty \frac{\ln(n_k)}{ {\sqrt{n_k}}} < +\infty.
		\end{align}  
		Let $(x^k)_{k\in\N}$ be the sequence generated by Algorithm~\ref{alg:1}.  Then, there exists a measurable set $\hat{\Omega}$ with $\mathbb{P}(\Omega\setminus\hat{\Omega})=0$ such that, for every $\omega\in\hat{\Omega}$, whenever the sequence $\bigl(x^k(\omega)\bigr)_{k\in\mathbb{N}}$ is bounded and has an accumulation point $\bar{x}$ where the \KLshort~condition holds with a desingularizing function $\theta $ satisfying Assumption~\ref{assu_QAP}  such that \begin{align}\label{desing_condi}
			\sum_{k=p_0}^\infty \left(  \theta'\left( {\sum_{t=k}^\infty \frac{\ln(n_t)}{ {\sqrt{n_t} }}     } \right)\right)^{-1}<+\infty, \quad \text{for some } p_0 \in\N,
		\end{align}
       and that $\vPhi({x}^k) >\vPhi (\bar{x})$ for all  $k \geq p_0$, then  the whole sequence $\bigl(x^k(\omega)\bigr)_{k\in \N} $  converges  to $\bar x$ as $k\rightarrow\infty$. 
		\end{theorem}

We shall present the proof of Theorem~\ref{Main_PKL_conv} in Section~\ref{s:proof-Main_PKL_conv} below. For this purpose, we first establish a set of preliminary results in the following section.

\subsubsection{Technical results for Theorem~\ref{Main_PKL_conv}}   

Let us start formally stating the following result, which is a simple consequence of the compactness of the set $  
\mathcal{S}_{\rr}:=\{x\in\mathcal{S}:\|x\|\leq \rr\}
$.

 \begin{lemma}\label{Proposition4}
    Assume that $\varphi$ is $C^{1,1}$ around $\mathcal
{S}$. Then for each $\rr \in \N$,  there exist $\epsilon_{\rr}>0$, a function $\kappa_{\rr}\in L^2_{\mu}$, and finitely many points $
\bar{x}^{\rr}_{1},\ldots,\bar{x}^{\rr}_{m_{\rr}} \in \mathcal
{S}
$ such that
\[
\mathcal{S}_{\rr}
\subseteq
\bigcup_{i=1}^{m_{\rr}}
\mathbb{B}_{\epsilon_{\rr}/2}(\bar{x}^{\rr}_{i}),
\]
and, for every $\xi\in\Xi$, every $i=1,\ldots,m_{\rr}$, and all
$y,z\in \mathbb{B}_{\epsilon_{\rr}}[\bar{x}^{\rr}_{i}]$, one has
\begin{align}\label{eq_Lipsc_cont_grad22}
\big\|
\nabla_x\varphi(\xi,y)-\nabla_x\varphi(\xi,z)
\big\|
\leq
\kappa_{\rr}(\xi)\|y-z\|.
\end{align}
 \end{lemma}

  Let us recall that we already showed  sufficient decrease conditions, up to errors, in \eqref{eqclaim2} and \eqref{eq:descent}, that play the role of \cite[(H1)]{AttouchBolteSvaiter2013}. We next establish a relative error condition, akin to \cite[(H2)]{AttouchBolteSvaiter2013},  which, in combination with the sufficient decrease property, will lead to the convergence guarantees in Theorem~\ref{Main_PKL_conv}.

\begin{lemma}\label{Tlemma}
Under the assumptions of Theorem~\ref{The01}, let $(x^k)_{k\in\N}$ be the sequence generated by Algorithm~\ref{alg:1}. Let $\mathcal{S}$  be the set of stationary points of $\vPhi$. Assume that $\varphi$ is $C^{1,1}$ around $\mathcal
{S}$. Then, there exists a measurable set $\hat{\Omega}$ with $\mathbb{P}(\Omega \setminus \hat{\Omega}) = 0$, such that for every $\omega\in\hat{\Omega}$, if $\bar{x}$ is an accumulation point of $\bigl(x^k(\omega)\bigr)_{k\in\mathbb{N}}$, there exist constants $\rB,\zeta>0$   and an index $\hat{k}\in\mathbb{N}$ such that
\begin{equation}\label{eq:Tlemma}
\mathrm{dist}\!\big(0;\,\partial \vPhi_{k}({x}^{k+1})\big)
\;\le\;
\zeta\,\|x^{k+1}-x^{k}\|, \quad \text{for all } k\geq \hat{k} \text{ whenever } x^{k}\in\mathbb{B}_{\rB}[\bar{x}].
\end{equation}

\end{lemma}

\begin{proof}
Let $\hat{\Omega}$  be a measurable set of full probability such that the conclusions of Theorem~\ref{The01} hold. Furthermore,   by Assumption~\ref{Assumption00} and the strong law of large numbers 
we can assume (shrinking the set if necessary) that for all $\omega\in\hat{\Omega}$ and $\rr \in \N$, 
\begin{align}\label{eq_conv01}
\frac{1}{n_k}\sum_{j=1}^{n_k}   \kappa_\rr(\xi_j) \to \mathbb{E} \left[ \kappa_\rr \right], 
\end{align}
as $k\to+\infty$, where $ \kappa_\rr$ are the (square-integrable) Lipschitz functions given in Lemma~\ref{Proposition4}. Now, let $\omega \in \hat{\Omega}$ and  $\bar x$ be an accumulation point of $\bigl(x^k(\omega)\bigr)_{k\in\mathbb{N}}$. It follows by Theorem~\ref{The01}  that $\bar x \in \mathcal{S}$. Now, let us consider $\rr_0 \in \N$  such that $\|\bar x\| \leq \rr_0$. Due to \eqref{eq_Lipsc_cont_grad22} and \eqref{eq_conv01}, we can assume that there exist  $L_1>0$ and $\rB \in ( 0, \epsilon_{\rr}/2)$  such that $\Intf{\varphi}^{k}(\,\cdot\,)$ is continuously differentiable with $L_1$-Lipschitz gradient on $\mathbb{B}_{2\rB}(\bar{x})$ for every $k \in \N$.  
In view of \eqref{sumsquare_as}, let $\hat{k}\in\mathbb{N}$ be such that 
\[
\|x^{k+1}-x^{k}\|\le \rB,
\quad\text{for all } k\ge \hat{k}.
\]
Hence, for every $x^k\in\mathbb{B}_{\rB}(\bar{x})$ with $k\ge \hat{k}$, its next iterate $x^{k+1}$  belongs to $\mathbb{B}_{2\rB}(\bar{x})$. 
Using the optimality condition of the problem  in~\eqref{eq:algx}, we obtain \begin{equation*}
\frac{x^{k}-{x}^{k+1}}{\gamma_k} - \nabla_x\Intf{\varphi}^{k}   (x^{k})
\;\in\;
\partial \psi({x}^{k+1}).
\end{equation*} This in turn implies that
\[
d^{k}
:=
\frac{x^{k}-{x}^{k+1}}{\gamma_k}
\;+\;
\nabla_x  \Intf{\varphi}^{k} ({x}^{k+1})-  \nabla_x\Intf{\varphi}^{k} (x^{k}) \in \partial \vPhi_k({x}^{k+1}).
\] 
The $L_1$-Lipschitz continuity of $\nabla \Intf{\varphi}^{k} $ on $\mathbb{B}_{2\rB}(\bar{x})$ yields
\[
\|d^{k}\|
\;\le\;
\left(\frac{1}{\gamma_k}+L_1\right)\|x^{k}-{x}^{k+1}\|.
\]
Finally, by Theorem~\ref{The01}(i), there exists a sufficiently large constant $\zeta>0$ such that
\[ \mathrm{dist}\!\big(0;\,\partial \vPhi_k({x}^{k+1})\big)
\;\le\;
\|d^{k}\|
\;\le\;
\zeta\,\|x^{k+1}-x^{k}\|,
\]
which concludes the proof.
\end{proof}

We continue with a technical result that establishes the setting to apply the \KLshort~property.

\begin{lemma}\label{l-conv:tech-1} Suppose that the assumptions of Theorem~\ref{Main_PKL_conv} hold. Then, relative to a measurable set  $\hat{\Omega}$ with $\mathbb{P}(\Omega \setminus \hat{\Omega}) = 0$, the following assertions hold. Let $\bar{x}$ be an accumulation point of the sequence $\bigr(x^k(\omega)\bigl)_{k\in\N}$, for $\omega\in\hat{\Omega}$, where the \KLshort~condition \eqref{PLK_cond} is satisfied for a given $\eta>0$. There exist $\rC,\rD>0$ such that $ \Intf{\varphi}^k$, $ \Intf{\varphi}$ are $C^{1,1}$ on $\mathbb{B}_{\rC}[\bar{x}]$, and $\mathbb{B}_\eta[\bar{x}] \subset \mathbb{B}_{\rD}[0]$. Moreover,  for all $k\in\mathbb{N}$, define the (random) sequences 
		\begin{align}
			a_k &:= 2 \sup \left\{ 
			\left| \Intf{\varphi}(x) - \Intf{\varphi}^k(x) \right| : {\|x\| \leq \rD }   \right\}, \label{eq:ak} \\
		u_k & := \sum_{t=k }^\infty a_t, \label{eq:uk}
		\end{align} and
		\begin{equation} \label{eq:bk}
        b_k := \sup\left\{  
			\left\| \nabla_x\Intf{\varphi}(x) 
			- \nabla_x\Intf{\varphi}^{k}(x) \right\| :   x \in  \mathbb{B}_{\rC}[\bar{x}]      \right\}.
		\end{equation} Then, for any $\varepsilon\in{(0,\min\{\rB,\eta/2,5\rC/6\})}$, there exists $k_0 \in \N$ such that for all $k \geq k_0$, the following conditions hold:\begin{align}
            &\| {x}^{k+1}-{x}^{k}\|\le\varepsilon/5; \label{P1}\tag{P1}\\
            &\vPhi( {x}^{k}) -\vPhi(\bar{x})+u_k  < \eta;  \label{P2}\tag{P2}\\
            &\vPhi(\bar{x})<\vPhi( {x}^{k})<\vPhi(\bar{x})+\eta; \label{P3}\tag{P3}\\
            &\max\left\{ 	\| {x}^{k_0}-\bar{x}\| ,	\frac{s_{k_0+1}}{\mathfrak{C}     },   \frac{1}{\zeta} \sum_{k=k_0}^{\infty} \left(    b_k +  \left[\theta'(u_{k+1} )\right]^{-1} \right)\right\}  \leq \epsilon/5;\label{P4}\tag{P4}
        \end{align}where  	  $ \mathfrak{C}   := \zeta^{-1} C \sigma  $, $C$ is the constant in the quasi-additive property  \eqref{des_g_inq}, and the constants $\rB$ and $\zeta$ are such that they satisfy~\eqref{eq:Tlemma}, and
        \begin{align*}
			s_k := \theta\big( \vPhi(x^k) - \vPhi(\bar{x}) + u_k \big) \geq 0.
		\end{align*}
\end{lemma}

\begin{proof}
Let $\hat{\Omega}$ be a measurable set with $\mathbb{P}(\Omega \setminus \hat{\Omega}) = 0$, on which the conclusions of
Theorem~\ref{The01} and Lemma~\ref{Tlemma} hold. Moreover, applying
Lemma~\ref{lemmarateconver} with $F=\varphi$, we obtain that, for every
$\rr>0$,
\begin{align}\label{eq01PLK0}
a^{\rr}_k
:=
\sup_{\|x\|\leq \rr}
\left|
\Intf{\varphi}(x)-\Intf{\varphi}^k(x)
\right|
=
\mathcal{O}\left(
\frac{\ln(n_k)}{\sqrt{n_k}}
\right)
\quad \text{a.s.}
\end{align} Furthermore, let $\rr\in\mathbb{N}$, and consider the points
$\bar{x}^{\rr}_i$, $i=1,\ldots,m_{\rr}$, together with the function
$\kappa_{\rr}$, provided by Lemma~\ref{Proposition4}. Applying
Lemma~\ref{lemmarateconver} on each ball
$\mathbb{B}_{\epsilon_{\rr}}[\bar{x}^{\rr}_i]$, we obtain that, for every
$\rr\in\mathbb{N}$ and every $i=1,\ldots,m_{\rr}$,
\begin{align}\label{eq01PLK00}
\sup_{x\in \mathbb{B}_{\epsilon_{\rr}}[\bar{x}^{\rr}_i]}
\left\|
\nabla_x\Intf{\varphi}(x)
-
\nabla_x\Intf{\varphi}^{k}(x)
\right\|
=
\mathcal{O}\left(
\frac{\ln(n_k)}{\sqrt{n_k}}
\right)
\quad \text{a.s.}
\end{align}
Thus, after possibly replacing $\hat{\Omega}$ by another measurable subset of full
probability, we may assume that, in addition, both estimates \eqref{eq01PLK0} and
\eqref{eq01PLK00} hold for every $\omega\in\hat{\Omega}$.

Fix now $\omega\in\hat{\Omega}$, and let $\bar{x}$ be an accumulation point of
the sequence $(x^k(\omega))_{k\in\mathbb{N}}$. Choose $\rD>0$ such that
$\mathbb{B}_{\eta}[\bar{x}]\subseteq \mathbb{B}_{\rD}[0]$. By \eqref{eq01PLK0}, it directly follows that,
\begin{align}\label{eq01PLK}
a_k=
a^{\rD}_k
=
\mathcal{O}\left(
\frac{\ln(n_k)}{\sqrt{n_k}}
\right).
\end{align} Next, take $\rr\in\mathbb{N}$ such that $\|\bar{x}\|\leq \rr$. Since
$\mathcal{S}_{\rr}\subseteq \bigcup_{i=1}^{m_{\rr}}
\mathbb{B}_{\epsilon_{\rr}/2}(\bar{x}^{\rr}_i)$, there exists
$i\in\{1,\ldots,m_{\rr}\}$ such that $\bar{x}\in \mathbb{B}_{\epsilon_{\rr}/2}(\bar{x}^{\rr}_i)$. Consequently, for every $\rC\in(0,\epsilon_{\rr}/2)$ sufficiently small, we have $\mathbb{B}_{\rC}[\bar{x}]
\subseteq
\mathbb{B}_{\epsilon_{\rr}}[\bar{x}^{\rr}_i]$, with
$\Intf{\varphi}^k$ and $\Intf{\varphi}$ being $C^{1,1}$ on
$\mathbb{B}_{\rC}[\bar{x}]$.
Therefore, by \eqref{eq01PLK00},
\begin{align}\label{conv:gradients}
b_k
=
\sup_{x\in \mathbb{B}_{\rC}[\bar{x}]}
\left\|
\nabla_x\Intf{\varphi}(x)
-
\nabla_x\Intf{\varphi}^{k}(x)
\right\|  \leq
\sup_{x\in \mathbb{B}_{\epsilon_{\rr}}[\bar{x}^{\rr}_i]}
\left\|
\nabla_x\Intf{\varphi}(x)
-
\nabla_x\Intf{\varphi}^{k}(x)
\right\| =
\mathcal{O}\left(
\frac{\ln(n_k)}{\sqrt{n_k}}
\right).
\end{align} Moreover, by \eqref{conv:gradients} and our sampling-size assumption \eqref{eq:n_k-series-condition}, we have $\sum_{k \in \N} b_k < +\infty$. Using  \eqref{eq01PLK} and \eqref{eq:n_k-series-condition}, there are $m,\hat{p} \in \mathbb{N}{\setminus}\{0\}$ such that
		\begin{align*}
			a_k \leq  m \cdot  \frac{\ln(n_k)}{\sqrt{n_k} }  \quad \mbox{and} \quad \sum_{t=k}^{\infty} \frac{\ln(n_t)}{\sqrt{n_t} }  < \frac{\eta}{2 m},  
			\quad \text{for all } k \geq \hat{p},
		\end{align*} so that $(u_k)_{k\in\N}$ is well-defined, and
		\begin{align} \label{eq:uk-bound}
			u_k  \leq  m \cdot \sum_{t=k}^{\infty} \frac{\ln(n_t)}{\sqrt{n_t} }  <  \frac{\eta}{2},
			\quad \text{for all } k \geq \hat{p}.
		\end{align} By using the quasi-additivity property of the desingularizing function $\theta$ followed by the fact that $\theta'$ is
nonincreasing (as $\theta$ is concave),  a simple induction argument shows that
for every $m\in\mathbb{N}{\setminus}\{0\}$ and every $t>0$ satisfying $mt\in(0,\eta)$, one has
\begin{align*}
    \left[\theta'(mt)\right]^{-1}
    \leq
    \left(\frac{2}{C}\right)^{m-1}
    \left[\theta'(t)\right]^{-1},
\end{align*}
from where we obtain for all $k\geq \hat{p}$,
\begin{equation*}
\begin{aligned} 
    \big[\theta'(u_k)\big]^{-1}
    &\leq
    \left[
        \theta'\!\left(
            m \sum_{t=k}^{\infty}
            \frac{\ln(n_t)}{\sqrt{n_t}}
        \right)
    \right]^{-1}  \leq
    \left(\frac{2}{C}\right)^{m-1}
    \left[
        \theta'\!\left(
            \sum_{t=k}^{\infty}
            \frac{\ln(n_t)}{\sqrt{n_t}}
        \right)
    \right]^{-1}.
\end{aligned}\end{equation*}
		Consequently, by~\eqref{desing_condi}, we have
		\begin{align}\label{eq:theta(u)-estimate} 
			\sum_{k=\hat{p}}^{\infty} \big[\theta'(u_k)\big]^{-1} < +\infty.
		\end{align}
        
        Altogether, the above implies that we can take an $\varepsilon\in{(0,\min\{\rB,\eta/2,5\rC/6\})}$ and choose $k_0\ge\max\{ \hat{k},\hat{p}\}$, where $\hat{k}$ is from Lemma~\ref{Tlemma}, so that for all $k\geq k_0$, \eqref{P1} follows from \eqref{sumsquare_as}, while \eqref{P2} and \eqref{P3} follow from Theorem~\ref{The01}\ref{item2} and \eqref{eq:uk-bound}. In particular, in view of \eqref{P2}, the sequence $(s_k)_{k\in\N}$ is well-defined. Regarding \eqref{P4}, if necessary, increase $k_0$ so that  $\| {x}^{k_0}-\bar{x}\| \leq \varepsilon/5$ and $s_{k_0+1}  \leq \varepsilon\mathfrak{C}/5$, where the latter is possible due to the continuity of $\theta^\prime$. The estimate of the remaining term follows from~\eqref{eq:n_k-series-condition}, \eqref{conv:gradients} and \eqref{eq:theta(u)-estimate}, yielding \eqref{P4}.
 \end{proof}

  The remainder of the analysis uses Lemma~\ref{l-conv:tech-1} to bound the distance from $(x^k)_{k\in\N}$ to the accumulation point $\bar{x}$. First, we bound the distance between two consecutive iterates $x^{k+1}$ and $x^{k+2}$, whenever $x^k$ is close enough to $\bar{x}$. 

  \begin{lemma} \label{l-conv:tech-2} 
      Under the assumptions of Lemma~\ref{l-conv:tech-1}, for every $k\ge k_0$ with $ {x}^{k}\in\mathbb{B}_\varepsilon[\bar{x}]$ it holds
		\begin{equation*}
			\| x^{k+2} - x^{k+1} \| \leq \frac{ \left( 	s_{k+1}-s_{k+2} \right) }{2  \mathfrak{C}  }  +\frac{    b_{k}  }{ 2 \zeta}  + \frac{ 1   }{ 2 }  \| x^{k+1} - x^{k} \|     + \frac{ 1}{ 2\zeta }  \left[ \theta'(u_{k+1} )\right]^{-1} . 
		\end{equation*}
\end{lemma}

\begin{proof}
  Let us consider $ {x}^{k}\in\mathbb{B}_\varepsilon[\bar{x}]$. As in  Lemma~\ref{l-conv:tech-1}, let $\rD>0$ such that $\mathbb{B}_\eta[\bar{x}] \subset \mathbb{B}_{\rD}[0]$. It follows from \eqref{P1} and the triangle inequality that  $x^{k+1} \in \mathbb{B}_{\rC}[\bar{x}]$ and $x^{k+1}, {x}^{k+2} \in \mathbb{B}_{\eta}[\bar{x}] \subset \mathbb{B}_{\rD}[0]$. Then, by adding and subtracting  both $\Intf{\varphi}(x^{k+1})$ and  $\Intf{\varphi}(x^{k+2})$ in the right-hand side of~\eqref{01ineq_est}, we get
		\begin{align}\label{01ineq_est_2}
			\sigma \| x^{k+2} - x^{k+1} \|^2 &\leq \vPhi(x^{k+1}) + u_{k+1}   - \left(  \vPhi(x^{k+2})  +  u_{k+2} 	\right) .  
		\end{align}
		Now, in view of \eqref{P2}, we can apply the concavity of $\theta$ and then \eqref{01ineq_est_2} yielding
		\begin{align*}
			s_{k+1}-s_{k+2}
			&\ge \theta'(\vPhi({x}^{k+1})-\vPhi(\bar x)+ u_{k+1})\bigl(\vPhi({x}^{k+1})+u_{k+1}-\vPhi({x}^{k+2})-u_{k+2}\bigr)\\
			&\ge \sigma \theta'(\vPhi(x^{k+1})-\vPhi(\bar x)+{u}_{k+1})   \| x^{k+2} - x^{k+1} \|^2.
		\end{align*}
		Using the quasi-additive property \eqref{des_g_inq}, \eqref{P3} and the \KLshort~property  \eqref{PLK_cond}, we get that
		\begin{equation}\label{eq:skq-sk20}
        		\begin{aligned}
			s_{k+1}-s_{k+2}
			&\geq \frac{C \sigma }{ \left[ \theta'( \vPhi({x}^{k+1})-\vPhi(\bar x)  )\right]^{-1} +  \left[ \theta'(u_{k+1}  )\right]^{-1}  }     \| x^{k+2} - x^{k+1} \|^2\\
			& \geq \frac{C \sigma}{  \dist ( 0 ; \partial\vPhi(  x^{k+1} )) +  \left[ \theta'(u_{k+1}  )\right]^{-1}  }      \| x^{k+2} - x^{k+1} \|^2.
		\end{aligned}
        \end{equation} From the triangle inequality, \eqref{eq:bk} and the fact that 
          we also have 
        \begin{equation} \label{eq:pre-KL-estimate}
            \begin{aligned}
                \dist ( 0 ; \partial\vPhi(  x^{k+1} ))  & \leq \dist ( 0 ; \partial\vPhi_{k}(  x^{k+1} )) + \dist ( \partial\vPhi(  x^{k+1} ) ; \partial\vPhi_{k}(  x^{k+1} )) \\
            & \leq \dist ( 0 ; \partial\vPhi_{k}(  x^{k+1} )) + \left\| \nabla\Intf{\varphi}(x^{k+1}) 
			- \nabla\Intf{\varphi}^{k}(x^{k+1}) \right\| \\
            & \leq \dist ( 0 ; \partial\vPhi_{k}(  x^{k+1} )) + b_{k}.
            \end{aligned}
        \end{equation}
        Since $\varepsilon <\eta < \rB$, then $x^{k} \in \mathbb{B}_{\varepsilon}[\bar{x}] \subseteq \mathbb{B}_{\rB}[\bar{x}]$, and thus the estimate in Lemma~\ref{Tlemma} holds. Therefore, combining~\eqref{eq:skq-sk20} and~\eqref{eq:pre-KL-estimate} yields
		\begin{align*}
			\left( 	s_{k+1}-s_{k+2} \right)  \left(    \| x^{k+1} - x^{k} \|  + \zeta^{-1}b_{k}   +  \zeta^{-1} \left[ \theta'(u_{k+1} )\right]^{-1}        \right) 	\geq &  \mathfrak{C}   \| x^{k+2} - x^{k+1} \|^2.
		\end{align*}
		Then it follows that 
        {
		\begin{align*}
			\mathfrak{C}   \| x^{k+2} - x^{k+1} \|
			 & \leq  \sqrt{   \mathfrak{C}              } \sqrt{      \left( 	s_{k+1}-s_{k+2} \right)  \left(       \| x^{k+1} - x^{k} \|   +\zeta^{-1} b_{k}   + \zeta^{-1} \left[ \theta'(u_{k+1} )\right]^{-1}        \right)        } \\
			& \leq \frac{ \left( 	s_{k+1}-s_{k+2} \right) }{2 }   +\frac{\mathfrak{C}        }{ 2 }  \| x^{k+1} - x^{k} \| +\frac{ \mathfrak{C}          }{ 2 }   \zeta^{-1}b_{k}    + \frac{ \mathfrak{C}          }{ 2 } \zeta^{-1}  \left[ \theta'(u_{k+1} )\right]^{-1},  
		\end{align*}}%
		which proves the claim.
\end{proof}

Finally, we bound the partial sums of the tail of the series defined by the step length of the sequence $(x_k)_{k \in \N}$, from where we will obtain global convergence in Section~\ref{s:proof-Main_PKL_conv}.

\begin{lemma} \label{l-conv:tech-lemma-2}
   Under the assumptions of Lemma~\ref{l-conv:tech-1}, let $k_1 > k_0$. If $x^k\in \mathbb{B}_\varepsilon[\bar{x}]$ for all $k =k_0, \dots, k_1$, then  
   \begin{equation}\label{eq:KL-tele}\sum_{k=k_0}^{k_1}\|{x}^{k+1}-{x}^{k}\|
			 \leq 2\|{x}^{k_0+1}-{x}^{k_0}\|+\frac{s_{k_0+1}}{\mathfrak{C}     }     +\frac{1}{\zeta}  \sum_{j=k_0+1}^{\infty} \left(    b_{j-1} +  \left[\theta'(u_{j} )\right]^{-1} \right)  
            .\end{equation}
\end{lemma}

\begin{proof}
 Resorting to  Lemma~\ref{l-conv:tech-2}, 
 by recursive substitution we get for all $k=k_0, \dots, k_1$,\begin{equation}\label{Inq_Claim2_ind}
			\begin{aligned}
				\|{x}^{k+2}-{x}^{k+1}\| &\le       \frac{1}{2^{k-k_0+1}}   \| x^{k_0+1} - x^{k_0}\|  +\frac{1}{  \mathfrak{C}} \sum_{i=0}^{ k -k_0} \left(  \frac{1}{2}\right)^{i+1} \left(  s_{k+1-i} -s_{k+2-i}  \right) \\ & \quad +   \zeta^{-1} \sum_{ i=0}^{ k-k_0}  \left(  \frac{1}{2}\right)^{i+1} \left( b_{k-i} + \left[      \theta'(u_{k+1-i} )\right]^{-1}\right).
			\end{aligned}
        \end{equation}
		By adding \eqref{Inq_Claim2_ind} for $k=k_0,\dots,k_1-1$, 
\begin{equation*}
\begin{aligned}
			\sum_{k=k_0}^{k_1 }\|{x}^{k+1}-{x}^{k}\|
			&=\|{x}^{k_0+1}-{x}^{k_0}\|+\sum_{k=k_0}^{k_1-1}\|{x}^{k+2}-{x}^{k+1}\|\\
			&\le 2\|{x}^{k_0+1}-{x}^{k_0}\|+\frac{1}{\mathfrak{C}     }\sum_{k=k_0}^{k_1-1}  
			\sum_{i=0}^{ k-k_0} \left(  \frac{1}{2}\right)^{i+1} \left(  s_{k+1-i} -s_{k+2-i}  \right) \\
			& \quad + \zeta^{-1} \sum_{k=k_0}^{k_1-1}   \sum_{ i=0}^{ k-k_0}  \left(  \frac{1}{2}\right)^{i+1}   \left( b_{k-i} + \left[ \theta'(u_{k+1-i} )\right]^{-1} \right)
			\\
			&=2\|{x}^{k_0+1}-{x}^{k_0}\|+\frac{1}{\mathfrak{C}     }\sum_{j=0}^{k_1-k_0-1}\frac{1}{2^{j+1}}\sum_{k=k_0+1}^{k_1-j}(s_k-s_{k+1}) \\&\quad +  \zeta^{-1}\sum_{j=0}^{k_1-k_0-1}\frac{1}{2^{j+1}}\sum_{k=k_0+1}^{k_1-j}\left( b_{k-1} +  \left[ \theta'(u_{k} )\right]^{-1} \right)  \\
			&\le 2\|{x}^{k_0+1}-{x}^{k_0}\|+\frac{1}{\mathfrak{C}     } s_{k_0+1} +   \frac{1}{\zeta} \sum_{k=k_0+1}^{\infty} \left(    b_{k-1} +  \left[\theta'(u_{k} )\right]^{-1} \right),
            \end{aligned}
		\end{equation*} where to get the second inequality we apply the telescopic sum identity, the fact that $s_{k_1-j+1}\geq0$ for all $j \leq k_1-k_0-1$, and $\sum_{j\ge1}2^{-j}=1$. This concludes the proof.

\end{proof}

\subsubsection{Proof of Theorem~\ref{Main_PKL_conv}} \label{s:proof-Main_PKL_conv}

We are now in a position to prove the main result of this section.

\medskip

\begin{proofof}{Main_PKL_conv}
It suffices to show that for all $\varepsilon >0$ defined in Lemma~\ref{l-conv:tech-1}, it holds that $x^k\in \mathbb{B}_\varepsilon[\bar{x}]$ for all $k \geq k_0$. We argue by induction. 
By assumption, the statement holds for $k=k_0$, implied by \eqref{P4}, as well as for $k=k_0+1$, since due to \eqref{P1} and \eqref{P4}, \[\|x^{k_0+1}-\bar{x}\| \leq  \|x^{k_0+1}-x^{k_0}\| + \|x^{k_0}-\bar{x}\| \leq \frac{2\epsilon}{5}.\] To proceed with our induction step, we assume that the claim holds for all $k\in\{k_0,\ldots,k_1\}$, with $k_1>k_0$. Let us prove that $x^{k_1+1}\in \mathbb{B}_\varepsilon[\bar{x}]$. Indeed, in view of Lemma~\ref{l-conv:tech-lemma-2}, \eqref{P1} and \eqref{P4}, \begin{equation*}
    \sum_{k=k_0}^{k_1}\|{x}^{k+1}-{x}^{k}\| \leq  \frac{2\varepsilon}{5} + \frac{\varepsilon}{5} + \frac{\varepsilon}{5} = \frac{4\varepsilon}{5}.
\end{equation*} 
 Calling \eqref{P4} once again and the triangle inequality, we deduce
        \[
        \|x^{k_1+1} - \bar{x}\| \leq \|x^{k_0} - \bar{x}\| + \sum_{k=k_0}^{k_1}\|x^{k+1}-x^k\| \leq \varepsilon,
        \] from where we deduce the desired result. 
        \end{proofof}

    \begin{remark}\label{r:Th4.2}

    We next comment on Theorem~\ref{Main_PKL_conv}.

    \begin{enumerate}[label=(\roman*)] 
        \item  The assumption $\vPhi({x}^k) >\vPhi (\bar{x})$ for all large enough $k \in \N$ in Theorem~\ref{Main_PKL_conv} holds trivially whenever $\bar{x}$ is a local minimizer (unless $x^k$ is a local minimizer itself). Alternatively, a slightly stronger variant of the \KLshort~condition \cite[Definition 2.1]{MR4605214} can be used, namely, requiring \begin{equation*}
	\theta'\big(|\vPhi(x)-\vPhi(\bar{x})|\big)\mathrm{dist}\!\big(0;\partial \vPhi(x)\big) \ge 1
\end{equation*} whenever \[
x\in\mathbb{B}_\eta(\bar{x}) \mbox{~and~}0<|\vPhi(x) - \vPhi(\bar{x})|<\eta.
\] Instead, we assume the condition $\vPhi({x}^k) >\vPhi (\bar{x})$ for ease of presentation.

\item \label{r:Th4.2-ii} The proof of Theorem~\ref{Main_PKL_conv} uses the bound in Lemma~\ref{Tlemma}, which is expected to hold for methods of \emph{implicit} nature, such as proximal-type methods (see, e.g. \cite{AttouchBolteSvaiter2013,frankel2015splitting,bento2025convergence}).  Methods of \emph{explicit} nature, such as the gradient method, are expected to satisfy Lemma~\ref{Tlemma} by evaluating the left-hand side of the estimate therein at $x^k$, namely, \begin{equation} \label{eq:remark}
    \mathrm{dist}\!\big(0;\,\partial \vPhi_{k}({x}^k)\big)
\;\le\;
\zeta\,\|x^{k+1}-x^{k}\|.
\end{equation} For instance, if $\psi \equiv 0$ and $\varphi$ is smooth in \eqref{ProblemP}, then the update \eqref{eq:algx} in Algorithm~\ref{alg:1}, after passing the line search test, becomes \[x^{k+1} = x^k - \gamma_k v^k, \mbox{~where~} v^k = \frac{1}{n_k}\sum_{i=1}^{n_k} \nabla \varphi_{\xi_i}(x^k) = \nabla_x \Intf{\varphi}^{k}(x^k).\] Then \eqref{eq:remark} holds for $\zeta = \left(\inf_{k\in\N}\gamma_k\right)^{-1}$, which is strictly positive a.s. due to Theorem~\ref{The01}\ref{itema}. Therefore, similar convergence results to Theorem~\ref{Main_PKL_conv}, 
\emph{mutatis mutandis}, follow for explicit methods satisfying a sufficient descent condition up to errors (Lemma~\ref{l-subseq:tech-2}), for nonconvex stochastic optimization problems. We refer to \cite{atenas2023unified} for a unified treatment of implicit and explicit methods of descent in the nonconvex deterministic setting.
    \end{enumerate}
\end{remark}

\subsection{Local convergence rates of PSS}\label{s:local-rates}

In this section we describe the rate of convergence of Algorithm~\ref{alg:1} when the \KLshort~condition is satisfied in the exponential case.  We first state the main theorem.  Observe that in order to get  convergence rates, on top of the \KLshort~condition, we assume that the sample sizes grow at least at a polynomial rate. 

    \begin{theorem}[Convergence rates] \label{th:rates} In the setting of Theorem~\ref{Main_PKL_conv}, let $\bar x$ be the limit of the sequence $(x^k)_{k\in\N}$, and suppose that the 
\KLshort~condition holds at $\bar{x}$ with 
desingularizing function
\( 
\theta(t)=Mt^{1-\beta},
\)
for some $M>0$ and $\beta\in(0,1)$.
    Furthermore,  let the sequence of sample sizes $(n_k)_{ k \in \N}$ satisfy
      \begin{equation} \label{eq:condition_series_conv}
		    n_k^{-1} = \mathcal{O}(k^{-\expo}) \mbox{, where~}  \expo \in \left( \frac{2 + 2 \beta}{\beta} , +\infty\right).
		\end{equation} 
        Then,  the following convergence rates hold. 

        \begin{enumerate}[label=(\roman*)] 
            \item If $\beta \in \bigl(0,\frac{1}{2}\bigr]$, then \[\vPhi(x^k) - \vPhi(\bar{x}) = \mathcal{O} \left(    \frac{\ln(k)}{k^{\expo/2-1}} \right) \mbox{, and~}\|x^k - \bar{x}\| =\mathcal{O}\left(\frac{\ln(k)}{k^{1/2(\expo/2-1)-1}}\right).\]
            \item If $\beta \in \bigl(\frac{1}{2},1\bigr)$, then 
            \begin{align*}
                \vPhi(x^k) - \vPhi(\bar{x}) &=  \left\{\begin{array}{ll}
                \mathcal{O}\bigl(k^{-2\beta(\expo/2-1)+1}\ln(k)\bigr), & \text{~if } \tfrac{2+2\beta}{\beta} < \expo \leq \tfrac{4\beta}{2\beta-1},\\ \mathcal{O}\bigl(k^{-2\beta/(2\beta-1)+1}\ln(k)\bigr), & \text{~if~}  \expo > \tfrac{4\beta}{2\beta-1},
                \end{array}\right.
                \end{align*}
                and
                \begin{align*}
                \|x^k - \bar{x}\| &= \left\{\begin{array}{ll}
                \mathcal{O}\left(k^{-\beta(\expo/2-1)+1} \ln(k)\right), & \text{~if } \tfrac{2+2\beta}{\beta} < \expo \leq \tfrac{4\beta}{2\beta-1},\\ \mathcal{O}\bigl(k^{-(1-\beta)/(2\beta-1)}\ln^{1-\beta}(k)\bigr), & \text{~if }  \expo > \tfrac{4\beta}{2\beta-1}.
                \end{array}\right. \end{align*}
        \end{enumerate}
    
    \end{theorem}

Figure~\ref{fig:rates_fill} shows a heat map that illustrates the convergence rates obtained in Theorem~\ref{th:rates}. The proof of Theorem~\ref{th:rates} will be developed in Section~\ref{s:proof-rates} . We first present in the following section some results instrumental to it.

\begin{figure}[ht!]
\begin{center}
\includegraphics[width=\textwidth]{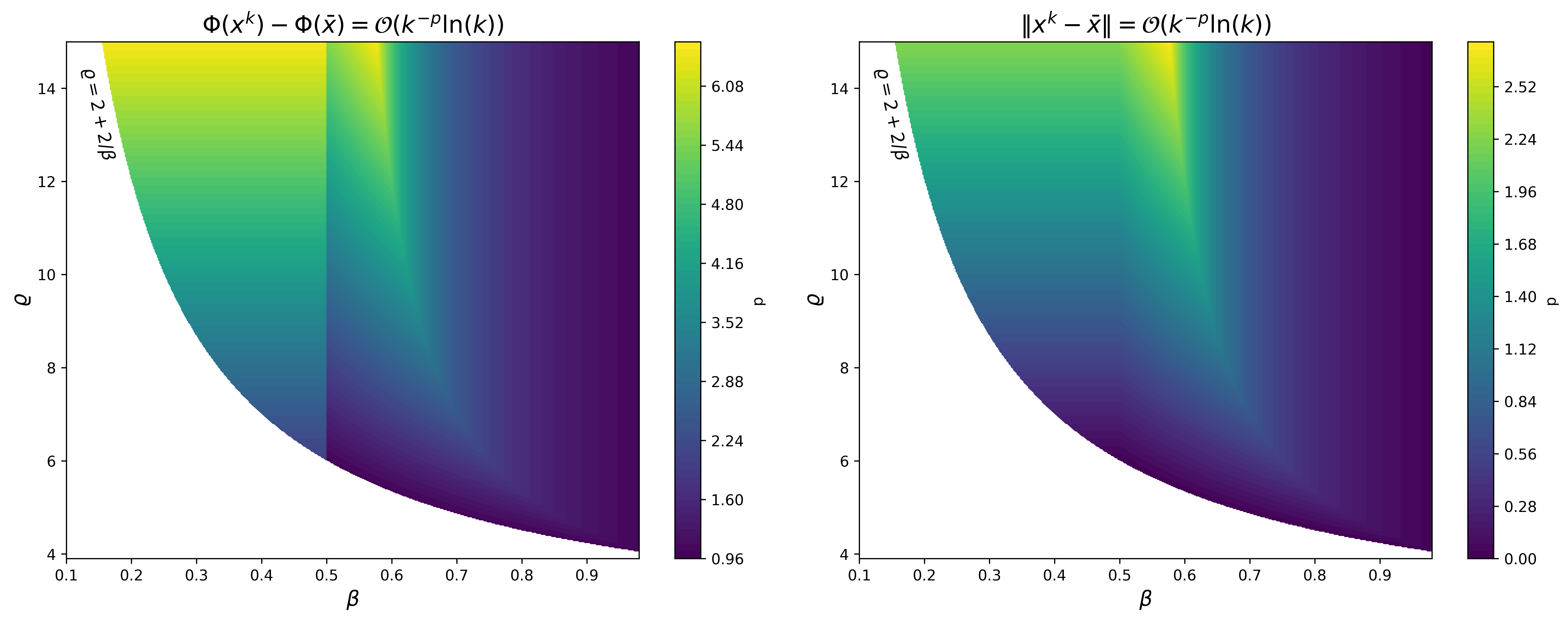}
\end{center}
\caption{Heat map of the exponent $p$ in the convergence rate $\mathcal{O}\bigl(k^{-p}\ln(k)\bigr)$ obtained in Theorem~\ref{th:rates} for the function values (left) and the iterates (right), displayed as a function of the \KLshort~exponent $\beta\in{(0,1})$ and the sample-size polynomial growth exponent $\expo > 2 + 2/\beta$. }\label{fig:rates_fill}
\end{figure}

\subsubsection{Preliminaries of Theorem~\ref{th:rates}}

Before delving into the analysis, we note that  a sequence of sample sizes $(n_k)_{k\in\N}$ verifying \eqref{eq:condition_series_conv}, satisfies  conditions \eqref{eq:n_k-series-condition} and \eqref{desing_condi} required in Theorem~\ref{Main_PKL_conv}. We defer the proof of this claim to Section~\ref{proof-l-rates:tech-result-1} in the Appendix, as it is merely a technicality. 

\begin{lemma} \label{B-ii}
Let $\beta \in (0,1)$, $\expo >0$ and $(n_k)_{k\in\N}$ such that \eqref{eq:condition_series_conv} holds. Then  $$\left[ \theta'\left( u_{k+1}\right)\right]^{-1} = \mathcal{O}\left(\frac{\ln^\beta(k)}{k^{\beta(\expo/2-1)}}\right), \quad b_k = \mathcal{O}\left(\frac{\ln(k)}{k^{\expo/2}}\right)  \mbox{~and~} \sum_{j=k}^\infty(b_{j} + \theta^\prime(u_{j+1})^{-1}) =  \mathcal{O}\left( \frac{\ln(k)}{k^{\beta(\expo/2-1)-1}}  \right)  .$$ Furthermore, 
 \eqref{eq:n_k-series-condition} and \eqref{desing_condi} hold. 
\end{lemma}

Capitalizing on Lemma~\ref{lemma:sequences}, we proceed to deduce   convergence rates for Algorithm~\ref{alg:1} for both the function values and the iterates. The arguments adapt those from the classical monotone setting to our nonmonotone case.
    For a different nonmonotone method, the authors in \cite[Theorem 3.10]{MR4605214} obtain similar convergence  rates for the iterates alone. Observe that the rate at which the approximation error vanishes, which was given by Lemma~\ref{lemmarateconver}, dominates over any possible faster linear rates, such as the classical linear \cite[Theorem 2]{AttouchBolte2009} or superlinear \cite[Theorem 2]{bento2025convergence}.

In order to obtain the convergence rates in Theorem~\ref{th:rates}, we first show a recursion 
for the sequence of function values that complies with Lemma~\ref{lemma:sequences}. 

\begin{lemma} \label{p-rates:tech-result-2}
    Suppose that the assumptions of Theorem~\ref{th:rates} are satisfied. Set $\alpha_k := \vPhi(x^k) - \vPhi(\bar{x}) + u_k \geq 0$ for all $k \geq k_0$, where $k_0$ is defined in Lemma~\ref{l-conv:tech-1}. Then, 
    there exist constants  $\hat{C}_1,\hat{C}_2 >0$ such that
\begin{equation*}
            \alpha_{k+1}^{2\beta} \leq \hat{C}_1(\alpha_k - \alpha_{k+1}) + \hat{C}_2\frac{\ln^{2\beta}(k)}{k^{2\beta(\expo/2-1)}},
    \end{equation*}
for all sufficiently large $k$.
\end{lemma}

\begin{proof}
     From \eqref{P2} and \eqref{P3}, we know that $\alpha_k, u_k, \vPhi( {x}^{k}) - \vPhi(\bar{x}) $ belong to $(0,\eta)$ for all $k \geq k_0$, where $\eta >0$ comes from the neighborhood of the \KLshort~condition.  Furthermore, taking $\rB>0$ for which  Lemma~\ref{Tlemma} holds and $\rC>0$ such that $\varphi$ is $C^{1,1}$ on $\mathbb{B}_{\rC}[\bar{x}]$ (as done in  Lemma~\ref{l-conv:tech-1}), Theorem~\ref{Main_PKL_conv} implies that for all $\hat{\epsilon} \in (0, \min\{\eta, \rC, \rB\})$, increasing $k_0$ if necessary, for all $k \geq k_0$, $x^k \in \mathbb{B}_{\hat{\epsilon}}[\bar{x}]$. Then
     \begin{align*}
            \alpha_{k+1}^{2\beta}  &= M^2(1-\beta)^2\theta^\prime(\vPhi(x^{k+1}) - \vPhi(\bar{x}) + u_{k+1})^{-2} \\
            & \leq M^2(1-\beta)^2 \big(  \theta^\prime(\vPhi(x^{k+1}) - \vPhi(\bar{x}))^{-1} +  [\theta^\prime( u_{k+1})]^{-1}\big)^2 \\
            & \leq M^2(1-\beta)^2 \big(  \mathrm{dist}\!\big(0;\partial \vPhi(x^{k+1})\big) +  [\theta^\prime( u_{k+1})]^{-1}\big)^2 \\
            & \leq M^2(1-\beta)^2 \big(  \zeta\|x^{k+1} - x^k\| + b_k +  [\theta^\prime( u_{k+1})]^{-1}\big)^2\\
            & \leq 2M^2(1-\beta)^2 \big[  \zeta^2\|x^{k+1} - x^k\|^2 + \big(b_k +  [\theta^\prime( u_{k+1})]^{-1}\big)^2\big]\\
            & \leq 2M^2(1-\beta)^2 \big[  \frac{\zeta^2}{\sigma}(\vPhi_k(x^k)-\vPhi_k(x^{k+1})) + \big(b_k +  [\theta^\prime( u_{k+1})]^{-1}\big)^2\big]\\
            & \leq 2M^2(1-\beta)^2 \big[  \frac{\zeta^2}{\sigma}(\alpha_k - \alpha_{k+1})+   \big(b_k +  [\theta^\prime( u_{k+1})]^{-1}\big)^2\big],
        \end{align*} where in the first line we use $\theta^\prime(t) = M(1-\beta)t^{-\beta}$, in the second line we apply the quasi-subadditivity property with $C=1$ (see Remark \ref{remark_kindoftheta}), in the third line we use the \KLshort~property \eqref{PLK_cond},
        in the fourth line we apply \eqref{eq:pre-KL-estimate} and Lemma~\ref{Tlemma}, 
         in the fifth line we employ Young's inequality, in the sixth line we apply \eqref{eq:descent}, and in the last line we combine \eqref{defPhik}, \eqref{eq:ak} and \eqref{eq:uk}. Hence, for $\hat{C}_1 = 2M^2(1-\beta)^2\frac{\zeta^2}{\sigma}$ and $C_2 = \hat{C}_1 \frac{\sigma}{\zeta^2}$, \begin{equation*}
            \alpha_{k+1}^{2\beta} \leq \hat{C}_1(\alpha_k - \alpha_{k+1}) + C_2\big(b_k +  [\theta^\prime( u_{k+1})]^{-1}\big)^2.
        \end{equation*} Finally,  the  second term on the right-hand side of the above equation is estimated by Lemma~\ref{B-ii}, which yields the existence of a constant $\hat{C}_2 >0$ such that
\begin{equation*}
            \alpha_{k+1}^{2\beta} \leq \hat{C}_1(\alpha_k - \alpha_{k+1}) + \hat{C}_2\frac{\ln^{2\beta}(k)}{k^{2\beta(\expo/2-1)}},
        \end{equation*} 
        concluding the proof.
\end{proof}

We proceed to show a technical lemma that relates the rate of convergence of the function values and the error associated with the sample average approximation to the rate of convergence of the iterates (cf. \cite[Theorem 3.3]{frankel2015splitting}).

\begin{lemma} \label{p-rates:tech-result-3}
Suppose that the assumptions of Theorem~\ref{th:rates} are satisfied.
    Let $\tilde{\theta}(t) = \max\{\theta(t),\sqrt{t}\}$, and $(\alpha_k)_{k\in\N}$ given as in Lemma~\ref{p-rates:tech-result-2}. Then, for all $k$ sufficiently large, \[\|x^k - \bar{x}\| = \mathcal{O}(\tilde{\theta}(\alpha_{k})) + \mathcal{O}\left(k^{-\beta(\expo/2-1)+1} \ln(k)\right).\]
\end{lemma}

\begin{proof}
    By rewriting~\eqref{01ineq_est_2}  with $k-1$ instead of $k$, we have, for all sufficiently large $k$, \begin{align} \label{eq:aux-descent}
     \sigma   \| x^{k+1} - x^k \|^2 &\leq     \vPhi(x^k) + u_k   - (\vPhi(x^{k+1})  +u_{k+1})
     = \alpha_{k} - \alpha_{k+1} \leq \alpha_{k}.
\end{align} Hence,\begin{align*}
    \|\bar{x} - x^k\| & \leq \sum_{j=k}^\infty \|x^{j+1} - x^j\| \\
    & \leq 2 \|x^{k+1}-x^k\| + \frac{1}{\mathfrak{C}}\theta(\alpha_{k+1}) + \frac{1}{\zeta}\sum_{j=k+1}^\infty(b_{j-1} + [\theta^\prime(u_j)]^{-1}) \\
    & \leq \frac{2}{\sqrt{\sigma}}\sqrt{\alpha_k}  + \frac{1}{\mathfrak{C}}\theta(\alpha_{k}) + \frac{1}{\zeta}\sum_{j=k}^\infty(b_{j} + [\theta^\prime(u_{j+1})]^{-1})\\
    & \leq \mathcal{O}(\tilde{\theta}(\alpha_{k})) + \mathcal{O}\left( k^{-\beta(\expo/2-1)+1} \ln(k)\right), 
\end{align*} where in the second line we use \eqref{eq:KL-tele},  in the third line we apply \eqref{eq:aux-descent} and the fact that $\theta$ is nondecreasing (since $\theta'>0$), and the last line follows from the definition of $\tilde{\theta}$ and Lemma~\ref{B-ii}. 
\end{proof}

\subsubsection{Proof of Theorem~\ref{th:rates}}\label{s:proof-rates}

We are now ready to deduce local convergence rates for both the sequence of function values and the iterates.

\medskip 
\begin{proofof}{th:rates}
     We first address the convergence rates for the function values. Observe that if the \KLshort~condition \eqref{PLK_cond} holds in the exponential case with $\beta \in (0,\frac{1}{2})$, then it also holds for $\beta = \frac{1}{2}$, so we concentrate on the latter case.  For $\beta = \frac{1}{2}$,   Lemma~\ref{p-rates:tech-result-2}  yields a  recursion of the type~\eqref{eq:recursion-lemma} with $q=p_1=1$ and $p_2=(\frac{\expo}{2} -1)$. Then  Lemma~\ref{lemma:sequences}(ii) implies that  $\alpha_k = \mathcal{O}\bigl(k^{-(\expo/2-1)}\ln(k)\bigr)$, from where the first claim for the function values in item (i) follows as $u_k \geq 0$. Similarly, for $\beta \in (\frac{1}{2},1)$, combine  Lemma~\ref{p-rates:tech-result-2} and  Lemma~\ref{lemma:sequences}(iii) to obtain $\alpha_k = \mathcal{O}\bigl(k^{-2\beta\min\{\frac{\expo}{2}-1,\frac{1}{2\beta-1}\}+1}\ln(k)\bigr)$, from where the first estimate in (ii) follows as well.  
    
   As for the convergence rate of the iterates, if $\beta \in (0, \frac{1}{2}]$, for all $k$ sufficiently large,  $\tilde{\theta}(\alpha_{k}) = \mathcal{O}(\sqrt{\alpha_{k}}) $ and $k^{-1/2(\expo/2-1)}\sqrt{\ln(k)} \leq k^{-1/2(\expo/2-1)+1}\ln(k)$, then 
    Lemma~\ref{p-rates:tech-result-3} implies $\|x^k-\bar{x}\| =    \mathcal{O}\bigl(k^{-1/2(\expo/2-1)+1}\ln(k)\bigr)$, showing the second rate in item (i).  If $\beta \in (\frac{1}{2},1)$, then for all $k$ large enough, $\tilde{\theta}(\alpha_{k}) = \mathcal{O}(\alpha_{k}^{1-\beta}) $, so that from Lemma~\ref{p-rates:tech-result-3}, \begin{equation} \label{eq:abstract-rate}
        \|x^k - \bar{x}\|= \mathcal{O}\bigl(k^{(1-\beta)(1-2\beta\nu)}\ln^{1-\beta}(k)\bigr) + \mathcal{O}\bigl(k^{-\beta(\expo/2-1)+1} \ln(k)\bigr),
    \end{equation}where $\nu = \min\{\frac{\expo}{2}-1,\frac{1}{2\beta-1}\}$. We separate the analysis in two cases.    
    
    \noindent Case 1: $\nu = \frac{\expo}{2}-1$ ($\frac{\expo}{2}-1 \leq \frac{1}{2\beta-1}$). In this case, \[\tilde{\nu} := (1-\beta)(-2\beta\nu+1) +\beta\left(\frac{\expo}{2}-1\right)-1 \leq 0 \implies \frac{k^{\tilde{\nu}}}{\ln^\beta(k)} \to 0, \] then for all $k$ large enough, 
    \[k^{(1-\beta)(-2\beta\nu+1)}\ln^{1-\beta}(k) \leq k^{-\beta(\expo/2-1)+1}\ln(k).\] Hence  the second term in \eqref{eq:abstract-rate} dominates for all $k$ sufficiently large, therefore \[\|x^k - \bar{x}\|=  \mathcal{O}\left(k^{-\beta(\expo/2-1)+1} \ln(k)\right).\]   Case 2: $\nu = \frac{1}{2\beta-1}$ ($\frac{\expo}{2}-1 > \frac{1}{2\beta-1}$). In this case, \[\tilde{\nu}:=(1-\beta)(-2\beta\nu+1)  +\beta\left(\frac{\expo}{2}-1\right)-1 >0 \implies \frac{k^{\tilde{\nu}}}{\ln^\beta(k)} \to +\infty,\] then for all $k$ large enough, \[k^{(1-\beta)(-2\beta\nu+1)}\ln^{1-\beta}(k) \geq k^{-\beta\left(\frac{\expo}{2}-1\right)+1} \ln(k).\] Hence  the first term in \eqref{eq:abstract-rate} dominates for all $k$ sufficiently large, consequently \[\|x^k - \bar{x}\|=  \mathcal{O}\left(k^{(1-\beta)(-2\beta\nu+1)}\ln^{1-\beta}(k)\right), \] from where we can conclude the proof.  

\end{proofof}

\begin{remark} A few comments about the convergence rates are in order.
\begin{enumerate}[label=(\roman*), ref=(\roman*)]
    \item In Theorem~\ref{th:rates}, we skip the case $\beta=0$ for the exponent, as the corresponding desingularizing function does not satisfy the condition in \eqref{desing_condi}. For this exponent, one usually obtains termination of the algorithm in finitely many steps (see, e.g. \cite[page 2]{li2018calculus}). Nonetheless, our conjecture is that such a result might not be possible to retrieve in general for our method due to the average approximation errors.

    \item The rates in Theorem~\ref{th:rates} refer to the \emph{outer iterations} of Algorithm~\ref{alg:1}, without counting how many \emph{inner iterations} the line search requires to produce a candidate point that passes the line search test. However, in the setting of Remark~\ref{r:The01}\ref{r:stabilize}, that is, when the stepsizes are nonincreasing, these rates become global.
    \item The exponents in Theorem~\ref{th:rates} reveal a sharp \emph{transition} at $\beta = \tfrac12$ and an interplay between the \KLshort~exponent $\beta$ and the sample-size growth rate $\expo$. For $\beta\in(0,\tfrac12]$, the function-value exponent $\tfrac{\expo}{2}-1$ grows linearly and unboundedly with $\expo$: a faster enrichment of the sample, which makes the approximation error vanish faster, directly translates into a faster decay of the objective gap. For $\beta\in(\tfrac12,1)$, however, the function-value exponent equals $2\beta\bigl(\tfrac{\expo}{2}-1\bigr)-1$ only while $\tfrac{2+2\beta}{\beta}<\expo\le\tfrac{4\beta}{2\beta-1}$, and \emph{saturates} at the value $\tfrac{1}{2\beta-1}$ once $\expo>\tfrac{4\beta}{2\beta-1}$. In other words, beyond a threshold sample-size growth, the stochastic rate ceases to improve and recovers, up to the logarithmic factor, the classical deterministic rate $\mathcal{O}\bigl(k^{-1/(2\beta-1)}\bigr)$ associated with the exponential \KLshort~condition; see, e.g., \cite[Theorem 2]{AttouchBolte2009}. The same phenomenon holds for the iterates, whose exponent saturates at $\tfrac{1-\beta}{2\beta-1}$, matching the deterministic iterate rate $\mathcal{O}\bigl(k^{-(1-\beta)/(2\beta-1)}\bigr)$. Hence, the deterministic \KLshort~rate is the natural barrier of the method, and it is attained as soon as the sampling error decays sufficiently fast.
    \end{enumerate}
\end{remark}

\vspace*{-0.2in}

\section{Numerical illustrations} \label{s:numerics}

This section provides numerical illustrations of the performance of PSS in two different applications: optimal quantization and  partial-label linear regression. 

\subsection{Optimal quantization}

The {optimal quantization} problem consists of approximating a (continuous) probability distribution by a discrete measure that minimizes the expected distortion. This problem has been extensively studied in multiple fields; see, for instance,~\cite{MR1764176} and the references therein. The problem can be mathematically formulated as follows. Let $\mu$ be a  probability distribution in $\R^d$. The goal of the optimal quantization problem is to find $s\in\N$ points $x_1,x_2,\ldots,x_s \in \R^{ d}$ solving the minimization problem
\[
\min_{x_1,\ldots,x_s\in\R^d} \quad \int_{\R^d}  \min_{t\in\{1,\ldots,s\}} \|x_t - \xi\|_2^2 \; d\mu(\xi).
\] The objective function of the  problem above can be posed as the expectation of the upper-$C^2$ normal integrand given by, for $X:= (x_1,x_2,\ldots,x_s) \in \R^{ d\times s}$,
\[
\varphi:\R^d \times \R^{d\times s} \to \R, \; (\xi,X)\mapsto \min_{t\in\{1,\ldots,s\}} \|x_t - \xi\|^2,
\]
see~\cite[Section~4.2]{aragonartacho2023boosted} for a deterministic, empirical-risk formulation of the problem. 

In our experiments we introduce an additional constraint/penalization on the points $X$. Namely, we are interested in the problem
\begin{equation*}
\min_{X\in\R^{d\times s}}  \quad \Intf{\varphi}(X) + \psi(X),
\end{equation*}
where $\psi$ denotes a penalization term. In the following, we describe two setups that differ in the choice of $\psi$.


\textit{Setup 1: constrained optimal quantization.} In this setting, we take $\psi$ to be the indicator function $\iota_C$ of a nonempty closed set $C$, which is lsc and prox-bounded (with prox-boundedness threshold $+\infty$). 


\textit{Setup 2: regularized optimal quantization.} In this setting, we take \[\psi(X) = \lambda_1 \sum_{j=1}^s\sum_{i<j}P(x_i-x_j),\] where $\lambda_1>0$ is the level of penalization parameter, and $P$ is a penalization term traditionally used to promote sparsity. This formulation is inspired by the models of fusion-based clustering \cite{hocking2011clusterpath}, the discrete version of the optimal quantization problem. The expected effect of $\psi$ on the solution is to induce the merge (fusion) of quantization points that are close enough to minimize the overall cost, reducing the number of effective centroids that approximate the original probability distribution.

For both of the setups above, we construct the original distribution $\mu$ as a uniform mixture of ten bivariate Gaussian distributions whose parameters are generated randomly. The centers of these distributions are  expressed in polar coordinates as $\bigl(r\cos(\theta),r\sin(\theta)\bigr)$, where the angles $\theta$ and the radii $r$ are drawn uniformly in the intervals $[0,2\pi]$ and $[10,11]$, respectively. For each distribution, the covariance matrix is given by $\sigma^2 I$, where $\sigma^2$ is sampled independently and uniformly from  $[1,3]$.  To generate samples $\xi \sim \mu$, one of the ten Gaussian components is selected with equal probability, followed by sampling from such a distribution. We aim to discretize the original distribution with $s=15$ points. Regarding the numerical implementation of Algorithm~\ref{alg:1}, we take $n_0=1000$ and define $n_k=n_{k-1} + k$, for all $k\geq 1$. Further, it is well-known  that $\varphi(\xi,\cdot)$ verifies the stronger property of being $1$-upper-$C^2$; see, e.g.,~\cite{aragonartacho2023boosted}. Hence, one can take $\gamma_k = \tilde{\gamma} \in{(0,1/2)}$. This avoids the computational burden of performing the \emph{Armijo-type line search} at every iteration of the method. In particular, we fixed $\tilde{\gamma}=0.9/2$. In our experiments, Algorithm~\ref{alg:1} stopped and returned the $k$-th iterate if
\begin{equation}\label{eq:relerror}
\frac{\|X^k-X^{k-1}\|}{\max\{1,\|X^{k-1}\|\}} < 10^{-3}.
\end{equation}
For the initialization of the algorithm, we randomly picked the initial points uniformly from the interval $[-1,1]$ for both setups.

Figure~\ref{fig:circ1} illustrates the output of Algorithm~\ref{alg:1} of experimental Setup 1. The original probability distribution $\mu$ is visualized as a gray-scale point cloud, where darker regions correspond to areas of higher probability mass. Figure~\ref{fig:circ1a} shows the optimal quantization results (orange dots) with no constraints (i.e., $C=\R^2$) for $s=15$ target points. The purple dotted lines represent the Voronoi partition associated with the computed quantization points. Observe that a few points try to coalesce toward the origin, suggesting that these points might be redundant for the quantization. Figure~\ref{fig:circ1b} displays the outcome of adding the circumference constraint $C:=\{ x\in\R^2\, : \, \|x\| = 10.5\}$ depicted in dashed blue lines, which induces no trivial quantization points. 

\begin{figure}[htbp]
    \centering
    \begin{subfigure}[b]{0.45\textwidth}
        \centering
\includegraphics[width=\textwidth]{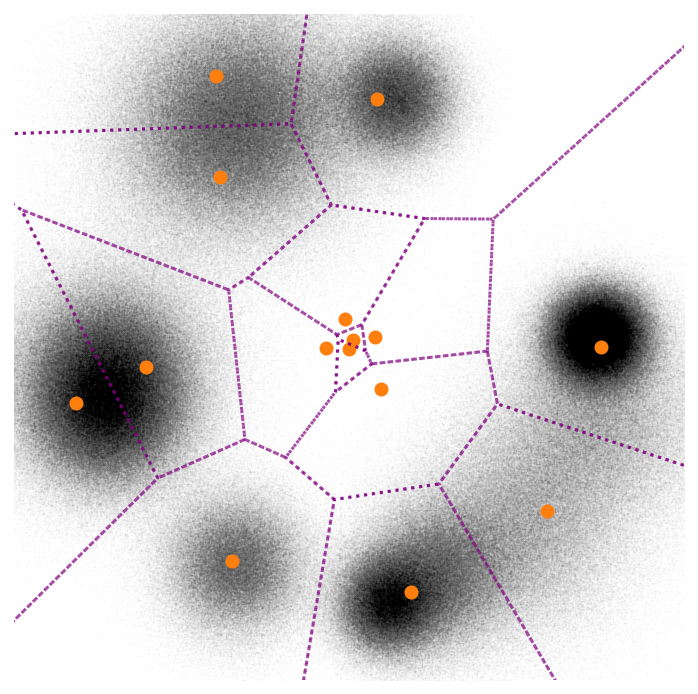} 
        \caption{$\psi = 0$ (unconstrained)}
        \label{fig:circ1a}
    \end{subfigure}
    \hfill 
    \begin{subfigure}[b]{0.45\textwidth}
        \centering
        \includegraphics[width=\textwidth]{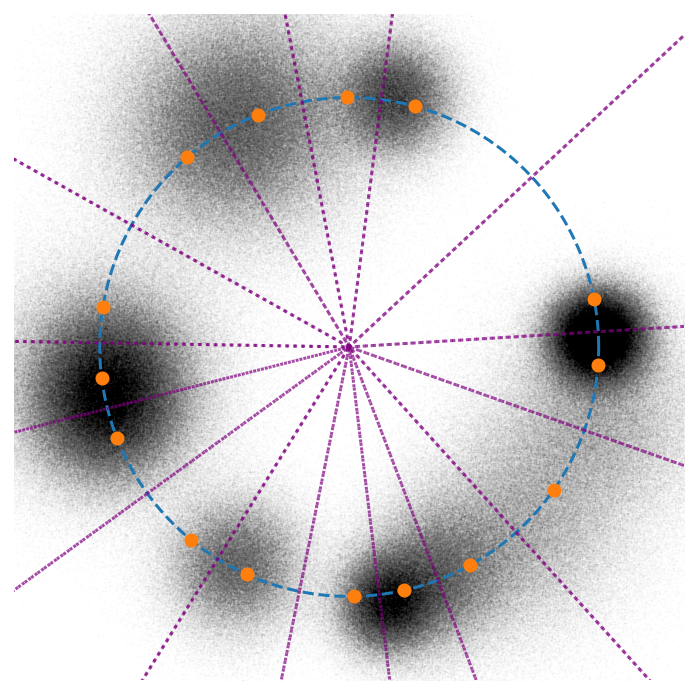} 
        \caption{$\psi = \iota_C$, constrained in the circumference (blue dashed line)}
        \label{fig:circ1b}
    \end{subfigure}
    \caption{Optimal quantization results (orange dots) of Setup 1 and corresponding Voronoi partition (purple dashed lines).}
    \label{fig:circ1}
\end{figure}

To promote the coalescence of quantization points observed in Figure~\ref{fig:circ1a}, we implement Setup 2 with the (lsc, prox-bounded) convex penalty $\ell_1$ and with the nonconvex penalty $\ell_1-\ell_2$ of \cite{esser2013method}. For the latter, we evaluate the proximal operator only with respect to the $\ell_1$-norm, since the $\ell_2$-norm is twice continuously differentiable and is therefore absorbed into the upper-$C^2$ term $\varphi$. Figures~\ref{fig:circ2} and~\ref{fig:circ3} report the results obtained with the penalization parameters $\lambda_1 = 0.01$ and $\lambda_1=0.0025$, respectively. 

For the larger parameter $\lambda_1=0.01$, Figures~\ref{fig:circ2a} and~\ref{fig:circ2b} yield essentially the same quantization (though numerically distinct): the points clustered near the origin in Figure~\ref{fig:circ1a} merge, but the strong penalization also drives the remaining points to underrepresent the data clouds. Decreasing the penalization to $\lambda_1=0.0025$ produces a qualitatively different and more desirable behavior, displayed in Figures~\ref{fig:circ3a} and~\ref{fig:circ3b}. Among all displayed outcomes, the $\ell_1-\ell_2$ penalty in Figure~\ref{fig:circ3b} gives the most faithful quantization, assigning essentially one point to each data cloud, while only the redundant points near the origin are not fully fused. This illustrates the expected advantage of the nonconvex penalty in reducing the shrinking bias of the $\ell_1$ regularization.



\begin{figure}[htbp]
    \centering
    \begin{subfigure}[b]{0.45\textwidth}
        \centering
        \includegraphics[width=\textwidth]{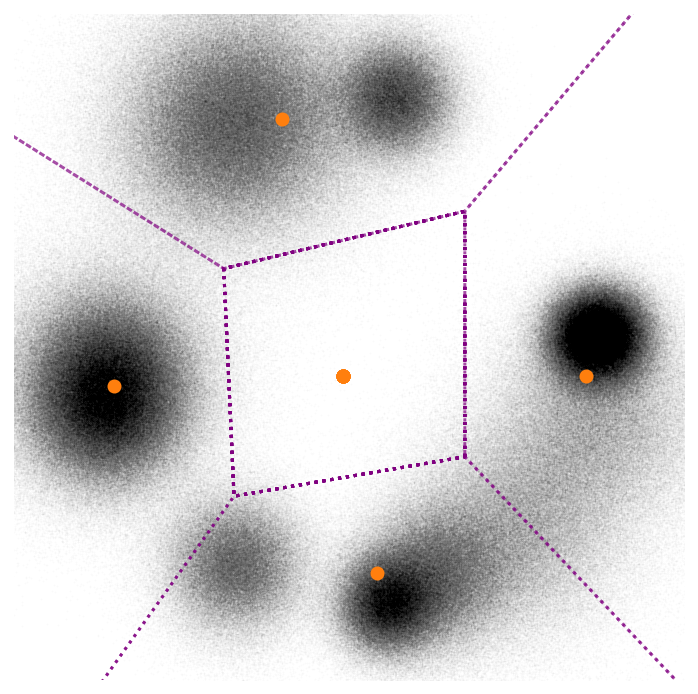} 
        \caption{$\ell_1$ penalization}
        \label{fig:circ2a}
    \end{subfigure}
    \hfill 
    \begin{subfigure}[b]{0.45\textwidth}
        \centering
        \includegraphics[width=\textwidth]{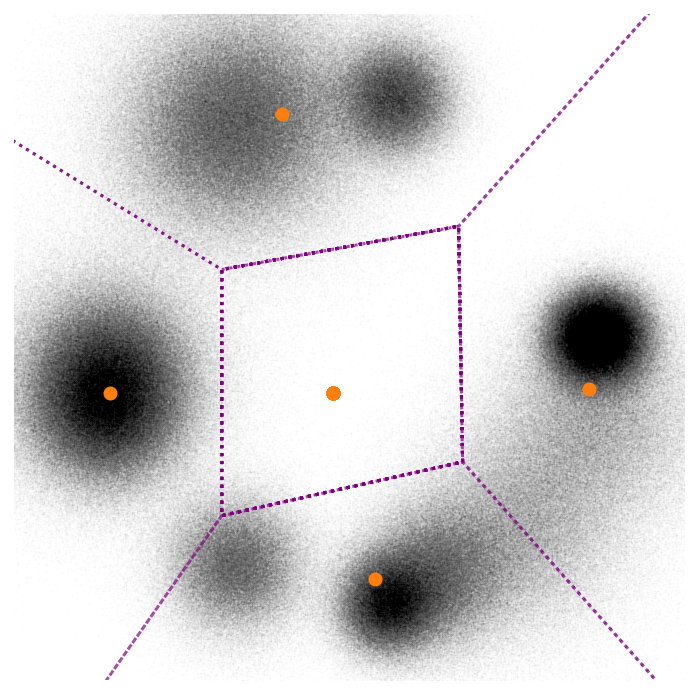} 
        \caption{$\ell_1-\ell_2$ penalization}
        \label{fig:circ2b}
    \end{subfigure}
    \caption{Optimal quantization results of Setup 2 with penalization parameter $\lambda_1 = 0.01$.}
    \label{fig:circ2}
\end{figure}

\begin{figure}[htbp]
    \centering
    \begin{subfigure}[b]{0.45\textwidth}
        \centering
        \includegraphics[width=\textwidth]{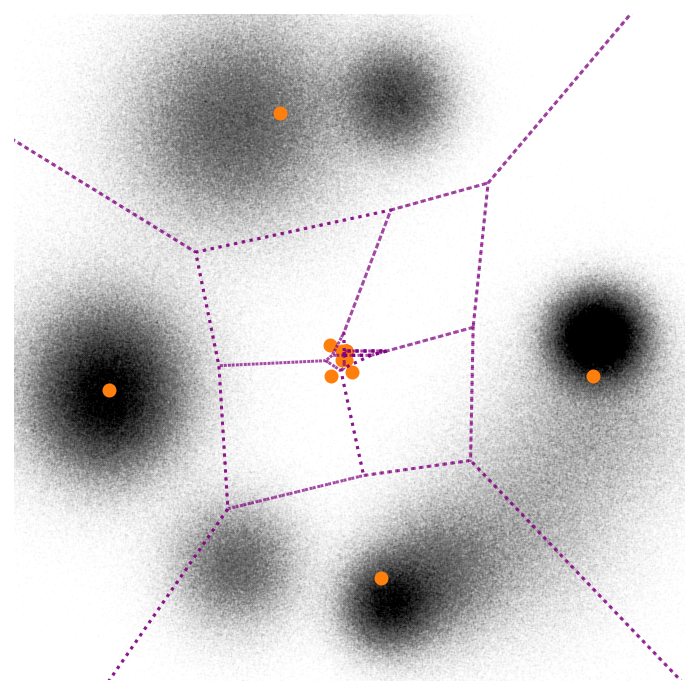} 
        \caption{$\ell_1$ penalization}
        \label{fig:circ3a}
    \end{subfigure}
    \hfill 
    \begin{subfigure}[b]{0.45\textwidth}
        \centering
        \includegraphics[width=\textwidth]{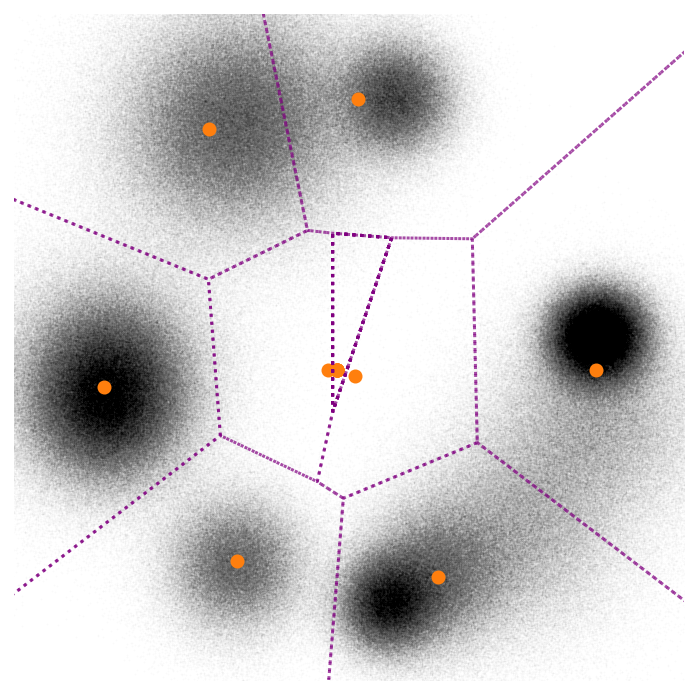} 
        \caption{$\ell_1-\ell_2$ penalization}
        \label{fig:circ3b}
    \end{subfigure}
    \caption{Optimal quantization results of Setup 2 with penalization parameter $\lambda_1 = 0.0025$.}
    \label{fig:circ3}
\end{figure}

\subsection{Partial-label linear regression}

Many real-world applications aim to learn information from regression models that present a set of multiple candidate labels. These labels may be produced from ambiguity and noise in data annotation tasks, and  only  a few of them might correspond to true labels. This scenario has driven  the field of weakly supervised learning known as \emph{partial-label learning}~\cite{JMLR:v12:cour11a}, which has found applications in face recognition systems and multimedia content analysis, among others; see, e.g.,~\cite{NIPS2010_c9e1074f,6618941,7968363}.
For regression tasks that learn with real-valued labels, a new model has recently been proposed in~\cite{Cheng_Wang_Feng_Zhang_An_2023} that minimizes the least loss incurred by candidate labels. In the context of linear regression, this model can be posed as a particular instance of~\eqref{ProblemP}, which also allows including a regularizing penalty to promote sparse solutions.

The problem of interest can be formulated as follows. An agent seeks to find a sparse vector of weights $w\in\mathbb{R}^d$ from a variable environment that, 
given a feature vector $x\in\mathbb{R}^d$, provides multiple  potential target values $y=(y_1,y_2,\ldots,y_s)\in\R^s$. Every $y_t$ represents competing regimes dictated by the same underlying features. The problem can then be tackled by the stochastic program
\begin{equation}\label{eq:pl-lregression}
\min_{w\in\mathbb{R}^d} \mathbb{E}_{(x,y)\sim\mu} \left[ \min \biggl\{ \frac{1}{2}(y_t - x^T w)^2\,:\, t=1,\ldots,s \biggr\} \right] + \psi(w),
\end{equation}
where $\mu$ is a joint probability distribution for $(x,y)$ and $\psi$ is a sparsity-inducing penalty. By using well-known characterizations of upper-$C^2$ functions (see, e.g.,~\cite[Section~3]{aragonartacho2025nonmonotonesubgradientmethodsbased}) it is not difficult to check that the above problem is in the form of~\eqref{ProblemP}.

We present here a preliminary experiment to illustrate the performance of  Algorithm~\ref{alg:1} applied to problems in the form of~\eqref{eq:pl-lregression}. Specifically, we take $s=2$, namely, only two candidate labels are given. To sample $(x,y_1,y_2)$, we first randomly construct two sparse vectors $w_1^*\in\R^{50}$ and $w_2^*\in\R^{50}$ that represent the true weights of the regime.  Then, for a feature vector $x$ sampled from a standard normal distribution, we take $y_t = x^Tw_t^* + e_t$, where the noise $e_t$ is drawn from a normal distribution centered at the origin and with standard deviation $0.05$, for $t=1,2$.  For the regularization, in our experiments we take (1) $\psi = \lambda_1 \|\cdot\|_1$ with regularization parameter $\lambda_1=0.1$, and (2) $\psi =  \mbox{MCP}$ \cite{Zhang2010MCP}. Experiments on supervised learning tasks \cite{atenas2025understanding,atenas2025shadow} indicate that the use of MCP and other nonconvex penalties can improve the quality of the solution compared to the results provided by the $\ell_1$ penalization. This corresponds to reducing the shrinking bias mentioned in the introduction. The  MCP is weakly convex \cite{bohm2021variable}, and thus (lsc) prox-bounded. Its proximity operator is  usually called the \emph{firm threshold}, its expression can be consulted, for instance, in~\cite{MR3548876}. Specifically, we  consider the  MCP regularizer and its proximity operator as given in~\cite{MR3548876} with $\tau=0.1$ and $\rho=2$.


In this setup, the implementation of our PSS method requires conducting a line search to estimate the stepsize parameter in Step~\ref{step:ls}. For it, we set the parameters $\sigma=0.1$, $\rho=0.5$, $\bar{\gamma}_0=2$  and $\bar{\gamma}_k = \max\{\gamma_{k-1},0.01\}$, for $k\geq 1$. Finally, we take $n_k=k+1$, for all $k\in\N$. We stop the algorithm after obtaining a relative error~\eqref{eq:relerror} smaller than $10^{-4}$.

Figure~\ref{fig:plrn} presents a stem plot showcasing the comparison between the true weight vectors ($w_1^*$ and $w_2^*$) and the weights obtained by PSS (red crosses) initialized at the origin for both sparse penalties. As observed, both cases approximate the  weight $w_1^*$. Figure~\ref{fig:plrn-2} shows a second representative instance where PSS initialized at a random vector converges to different weight vectors depending on the regularization strategy.

\begin{figure}[ht!]
\begin{center}
\includegraphics[width=\textwidth]{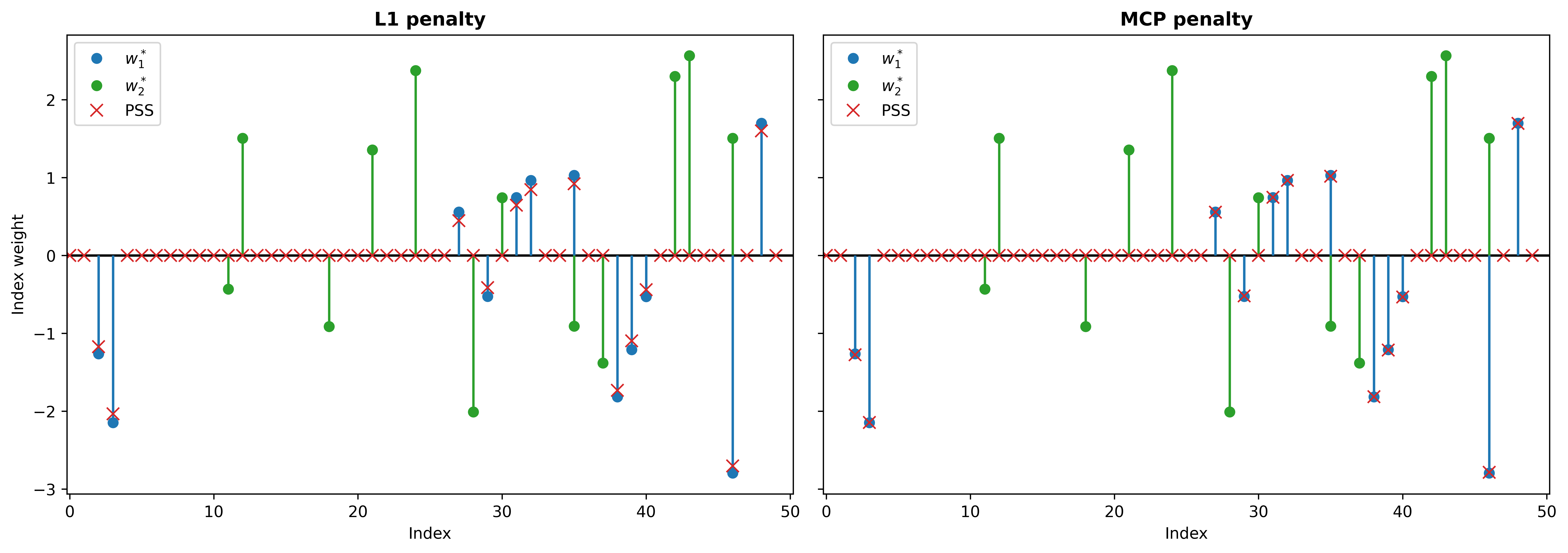}
\end{center}
\caption{
Instance of problem~\eqref{eq:pl-lregression}  where PSS approximates the same true weight vector for the different regularization penalties.
}\label{fig:plrn}
\end{figure}

\begin{figure}[ht!]
\begin{center}
\includegraphics[width=\textwidth]{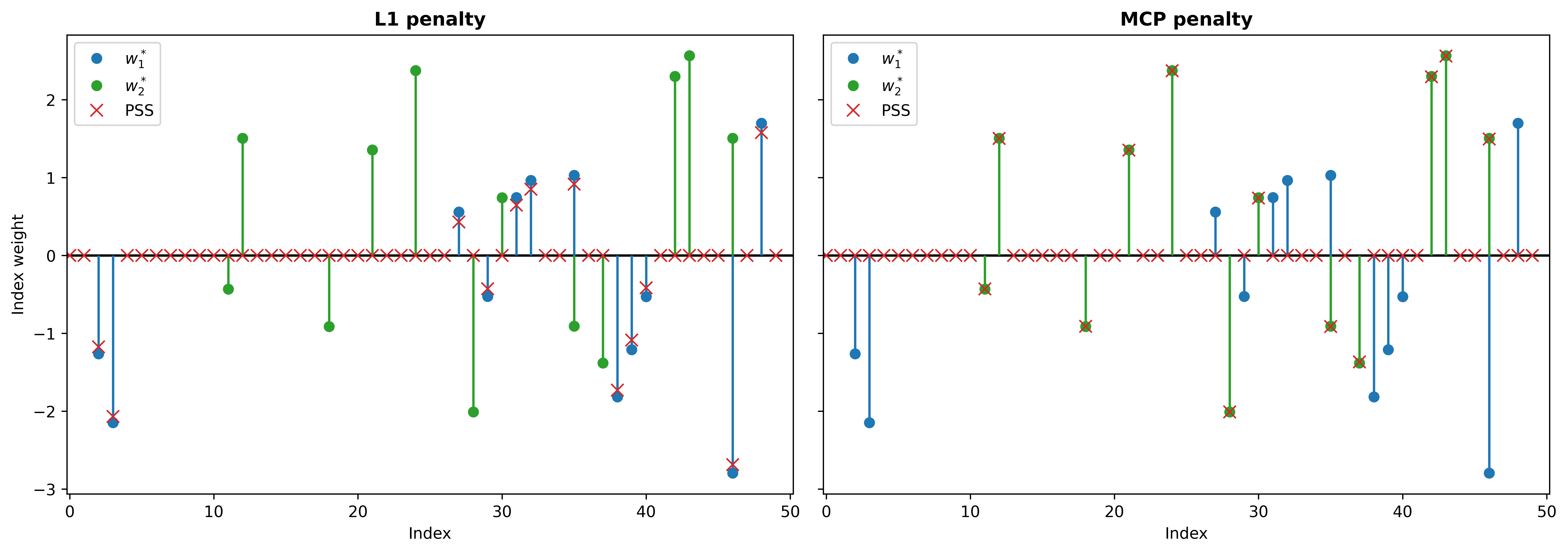}
\end{center}
\caption{Instance of problem~\eqref{eq:pl-lregression}  where PSS approximates different true weight vectors depending on the  regularization penalty.}\label{fig:plrn-2}
\end{figure}

\section{Conclusion} \label{s:conclusion}


In this work, we developed a convergence analysis for a proximal sample-based
subgradient method with an Armijo-type line search for nonsmooth and nonconvex
stochastic optimization problems. The method applies to objectives given by the
sum of an expected upper-$C^2$ normal integrand and a proper lower semicontinuous
prox-bounded function, thereby covering nonsmooth losses, constraints, and
possibly nonconvex regularization terms. The main idea of the analysis is to
interpret the sample-based sufficient decrease condition as an inexact descent
mechanism for the true objective, where the perturbation terms are induced by
the empirical approximation and vanish asymptotically. This viewpoint is inspired
by the deterministic convergence framework, but requires additional probabilistic estimates
to handle the nonmonotonicity produced by stochastic errors.

We established almost sure convergence of both the sample-based and true
objective values, square summability of the increments, and stationarity of every
accumulation point of bounded trajectories. A distinctive feature of the result
is that, for subsequential convergence, no prescribed growth rate is imposed on
the sample-size sequence: it is enough to assume that the sequence is
nondecreasing and unbounded. This flexibility is obtained through a localization
argument based on stopping times, which allows the analysis to be carried out
along bounded sample paths. In particular, the convergence result does not
require an a priori almost sure boundedness assumption on the whole stochastic
process.

Under the Kurdyka--{\L}ojasiewicz condition, we further strengthened the
subsequential convergence result to convergence of the whole trajectory to a
single stationary point, provided that the sample errors satisfy suitable
summability requirements. Moreover, for desingularizing functions of the form
$\theta(t)=Mt^{1-\beta}$ and polynomially increasing sample sizes, we derived
explicit local convergence rates for both the objective values and the iterates.
The proof combines uniform convergence estimates for empirical averages,
inexact descent inequalities, relative-error bounds adapted to sample-based
models, and recursion estimates tailored to stochastic perturbations. These
results show that the \KLshort~methodology, which is classical in deterministic
nonconvex optimization, can also be effectively used in stochastic nonmonotone
settings when the approximation errors are sufficiently controlled.

The numerical experiments on optimal quantization and partial-label learning
illustrate the applicability of the proposed method to relevant nonsmooth
stochastic models. They also suggest that the line search mechanism can provide
a practical way to balance descent and statistical accuracy without enforcing a
rigid sample-size schedule.

Several questions remain open for future research. On the theoretical side, it
would be interesting to refine the stopping-time localization technique in order
to weaken standard assumptions commonly imposed in stochastic optimization, such
as uniform bounded variance or global uniform smoothness. Other natural directions  to investigate are alternative sample-size regimes, adaptive sampling
strategies, and variance-reduction mechanisms within the present nonsmooth nonconvex function setting and the extension of the variational inequality framework developed in previous works by Iusem, Jofre, Oliveira and Thompson ~\cite{iusem2017extragradient, iusem2019variance, iusem2019incremental}{. In addition, Remark~\ref{r:Th4.2}\ref{r:Th4.2-ii} suggests the possibility
of developing an abstract convergence theory for explicit and implicit stochastic
methods satisfying descent-type conditions with controlled errors, in the spirit
of deterministic \KLshort-based frameworks. On the numerical side, further experiments
are needed to assess the practical performance of the method on large-scale
learning problems and to compare different strategies for choosing the sample
sizes and the initial trial stepsizes in the line search procedure.

\section*{Statements and Declarations}

\paragraph{Competing interests} 
The authors declare that they have no competing interests.

 \bibliographystyle{acm}
 \bibliography{references} 

@article{absil2005convergence,
  title={Convergence of the iterates of descent methods for analytic cost functions},
  author={Absil, Pierre-Antoine and Mahony, Robert and Andrews, Ben},
  journal={SIAM Journal on Optimization},
  volume={16},
  number={2},
  pages={531--547},
  year={2005},
  publisher={SIAM}
}

@article{nguyen2025stochastic,
  title={Stochastic {ISTA/FISTA} adaptive step search algorithms for convex composite optimization},
  author={Nguyen, Lam M and Scheinberg, Katya and Tran, Trang H},
  journal={J. Optim. Theory Appl.},
  volume={205},
  number={1},
  pages={10},
  year={2025},
  publisher={Springer}
}

@article{paquette2020stochastic,
  title={A stochastic line search method with expected complexity analysis},
  author={Paquette, Courtney and Scheinberg, Katya},
  journal={SIAM J.  Optim.},
  volume={30},
  number={1},
  pages={349--376},
  year={2020},
  publisher={SIAM}
}

@article{esser2013method,
  title={A method for finding structured sparse solutions to nonnegative least squares problems with applications},
  author={Esser, Ernie and Lou, Yifei and Xin, Jack},
  journal={SIAM J.  Imaging Sci.},
  volume={6},
  number={4},
  pages={2010--2046},
  year={2013},
  publisher={SIAM}
}

@inproceedings{hocking2011clusterpath,
  title={Clusterpath an algorithm for clustering using convex fusion penalties},
  author={Hocking, Toby Dylan and Joulin, Armand and Bach, Francis and Vert, Jean-Philippe},
  booktitle={28th ICML},
  pages={1},
  year={2011}
}

@article{pougkakiotis2023zeroth,
  title={A zeroth-order proximal stochastic gradient method for weakly convex stochastic optimization},
  author={Pougkakiotis, Spyridon and Kalogerias, Dionysis},
  journal={SIAM J.  Sci. Comput.},
  volume={45},
  number={5},
  pages={A2679--A2702},
  year={2023},
  publisher={SIAM}
}

@article{zhang2022stochastic,
  title={Stochastic variance-reduced prox-linear algorithms for nonconvex composite optimization},
  author={Zhang, Junyu and Xiao, Lin},
  journal={Math. Program.},
  volume={195},
  number={1},
  pages={649--691},
  year={2022},
  publisher={Springer}
}

@article{jia2025first,
  title={First-order methods for nonsmooth nonconvex functional constrained optimization with or without slater points},
  author={Jia, Zhichao and Grimmer, Benjamin},
  journal={SIAM J.  Optim.},
  volume={35},
  number={2},
  pages={1300--1329},
  year={2025},
  publisher={SIAM}
}

@article{fatkhullin2025stochastic,
  title={Stochastic optimization under hidden convexity},
  author={Fatkhullin, Ilyas and He, Niao and Hu, Yifan},
  journal={SIAM J.  Optim.},
  volume={35},
  number={4},
  pages={2544--2571},
  year={2025},
  publisher={SIAM}
}

@article{atenas2025shadow,
  title={Shadow splitting methods for nonconvex optimisation: epi-approximation, convergence and saddle point avoidance},
  author={Atenas, Felipe},
  year={2025}, 
  url={https://arxiv.org/abs/2512.20433}, 
      note={Preprint, available at: \href{https://arxiv.org/abs/2512.20433}{arXiv:2512.20433}}
}

@article{atenas2025understanding,
  title={Understanding the {D}ouglas--{R}achford splitting method through the lenses of {M}oreau-type envelopes},
  author={Atenas, Felipe},
  journal={Comput. Optim.  Appl.},
  volume={90},
  number={3},
  pages={881--910},
  year={2025},
  publisher={Springer}
}

@article{bohm2021variable,
  title={Variable Smoothing for Weakly Convex Composite Functions},
  author={B{\"o}hm, Axel and Wright, Stephen J},
  journal={J. Optim. Theory Appl.},
  volume={188},
  number={3},
  pages={628--649},
  year={2021},
  publisher={Springer}
}

@article {MR3902455,
	AUTHOR = {Davis, Damek and Drusvyatskiy, Dmitriy},
	TITLE = {Stochastic model-based minimization of weakly convex
	functions},
	JOURNAL = {SIAM J. Optim.},
	FJOURNAL = {SIAM Journal on Optimization},
	VOLUME = {29},
	YEAR = {2019},
	NUMBER = {1},
	PAGES = {207--239},
	ISSN = {1052-6234,1095-7189},
	MRCLASS = {90C25 (65K05 65K10 90C15)},
	MRNUMBER = {3902455},
	MRREVIEWER = {Wim\ van Ackooij},
	DOI = {10.1137/18M1178244},
	URL = {https://doi.org/10.1137/18M1178244},
}

@article {MR3982682,
	AUTHOR = {Davis, Damek and Grimmer, Benjamin},
	TITLE = {Proximally guided stochastic subgradient method for nonsmooth,
	nonconvex problems},
	JOURNAL = {SIAM J. Optim.},
	FJOURNAL = {SIAM Journal on Optimization},
	VOLUME = {29},
	YEAR = {2019},
	NUMBER = {3},
	PAGES = {1908--1930},
	ISSN = {1052-6234,1095-7189},
	MRCLASS = {90C15 (65K05 65K10 90C26 90C30)},
	MRNUMBER = {3982682},
	MRREVIEWER = {Shouqiang\ Du},
	DOI = {10.1137/17M1151031},
	URL = {https://doi.org/10.1137/17M1151031},
}

@article {MR3439803,
	AUTHOR = {Ghadimi, Saeed and Lan, Guanghui and Zhang, Hongchao},
	TITLE = {Mini-batch stochastic approximation methods for nonconvex
	stochastic composite optimization},
	JOURNAL = {Math. Program.},
	FJOURNAL = {Mathematical Programming},
	VOLUME = {155},
	YEAR = {2016},
	NUMBER = {1-2},
	PAGES = {267--305},
	ISSN = {0025-5610,1436-4646},
	MRCLASS = {90C25 (90C06 90C15 90C22)},
	MRNUMBER = {3439803},
	DOI = {10.1007/s10107-014-0846-1},
	URL = {https://doi.org/10.1007/s10107-014-0846-1},
}

@article {MR3459195,
	AUTHOR = {Ghadimi, Saeed and Lan, Guanghui},
	TITLE = {Accelerated gradient methods for nonconvex nonlinear and
	stochastic programming},
	JOURNAL = {Math. Program.},
	FJOURNAL = {Mathematical Programming},
	VOLUME = {156},
	YEAR = {2016},
	NUMBER = {1-2},
	PAGES = {59--99},
	ISSN = {0025-5610,1436-4646},
	MRCLASS = {62L20 (68Q25 90C15 90C25)},
	MRNUMBER = {3459195},
	DOI = {10.1007/s10107-015-0871-8},
	URL = {https://doi.org/10.1007/s10107-015-0871-8},
}

@article{kanzow2026nonmonotonedescentmethodoptimization,
      title={A Nonmonotone Descent Method for Optimization Problems Defined by Upper-$\mathcal{C}^2 $ Functions over Submanifolds}, 
      author={Christian Kanzow and Leo Lehmann},
      year={2026},
      eprint={2605.26909},
      archivePrefix={arXiv},
      primaryClass={math.OC},
      url={https://arxiv.org/abs/2605.26909}, 
      note={Preprint, available at: \href{https://arxiv.org/abs/2605.26909}{arXiv:2605.26909}}
}

@article {MR4486508,
	AUTHOR = {Lei, Jinlong and Shanbhag, Uday V.},
	TITLE = {Asynchronous variance-reduced block schemes for composite
	non-convex stochastic optimization: block-specific steplengths
	and adapted batch-sizes},
	JOURNAL = {Optim. Methods Softw.},
	FJOURNAL = {Optimization Methods \& Software},
	VOLUME = {37},
	YEAR = {2022},
	NUMBER = {1},
	PAGES = {264--294},
	ISSN = {1055-6788,1029-4937},
	MRCLASS = {90C15 (65K10 90C26)},
	MRNUMBER = {4486508},
	DOI = {10.1080/10556788.2020.1746963},
	URL = {https://doi.org/10.1080/10556788.2020.1746963},
}

@article{iusem2017extragradient,
  title={Extragradient method with variance reduction for stochastic variational inequalities},
  author={Iusem, Alfredo N and Jofr{\'e}, Alejandro and Oliveira, Roberto I and Thompson, Philip},
  journal={SIAM Journal on Optimization},
  volume={27},
  number={2},
  pages={686--724},
  year={2017},
  publisher={SIAM}
}

@article{iusem2019variance,
  title={Variance-based extragradient methods with line search for stochastic variational inequalities},
  author={Iusem, Alfredo N and Jofr{\'e}, Alejandro and Oliveira, Roberto I and Thompson, Philip},
  journal={SIAM Journal on Optimization},
  volume={29},
  number={1},
  pages={175--206},
  year={2019},
  publisher={SIAM},
  doi={10.1137/17M1144799}
}

@article{iusem2019incremental,
  title={Incremental constraint projection methods for monotone stochastic variational inequalities},
  author={Iusem, Alfredo N and Jofré, Alejandro and Thompson, Philip},
  journal={Mathematics of Operations Research},
  volume={44},
  number={1},
  pages={236--263},
  year={2019},
  publisher={INFORMS}
}

@article {MR3935081,
	AUTHOR = {Jofr\'e, Alejandro and Thompson, Philip},
	TITLE = {On variance reduction for stochastic smooth convex
	optimization with multiplicative noise},
	JOURNAL = {Math. Program.},
	FJOURNAL = {Mathematical Programming},
	VOLUME = {174},
	YEAR = {2019},
	NUMBER = {1-2},
	PAGES = {253--292},
	ISSN = {0025-5610,1436-4646},
	MRCLASS = {90C25 (62L20 65K05 68Q25 90C15)},
	MRNUMBER = {3935081},
	MRREVIEWER = {Wim\ van Ackooij},
	DOI = {10.1007/s10107-018-1297-x},
	URL = {https://doi.org/10.1007/s10107-018-1297-x},
}

@article {MR4902793,
    AUTHOR = {Lei, Jinlong and Shanbhag, Uday V.},
     TITLE = {Variance-reduced accelerated first-order methods: central
              limit theorems and confidence statements},
   JOURNAL = {Math. Oper. Res.},
  FJOURNAL = {Mathematics of Operations Research},
    VOLUME = {50},
      YEAR = {2025},
    NUMBER = {2},
     PAGES = {1364--1397},
      ISSN = {0364-765X,1526-5471},
   MRCLASS = {65K15 (60F05 90C15 90C25 90C33)},
  MRNUMBER = {4902793},
       DOI = {10.1287/moor.2021.0068},
       URL = {https://doi.org/10.1287/moor.2021.0068},
}

@article{geiersbach2021stochastic,
  title={Stochastic proximal gradient methods for nonconvex problems in {H}ilbert spaces},
  author={Geiersbach, Caroline and Scarinci, Teresa},
  journal={Comput. Optim. Appl.},
  volume={78},
  number={3},
  pages={705--740},
  year={2021},
  publisher={Springer}
}

@book {MR4703976,
    AUTHOR = {Baldi, Paolo},
     TITLE = {Probability---an introduction through theory and exercises},
    SERIES = {Universitext},
 PUBLISHER = {Springer, Cham},
      YEAR = {[2023] \copyright 2023},
     PAGES = {ix+389},
      ISBN = {978-3-031-38491-2; 978-3-031-38492-9},
   MRCLASS = {60-01},
  MRNUMBER = {4703976},
       DOI = {10.1007/978-3-031-38492-9},
       URL = {https://doi.org/10.1007/978-3-031-38492-9},
}

@article{atenas2023unified,
  title={A unified analysis of descent sequences in weakly convex optimization, including convergence rates for bundle methods},
  author={Atenas, Felipe and Sagastiz{\'a}bal, Claudia and Silva, Paulo JS and Solodov, Mikhail},
  journal={SIAM J. Optim.},
  volume={33},
  number={1},
  pages={89--115},
  year={2023},
  publisher={SIAM}
}

@article {MR1363357,
	AUTHOR = {Artstein, Zvi and Wets, Roger J.-B.},
	TITLE = {Consistency of minimizers and the {SLLN} for stochastic
	programs},
	JOURNAL = {J. Convex Anal.},
	FJOURNAL = {Journal of Convex Analysis},
	VOLUME = {2},
	YEAR = {1995},
	NUMBER = {1-2},
	PAGES = {1--17},
	ISSN = {0944-6532,2363-6394},
	MRCLASS = {90C15 (49J45)},
	MRNUMBER = {1363357},
	MRREVIEWER = {C.\ C. Y. Dorea},
}

@book {MR4362585,
	AUTHOR = {Shapiro, Alexander and Dentcheva, Darinka and Ruszczy\'nski,
	Andrzej},
	TITLE = {Lectures on stochastic programming---modeling and theory},
	SERIES = {MOS-SIAM Series on Optimization},
	VOLUME = {28},
	EDITION = {Third},
	PUBLISHER = {Society for Industrial and Applied Mathematics (SIAM),
	Philadelphia, PA; Mathematical Optimization Society,
	Philadelphia, PA},
	YEAR = {[2021] \copyright 2021},
	PAGES = {xv+525},
	ISBN = {978-1-611976-58-8},
	MRCLASS = {90-01 (90C15)},
	MRNUMBER = {4362585},
	DOI = {10.1137/1.9781611976595},
	URL = {https://doi.org/10.1137/1.9781611976595},
}

@book {MR1058436,
	AUTHOR = {Clarke, F. H.},
	TITLE = {Optimization and nonsmooth analysis},
	SERIES = {Classics in Applied Mathematics},
	VOLUME = {5},
	EDITION = {Second},
	PUBLISHER = {SIAM,
	Philadelphia, PA},
	YEAR = {1990},
	PAGES = {xii+308},
	ISBN = {0-89871-256-4},
	MRCLASS = {49-02 (01A75 49J52 58C20 90C48)},
	MRNUMBER = {1058436},
	DOI = {10.1137/1.9781611971309},
	URL = {https://doi.org/10.1137/1.9781611971309},
}

@book {MR1955649,
	AUTHOR = {Facchinei, F. and Pang, J.-S.},
	TITLE = {Finite-Dimensional Variational Inequalities and
	Complementarity Problems, {V}olume {II}},
	PUBLISHER = {Springer-Verlag, New York},
	YEAR = {2003},
	PAGES = {i-xxxiv, 625--1234 and II1--II57},
	ISBN = {0-387-95581-X},
	MRCLASS = {90-02 (15A39 49J40 65K10 90C33)},
	MRNUMBER = {1955649},
	MRREVIEWER = {Daniel Ralph},
}

@article {aragonartacho2023boosted,
    AUTHOR = {Arag\'{o}n-Artacho, Francisco J. and P\'{e}rez-Aros, Pedro and
              Torregrosa-Bel\'{e}n, David},
     TITLE = {The {B}oosted {D}ouble-proximal {S}ubgradient {A}lgorithm for
              nonconvex optimization},
   JOURNAL = {Math. Program.},
  FJOURNAL = {Mathematical Programming},
    VOLUME = {214},
      YEAR = {2025},
    NUMBER = {1-2, Ser. A},
     PAGES = {491--537},
      ISSN = {0025-5610,1436-4646},
   MRCLASS = {49J53 (65K05 90C26 90C30 90C46)},
  MRNUMBER = {4996113},
       DOI = {10.1007/s10107-024-02190-0},
       URL = {https://doi.org/10.1007/s10107-024-02190-0},
}

@book {MR3289054,
	AUTHOR = {Izmailov, Alexey F. and Solodov, Mikhail V.},
	TITLE = {Newton-Type Methods for Optimization and Variational Problems},
	PUBLISHER = {Springer, Cham},
	YEAR = {2014},
	PAGES = {xx+573},
	ISBN = {978-3-319-04246-6; 978-3-319-04247-3},
	MRCLASS = {49-02 (49M15 65-02 90-02 90C53 90C55)},
	MRNUMBER = {3289054},
	MRREVIEWER = {Christian Kanzow},
	DOI = {10.1007/978-3-319-04247-3},
	URL = {https://doi.org/10.1007/978-3-319-04247-3},
}

@book {MR1491362,
	AUTHOR = {Rockafellar, R. Tyrrell and Wets, Roger J.-B.},
	TITLE = {Variational analysis},
	SERIES = {Grundlehren der mathematischen Wissenschaften [Fundamental
	Principles of Mathematical Sciences]},
	VOLUME = {317},
	PUBLISHER = {Springer-Verlag, Berlin},
	YEAR = {1998},
	PAGES = {xiv+733},
	ISBN = {3-540-62772-3},
	MRCLASS = {49-02 (46N10 47N10 49J52 49K40 90C30)},
	MRNUMBER = {1491362},
	MRREVIEWER = {Francis\ H.\ Clarke},
	DOI = {10.1007/978-3-642-02431-3},
	URL = {https://doi.org/10.1007/978-3-642-02431-3},
}

@article{vanAckooij_Oliveira2025methods,
	title={Methods of Nonsmooth Optimization in Stochastic Programming},
	author={van Ackooij, Wim Stefanus and de Oliveira, Welington Luis},
	journal={International Series in Operations Research and Management Science},
	year={2025},
	publisher={Springer}
}

@book{bach2024learning,
  title={Learning theory from first principles},
  author={Bach, Francis},
  year={2024},
  publisher={MIT press}
}

@book{MR4436019,
	AUTHOR = {Royset, Johannes O. and Wets, Roger J.-B.},
	TITLE = {An optimization primer},
	SERIES = {Springer Series in Operations Research and Financial
	Engineering},
	PUBLISHER = {Springer, Cham},
	YEAR = {[2021] \copyright 2021},
	PAGES = {xviii+676},
	ISBN = {978-3-030-76274-2; 978-3-030-76275-9},
	MRCLASS = {90-01 (90Cxx)},
	MRNUMBER = {4436019},
	MRREVIEWER = {Giorgio\ Giorgi},
	DOI = {10.1007/978-3-030-76275-9},
	URL = {https://doi.org/10.1007/978-3-030-76275-9},
}

@article {MR4726003,
	AUTHOR = {Kr\"atschmer, V.},
	TITLE = {Nonasymptotic upper estimates for errors of the sample average
	approximation method to solve risk-averse stochastic programs},
	JOURNAL = {SIAM J. Optim.},
	FJOURNAL = {SIAM Journal on Optimization},
	VOLUME = {34},
	YEAR = {2024},
	NUMBER = {2},
	PAGES = {1264--1294},
	ISSN = {1052-6234,1095-7189},
	MRCLASS = {90C15 (60B12 60E15 91B05)},
	MRNUMBER = {4726003},
	DOI = {10.1137/22M1535425},
	URL = {https://doi.org/10.1137/22M1535425},
}

@article {MR1363364,
	AUTHOR = {Clarke, F. H. and Stern, R. J. and Wolenski, P. R.},
	TITLE = {Proximal smoothness and the lower-{$C^2$} property},
	JOURNAL = {J. Convex Anal.},
	FJOURNAL = {Journal of Convex Analysis},
	VOLUME = {2},
	YEAR = {1995},
	NUMBER = {1-2},
	PAGES = {117--144},
	ISSN = {0944-6532},
	MRCLASS = {49J52 (26B05 90C26)},
	MRNUMBER = {1363364},
	MRREVIEWER = {Philip D. Loewen},
}

@inproceedings {MR343355,
	AUTHOR = {Robbins, H. and Siegmund, D.},
	TITLE = {A convergence theorem for non negative almost supermartingales
	and some applications},
	BOOKTITLE = {Optimizing methods in statistics ({P}roc. {S}ympos., {O}hio
	{S}tate {U}niv., {C}olumbus, {O}hio, 1971)},
	PAGES = {233--257},
	PUBLISHER = {Academic Press, New York-London},
	YEAR = {1971},
	MRCLASS = {60G45 (62L20)},
	MRNUMBER = {343355},
	MRREVIEWER = {P. R\'{e}v\'{e}sz},
}

@book {bogachev,
	AUTHOR = {Bogachev, V. I.},
	TITLE = {Measure theory. {V}ol. {I}, {II}},
	PUBLISHER = {Springer-Verlag, Berlin},
	YEAR = {2007},
	PAGES = {Vol. I: xviii+500 pp., Vol. II: xiv+575},
	ISBN = {978-3-540-34513-8; 3-540-34513-2},
	MRCLASS = {28-02 (28Axx 28Cxx 46G12 60G42 60G44)},
	MRNUMBER = {2267655},
	MRREVIEWER = {Ren\'{e}\ L.\ Schilling},
	DOI = {10.1007/978-3-540-34514-5},
	URL = {https://doi.org/10.1007/978-3-540-34514-5},
}

@article{aragonartacho2025nonmonotonesubgradientmethodsbased,
	title={Nonmonotone subgradient methods based on a local descent lemma}, 
	author={Francisco J. Aragón-Artacho and Rubén Campoy and Pedro Pérez-Aros and David Torregrosa-Belén},
	year={2025},
	eprint={2510.19341},
	archivePrefix={arXiv},
	primaryClass={math.OC},
	note={Preprint, available at: \href{https://arxiv.org/abs/2510.19341}{arXiv:2510.19341}}, 
}

@article{perezaros2025randomizedblockproximalmethod,
      title={Randomized block proximal method with locally Lipschitz continuous gradient}, 
      author={Pedro Pérez-Aros and David Torregrosa-Belén},
      year={2025},
      eprint={2504.11410},
      archivePrefix={arXiv},
      primaryClass={math.OC},
      note={Preprint, available at: \href{https://arxiv.org/abs/2504.11410}{arXiv:2504.11410}}, 
}

@article{khanh2024fundamental,
	title={Fundamental convergence analysis of sharpness-aware minimization},
	author={Khanh, Pham and Luong, Hoang-Chau and Mordukhovich, Boris and Tran, Dat},
	journal={NeurIPS},
	volume={37},
	pages={13149--13182},
	year={2024}
}

@article {MR4605214,
	AUTHOR = {Li, Xiao and Milzarek, Andre and Qiu, Junwen},
	TITLE = {Convergence of random reshuffling under the {K}urdyka--{{\L}}ojasiewicz inequality},
	JOURNAL = {SIAM J. Optim.},
	FJOURNAL = {SIAM Journal on Optimization},
	VOLUME = {33},
	YEAR = {2023},
	NUMBER = {2},
	PAGES = {1092--1120},
	ISSN = {1052-6234,1095-7189},
	MRCLASS = {90C26 (90C06 90C15)},
	MRNUMBER = {4605214},
	MRREVIEWER = {Shouqiang\ Du},
	DOI = {10.1137/21M1468048},
	URL = {https://doi.org/10.1137/21M1468048},
}

@book {MR1764176,
    AUTHOR = {Graf, Siegfried and Luschgy, Harald},
     TITLE = {Foundations of quantization for probability distributions},
    SERIES = {Lecture Notes in Mathematics},
    VOLUME = {1730},
 PUBLISHER = {Springer-Verlag, Berlin},
      YEAR = {2000},
     PAGES = {x+230},
      ISBN = {3-540-67394-6},
   MRCLASS = {60E99 (60F25 62H05 94A12)},
  MRNUMBER = {1764176},
MRREVIEWER = {Kalev\ P\"{a}rna},
       DOI = {10.1007/BFb0103945},
       URL = {https://doi.org/10.1007/BFb0103945},
}

@book {MR982264,
    AUTHOR = {Dudley, Richard M.},
     TITLE = {Real analysis and probability},
    SERIES = {The Wadsworth \& Brooks/Cole Mathematics Series},
 PUBLISHER = {Wadsworth \& Brooks/Cole Advanced Books \& Software, Pacific
              Grove, CA},
      YEAR = {1989},
     PAGES = {xii+436},
      ISBN = {0-534-10050-3},
   MRCLASS = {60-01 (00A05 28-01 46-01)},
  MRNUMBER = {982264},
MRREVIEWER = {Evarist\ Gin\'{e}},
}

@article{frankel2015splitting,
  title={Splitting methods with variable metric for {K}urdyka--{{\L}}ojasiewicz functions and general convergence rates},
  author={Frankel, Pierre and Garrigos, Guillaume and Peypouquet, Juan},
  journal={J. Optim. Theory Appl.},
  volume={165},
  number={3},
  pages={874--900},
  year={2015},
  publisher={Springer}
}

@book {PolyakBook,
    AUTHOR = {Polyak, Boris T.},
     TITLE = {Introduction to optimization},
    SERIES = {Translations Series in Mathematics and Engineering},
      NOTE = {Translated from the Russian,
              With a foreword by Dimitri P. Bertsekas},
 PUBLISHER = {Optimization Software, Inc., Publications Division, New York},
      YEAR = {1987},
     PAGES = {xxvii+438},
      ISBN = {0-911575-14-6},
   MRCLASS = {49-01 (65Kxx 90Cxx)},
  MRNUMBER = {1099605},
}

@article {MR1644089,
    AUTHOR = {Kurdyka, Krzysztof},
     TITLE = {On gradients of functions definable in o-minimal structures},
   JOURNAL = {Ann. Inst. Fourier (Grenoble)},
  FJOURNAL = {Universit\'{e} de Grenoble. Annales de l'Institut Fourier},
    VOLUME = {48},
      YEAR = {1998},
    NUMBER = {3},
     PAGES = {769--783},
      ISSN = {0373-0956,1777-5310},
   MRCLASS = {03C65 (14P15 26D10 26E05)},
  MRNUMBER = {1644089},
MRREVIEWER = {A.\ J.\ Wilkie},
       URL = {http://www.numdam.org/item?id=AIF_1998__48_3_769_0},
}

@book{lojasiewicz1965EnsemblesS,
  title={Ensembles semi-analytiques},
  author={Stanisław {\L}ojasiewicz},
  year={1965},
  url={https://api.semanticscholar.org/CorpusID:118074635},
  publisher = {Institut des Hautes Etudes Scientifiques, Bures-sur-Yvette (Seine-et-Oise), France},
}

@article{POLYAK1963864,
title = {Gradient methods for the minimisation of functionals},
journal = { USSR Comput. Math. \& Math. Phys.},
volume = {3},
number = {4},
pages = {864-878},
year = {1963},
issn = {0041-5553},
doi = {https://doi.org/10.1016/0041-5553(63)90382-3},
url = {https://www.sciencedirect.com/science/article/pii/0041555363903823},
author = {B.T. Polyak}
}

@article{bento2025convergence,
  title={Convergence of descent optimization algorithms under {P}olyak-{{\L}}ojasiewicz-{K}urdyka conditions},
  author={Bento, Glaydston and Mordukhovich, Boris and Mota, Tiago and Nesterov, Yurii},
  journal={J.  Optim. Theory  Appl.},
  volume={207},
  number={3},
  pages={41},
  year={2025},
  publisher={Springer}
}

@article{bolte2007lojasiewicz,
  title={The {{\L}}ojasiewicz inequality for nonsmooth subanalytic functions with applications to subgradient dynamical systems},
  author={Bolte, J{\'e}r{\^o}me and Daniilidis, Aris and Lewis, Adrian},
  journal={SIAM J. Optim.},
  volume={17},
  number={4},
  pages={1205--1223},
  year={2007},
  publisher={SIAM}
}

@article{bolte2006nonsmooth,
  title={A nonsmooth {M}orse--{S}ard theorem for subanalytic functions},
  author={Bolte, J{\'e}r{\^o}me and Daniilidis, Aris and Lewis, Adrian},
  journal={J. Math. Anal. Appl.},
  volume={321},
  number={2},
  pages={729--740},
  year={2006},
  publisher={Elsevier}
}

@article{bolte2007clarke,
  title={Clarke subgradients of stratifiable functions},
  author={Bolte, J{\'e}r{\^o}me and Daniilidis, Aris and Lewis, Adrian and Shiota, Masahiro},
  journal={SIAM J.  Optim.},
  volume={18},
  number={2},
  pages={556--572},
  year={2007},
  publisher={SIAM}
}

@article{li2018calculus,
  title={Calculus of the exponent of {K}urdyka--{{\L}}ojasiewicz inequality and its applications to linear convergence of first-order methods},
  author={Li, Guoyin and Pong, Ting Kei},
  journal={Found. Comput. Math.},
  volume={18},
  number={5},
  pages={1199--1232},
  year={2018},
  publisher={Springer}
}

@article{aragon2018accelerating,
  title={Accelerating the {DC} algorithm for smooth functions},
  author={Arag{\'o}n Artacho, Francisco J and Fleming, Ronan MT and Vuong, Phan T},
  journal={Math. Program.},
  volume={169},
  number={1},
  pages={95--118},
  year={2018},
  publisher={Springer}
}

@article{RobbinsMonro1951,
	author  = {Robbins, Herbert and Monro, Sutton},
	title   = {A Stochastic Approximation Method},
	journal = {Ann. Math. Statist.},
	volume  = {22},
	number  = {3},
	pages   = {400--407},
	year    = {1951},
	doi     = {10.1214/aoms/1177729586}
}

@article{KieferWolfowitz1952,
	author  = {Kiefer, Jack and Wolfowitz, Jacob},
	title   = {Stochastic Estimation of the Maximum of a Regression Function},
	journal = {Ann. Math. Statist.},
	volume  = {23},
	number  = {3},
	pages   = {462--466},
	year    = {1952},
	doi     = {10.1214/aoms/1177729392}
}

@article{NemirovskiJuditskyLanShapiro2009,
  author  = {Nemirovski, Arkadi and Juditsky, Anatoli and Lan, Guanghui and Shapiro, Alexander},
  title   = {Robust stochastic approximation approach to stochastic programming},
  journal = {SIAM J.  Optim.},
  volume  = {19},
  number  = {4},
  pages   = {1574--1609},
  year    = {2009},
  doi     = {10.1137/070704277}
}

@article{BottouCurtisNocedal2018,
  author  = {Bottou, L{\'e}on and Curtis, Frank E. and Nocedal, Jorge},
  title   = {Optimization methods for large-scale machine learning},
  journal = {SIAM Review},
  volume  = {60},
  number  = {2},
  pages   = {223--311},
  year    = {2018},
  doi     = {10.1137/16M1080173}
}

@article{GhadimiLan2013,
  author  = {Ghadimi, Saeed and Lan, Guanghui},
  title   = {Stochastic first- and zeroth-order methods for nonconvex stochastic programming},
  journal = {SIAM J.  Optim.},
  volume  = {23},
  number  = {4},
  pages   = {2341--2368},
  year    = {2013},
  doi     = {10.1137/120880811}
}

@article{KleywegtShapiroHomemDeMello2002,
  author  = {Kleywegt, Anton J. and Shapiro, Alexander and Homem-de-Mello, Tito},
  title   = {The sample average approximation method for stochastic discrete optimization},
  journal = {SIAM J.  Optim.},
  volume  = {12},
  number  = {2},
  pages   = {479--502},
  year    = {2002},
  doi     = {10.1137/S1052623499363220}
}

@article{KingRockafellar1993,
  author  = {King, Alan J. and Rockafellar, R. Tyrrell},
  title   = {Asymptotic theory for solutions in statistical estimation and stochastic programming},
  journal = {Math. Oper. Res.},
  volume  = {18},
  number  = {1},
  pages   = {148--162},
  year    = {1993},
  doi     = {10.1287/moor.18.1.148}
}

@article{Shapiro1993,
  author  = {Shapiro, Alexander},
  title   = {Asymptotic behavior of optimal solutions in stochastic programming},
  journal = {Math. Oper. Res.},
  volume  = {18},
  number  = {4},
  pages   = {829--845},
  year    = {1993},
  doi     = {10.1287/moor.18.4.829}
}

@article{HomemDeMello2008,
  author  = {Homem-de-Mello, Tito},
  title   = {On rates of convergence for stochastic optimization problems under non-{IID} sampling},
  journal = {SIAM J.  Optim.},
  volume  = {19},
  number  = {2},
  pages   = {524--551},
  year    = {2008},
  doi     = {10.1137/060657418}
}

@article{DuchiRuan2018,
  author  = {Duchi, John C. and Ruan, Feng},
  title   = {Stochastic methods for composite and weakly convex optimization problems},
  journal = {SIAM J.  Optim.},
  volume  = {28},
  number  = {4},
  pages   = {3229--3259},
  year    = {2018},
  doi     = {10.1137/17M1135086}
}

@article{LeThiHuynhPhamDinhLuu2022,
  author  = {Le Thi, Hoai An and Huynh, Van Ngai and Pham Dinh, Tao and Luu, Hoang Phuc Hau},
  title   = {Stochastic difference-of-convex-functions algorithms for nonconvex programming},
  journal = {SIAM J.  Optim.},
  volume  = {32},
  number  = {3},
  pages   = {2263--2293},
  year    = {2022},
  doi     = {10.1137/20M1385706}
}

@article{PhamNguyenPhanTranDinh2020,
  author  = {Pham, Nhan H. and Nguyen, Lam M. and Phan, Dzung T. and Tran-Dinh, Quoc},
  title   = {{ProxSARAH}: An efficient algorithmic framework for stochastic composite nonconvex optimization},
  journal = {J. Mach. Learn. Res.},
  volume  = {21},
  number  = {110},
  pages   = {1--48},
  year    = {2020},
  url     = {https://jmlr.org/papers/v21/19-248.html}
}

@article{TranDinhPhamPhanNguyen2022,
  author  = {Tran-Dinh, Quoc and Pham, Nhan H. and Phan, Dzung T. and Nguyen, Lam M.},
  title   = {A hybrid stochastic optimization framework for composite nonconvex optimization},
  journal = {Math. Program.},
  volume  = {191},
  number  = {2},
  pages   = {1005--1071},
  year    = {2022},
  doi     = {10.1007/s10107-020-01583-1}
}

@article{Li2022,
  author  = {Li, Zhize},
  title   = {Simple and optimal stochastic gradient methods for nonsmooth nonconvex optimization},
  journal = {J. Mach. Learn. Res},
  volume  = {23},
  number  = {239},
  pages   = {1--61},
  year    = {2022},
  url     = {https://jmlr.org/papers/v23/21-0028.html}
}

@article{MetelTakeda2021,
  author  = {Metel, Michael R. and Takeda, Akiko},
  title   = {Stochastic proximal methods for non-smooth non-convex constrained sparse optimization},
  journal = {J. Mach. Learn. Res},
  volume  = {22},
  number  = {115},
  pages   = {1--36},
  year    = {2021},
  url     = {https://jmlr.org/papers/v22/20-287.html}
}

@article{MairalBachPonceSapiro2010,
  author  = {Mairal, Julien and Bach, Francis and Ponce, Jean and Sapiro, Guillermo},
  title   = {Online learning for matrix factorization and sparse coding},
  journal = {J. Mach. Learn. Res},
  volume  = {11},
  pages   = {19--60},
  year    = {2010},
  url     = {https://jmlr.org/papers/v11/mairal10a.html}
}

@article{FanLi2001,
  author  = {Fan, Jianqing and Li, Runze},
  title   = {Variable selection via nonconcave penalized likelihood and its oracle properties},
  journal = {J. Amer. Statist. Assoc.},
  volume  = {96},
  number  = {456},
  pages   = {1348--1360},
  year    = {2001},
  doi     = {10.1198/016214501753382273}
}

@article{Zhang2010MCP,
  author  = {Zhang, Cun-Hui},
  title   = {Nearly unbiased variable selection under minimax concave penalty},
  journal = {Ann. Statist.},
  volume  = {38},
  number  = {2},
  pages   = {894--942},
  year    = {2010},
  doi     = {10.1214/09-AOS729}
}

@article{LohWainwright2015,
  author  = {Loh, Po-Ling and Wainwright, Martin J.},
  title   = {Regularized {$M$}-estimators with nonconvexity: Statistical and algorithmic theory for local optima},
  journal = {J. Mach. Learn. Res},
  volume  = {16},
  pages   = {559--616},
  year    = {2015},
  url     = {https://jmlr.org/papers/v16/loh15a.html}
}

@article{AttouchBolte2009,
  author  = {Attouch, Hedy and Bolte, J{\'e}r{\^o}me},
  title   = {On the convergence of the proximal algorithm for nonsmooth functions involving analytic features},
  journal = {Math. Program.},
  volume  = {116},
  number  = {1--2},
  pages   = {5--16},
  year    = {2009},
  doi     = {10.1007/s10107-007-0133-5}
}

@article{AttouchBolteRedontSoubeyran2010,
  author  = {Attouch, Hedy and Bolte, J{\'e}r{\^o}me and Redont, Patrick and Soubeyran, Antoine},
  title   = {Proximal alternating minimization and projection methods for nonconvex problems: An approach based on the {Kurdyka--{\L}ojasiewicz} inequality},
  journal = {Math. Oper. Res.},
  volume  = {35},
  number  = {2},
  pages   = {438--457},
  year    = {2010},
  doi     = {10.1287/moor.1100.0449}
}

@article{AttouchBolteSvaiter2013,
  author  = {Attouch, Hedy and Bolte, J{\'e}r{\^o}me and Svaiter, Benar Fux},
  title   = {Convergence of descent methods for semi-algebraic and tame problems: proximal algorithms, forward--backward splitting, and regularized {Gauss--Seidel} methods},
  journal = {Math. Program.},
  volume  = {137},
  number  = {1--2},
  pages   = {91--129},
  year    = {2013},
  doi     = {10.1007/s10107-011-0484-9}
}

@article{BolteSabachTeboulle2014,
  author  = {Bolte, J{\'e}r{\^o}me and Sabach, Shoham and Teboulle, Marc},
  title   = {Proximal alternating linearized minimization for nonconvex and nonsmooth problems},
  journal = {Math. Program.},
  volume  = {146},
  number  = {1--2},
  pages   = {459--494},
  year    = {2014},
  doi     = {10.1007/s10107-013-0701-9}
}

@article{JMLR:v12:cour11a,
  author  = {Timothee Cour and Ben Sapp and Ben Taskar},
  title   = {Learning from Partial Labels},
  journal = {J. Mach. Learn. Res.h},
  year    = {2011},
  volume  = {12},
  number  = {42},
  pages   = {1501-1536},
  url     = {http://jmlr.org/papers/v12/cour11a.html}
}

@inproceedings{NIPS2010_c9e1074f,
 author = {Luo, Jie and Orabona, Francesco},
 booktitle = {Adv. Neural Inf. Process. Syst.},
 editor = {J. Lafferty and C. Williams and J. Shawe-Taylor and R. Zemel and A. Culotta},
 pages = {},
 publisher = {Curran Associates, Inc.},
 title = {Learning from Candidate Labeling Sets},
 url = {https://proceedings.neurips.cc/paper_files/paper/2010/file/c9e1074f5b3f9fc8ea15d152add07294-Paper.pdf},
 volume = {23},
 year = {2010}
}

@INPROCEEDINGS{6618941,
  author={Zeng, Zinan and Xiao, Shijie and Jia, Kui and Chan, Tsung-Han and Gao, Shenghua and Xu, Dong and Ma, Yi},
  booktitle={2013 IEEE Conf. Comput. Vis. Pattern Recognit.}, 
  title={Learning by Associating Ambiguously Labeled Images}, 
  year={2013},
  volume={},
  number={},
  pages={708-715},
  doi={10.1109/CVPR.2013.97}}

@ARTICLE{7968363,
  author={Chen, Ching-Hui and Patel, Vishal M. and Chellappa, Rama},
  journal={IEEE Trans. Pattern Anal. Mach. Intell.}, 
  title={Learning from Ambiguously Labeled Face Images}, 
  year={2018},
  volume={40},
  number={7},
  pages={1653-1667},
  doi={10.1109/TPAMI.2017.2723401}}

@article{Cheng_Wang_Feng_Zhang_An_2023, title={Partial-Label Regression}, volume={37}, url={https://ojs.aaai.org/index.php/AAAI/article/view/25871}, DOI={10.1609/aaai.v37i6.25871}, number={6}, journal={Proc. AAAI Conf. Artif. Intell.}, author={Cheng, Xin and Wang, Deng-Bao and Feng, Lei and Zhang, Min-Ling and An, Bo}, year={2023}, month={Jun.}, pages={7140–7147} }

@article {MR1167814,
    AUTHOR = {Polyak, B. T. and Juditsky, A. B.},
     TITLE = {Acceleration of stochastic approximation by averaging},
   JOURNAL = {SIAM J. Control Optim.},
  FJOURNAL = {SIAM Journal on Control and Optimization},
    VOLUME = {30},
      YEAR = {1992},
    NUMBER = {4},
     PAGES = {838--855},
      ISSN = {0363-0129},
   MRCLASS = {62L20 (93E25)},
  MRNUMBER = {1167814},
MRREVIEWER = {George\ Yin},
       DOI = {10.1137/0330046},
       URL = {https://doi.org/10.1137/0330046},
}

@article {MR2921104,
    AUTHOR = {Lan, Guanghui},
     TITLE = {An optimal method for stochastic composite optimization},
   JOURNAL = {Math. Program.},
  FJOURNAL = {Mathematical Programming},
    VOLUME = {133},
      YEAR = {2012},
    NUMBER = {1-2, Ser. A},
     PAGES = {365--397},
      ISSN = {0025-5610,1436-4646},
   MRCLASS = {62L20 (68Q25 90C15)},
  MRNUMBER = {2921104},
       DOI = {10.1007/s10107-010-0434-y},
       URL = {https://doi.org/10.1007/s10107-010-0434-y},
}

@article {MR64365,
    AUTHOR = {Chung, K. L.},
     TITLE = {On a stochastic approximation method},
   JOURNAL = {Ann. Math. Statistics},
  FJOURNAL = {Annals of Mathematical Statistics},
    VOLUME = {25},
      YEAR = {1954},
     PAGES = {463--483},
      ISSN = {0003-4851},
   MRCLASS = {62.0X},
  MRNUMBER = {64365},
MRREVIEWER = {J.\ Wolfowitz},
       DOI = {10.1214/aoms/1177728716},
       URL = {https://doi.org/10.1214/aoms/1177728716},
}

@article {MR4261271,
    AUTHOR = {Correa, Rafael and Hantoute, Abderrahim and P\'erez-Aros,
              Pedro},
     TITLE = {Qualification conditions-free characterizations of the
              {$\varepsilon$}-subdifferential of convex integral functions},
   JOURNAL = {Appl. Math. Optim.},
  FJOURNAL = {Applied Mathematics and Optimization},
    VOLUME = {83},
      YEAR = {2021},
    NUMBER = {3},
     PAGES = {1709--1737},
      ISSN = {0095-4616,1432-0606},
   MRCLASS = {49J52 (49J45)},
  MRNUMBER = {4261271},
MRREVIEWER = {Mira\ Bivas},
       DOI = {10.1007/s00245-019-09604-y},
       URL = {https://doi.org/10.1007/s00245-019-09604-y},
}

@article {MR191071,
    AUTHOR = {Armijo, Larry},
     TITLE = {Minimization of functions having {L}ipschitz continuous first
              partial derivatives},
   JOURNAL = {Pacific J. Math.},
  FJOURNAL = {Pacific Journal of Mathematics},
    VOLUME = {16},
      YEAR = {1966},
     PAGES = {1--3},
      ISSN = {0030-8730,1945-5844},
   MRCLASS = {65.10 (65.30)},
  MRNUMBER = {191071},
MRREVIEWER = {M.\ Lotkin},
       URL = {http://projecteuclid.org/euclid.pjm/1102995080},
}

@article {MR3548876,
    AUTHOR = {Bayram, {\.I}lker},
     TITLE = {On the convergence of the iterative shrinkage/thresholding
              algorithm with a weakly convex penalty},
   JOURNAL = {IEEE Trans. Signal Process.},
  FJOURNAL = {IEEE Transactions on Signal Processing},
    VOLUME = {64},
      YEAR = {2016},
    NUMBER = {6},
     PAGES = {1597--1608},
      ISSN = {1053-587X,1941-0476},
   MRCLASS = {94A12},
  MRNUMBER = {3548876},
MRREVIEWER = {Mohamed\ Ali\ El-Gebeily},
       DOI = {10.1109/TSP.2015.2502551},
       URL = {https://doi.org/10.1109/TSP.2015.2502551},
}

@article {MR1099305,
    AUTHOR = {Ermoliev, Yu.\ M. and Norkin, V. I.},
     TITLE = {Normalized convergence of random variables and its
              applications},
   JOURNAL = {Kibernetika (Kiev)},
  FJOURNAL = {Otdelenie Matematiki, Mekhaniki i Kibernetiki Akademii Nauk
              Ukrainsko\u i\ SSR. Kibernetika},
      YEAR = {1990},
    NUMBER = {6},
     PAGES = {85--89, 134},
      ISSN = {0023-1274},
   MRCLASS = {60F05 (60B10)},
  MRNUMBER = {1099305},
       DOI = {10.1007/BF01069497},
       URL = {https://doi.org/10.1007/BF01069497},
}

@article {MR3129765,
    AUTHOR = {Ermoliev, Yuri M. and Norkin, Vladimir I.},
     TITLE = {Sample average approximation method for compound stochastic
              optimization problems},
   JOURNAL = {SIAM J. Optim.},
  FJOURNAL = {SIAM Journal on Optimization},
    VOLUME = {23},
      YEAR = {2013},
    NUMBER = {4},
     PAGES = {2231--2263},
      ISSN = {1052-6234,1095-7189},
   MRCLASS = {90C15 (60B10 60B12)},
  MRNUMBER = {3129765},
MRREVIEWER = {A.\ H.\ \v Zilinskas},
       DOI = {10.1137/120863277},
       URL = {https://doi.org/10.1137/120863277},
}

 \appendix
 \section{Appendix}\label{s:appendix}

 \subsection{Proof of Lemma~\ref{l:bound_sample01}} \label{proof-l:bound_sample01}


 \begin{proofoflemma}{l:bound_sample01} Since $0 < n_k \leq n_{k+1}$, then $n_k\sqrt{n_{k+1}} \leq \sqrt{n_k}n_{k+1}$, and thus \[\frac{n_{k+1}-n_k}{\sqrt{n_k}n_{k+1}} \geq \frac{n_{k+1}-n_k}{\sqrt{n_k}n_{k+1} + n_k\sqrt{n_{k+1}}} \geq \frac{n_{k+1}-n_k}{2\sqrt{n_k}n_{k+1}}. \]
For all $k\in\N$, let \[p_k = \frac{n_{k+1}-n_k}{\sqrt{n_k}n_{k+1}} \mbox{~and~} q_k = \frac{n_{k+1}-n_k}{\sqrt{n_k}n_{k+1} + n_k\sqrt{n_{k+1}}}.\]The estimate above then yields $q_k \leq p_k \leq 2q_k$. Therefore, it suffices to verify that $\sum_{k\in\N} q_k$ converges for $\sum_{k\in\N} p_k$ to converge. Note that \[q_k = \frac{n_{k+1}-n_k}{\sqrt{n_k}\sqrt{n_{k+1}}(\sqrt{n_{k+1}} + \sqrt{n_k})} = \frac{\sqrt{n_{k+1}} - \sqrt{n_k}}{\sqrt{n_k}\sqrt{n_{k+1}}} = \frac{1}{\sqrt{n_k}} - \frac{1}{\sqrt{n_{k+1}}}. \] Since $n_{k+1} \to +\infty$, then \[\sum_{k\in\N} q_k = \frac{1}{\sqrt{n_1}} , \] and thus $\sum_{k\in\N} p_k < +\infty$. 
\end{proofoflemma}

\subsection{Technical result for Theorem~\ref{th:rates}} \label{proof-l-rates:tech-result-1}


\begin{proofoflemma}{B-ii} We first show that for all $k \geq 3$, \begin{equation*}
     \sum_{t=k+1}^\infty \frac{\ln(n_t)}{ {\sqrt{n_t} }} = \mathcal{O} \left(   \frac{\ln(k)}{k^{\expo/2-1}}\right).
 \end{equation*} By~\eqref{eq:condition_series_conv}, there exists $C_1 >0$ such that for all $k \in \N$, $n_k \geq C_1 k^\expo$. 
Then, the result for $b_k$ follows directly by recalling~\eqref{conv:gradients} and that $t \in (e^{2/\expo},+\infty) \mapsto \frac{\ln(t)}{t^{\expo/2}}$ is nonincreasing.  Using again that the latter function is nonincreasing and the integration by parts formula, it follows that for all $k > e \; (> e^{2/\expo})$,
\begin{align*}
    {\sum_{t=k+1}^\infty \frac{\ln(n_t)}{ {\sqrt{n_t} }} } & \leq  \sum_{t=k+1}^\infty \frac{\ln(C_1) + \gamma\ln(t)}{\sqrt{C_1}t^{\expo/2}} \leq \int_{k}^{\infty}
            \frac{\max\{\ln(C_1),0\} + \expo\ln(t)}{\sqrt{C_1} t^{\expo/2}}dt\\
            & = \frac{1}{\sqrt{C_1}}\left[ \frac{\max\{\ln(C_1),0\}}{(\frac{\expo}{2}-1)k^{\expo/2-1}} + \frac{\expo}{\frac{\expo}{2}-1}\left(\frac{\ln(k)}{k^{\expo/2-1}} + \frac{1}{\frac{\expo}{2}-1}\frac{1}{k^{\expo/2-1}}\right)\right] \\ & \leq   C_\expo\frac{\ln(k)}{k^{\expo/2-1}},
\end{align*} with $C_\expo = \frac{1}{\sqrt{C_1}}\left(\frac{\max\{\ln(C_1),0\}}{\frac{\expo}{2}-1}+\frac{\expo^2}{2(\frac{\expo}{2}-1)^2}\right)$, from where the first claim follows, as well as  \eqref{eq:n_k-series-condition}. 
This same estimate and the fact that  $\theta'$ is nonincreasing imply
\begin{align*}
   \left[ \theta'\left( {\sum_{t=k+1}^\infty \frac{\ln(n_t)}{ {\sqrt{n_t} }}     } \right)\right]^{-1}  \leq  \left[  \theta'\left(C_\expo \frac{\ln(k)}{k^{\expo/2-1}}\right)\right]^{-1}  =  \frac{C_\expo^\beta}{M(1-\beta)} \frac{\ln^\beta(k)}{k^{\beta(\expo/2-1)}},
\end{align*}
for sufficiently large $k$. Combining this last inequality with \eqref{eq:uk} and \eqref{eq01PLK} yields the claim for $\left[ \theta'\left( u_{k+1}\right)\right]^{-1}$.  Next, increasing $p_0$ if necessary, it holds
\begin{align*}
    \sum_{k=p_0}^\infty \left[  \theta'\left( {\sum_{t=k}^\infty \frac{\ln(n_t)}{\sqrt{n_t} }     } \right)\right]^{-1} & \leq 
    \frac{C_\expo^\beta}{M(1-\beta)}\sum_{k=p_0}^\infty \frac{\ln^\beta(k-1)}{(k-1)^{\beta(\expo/2-1)}}.
\end{align*} Since $\beta(\frac{\expo}{2}-1)>1$, then the series converges, and \eqref{desing_condi} holds. 

Finally,  similarly to above, using that $ t \in (e^{\beta^{-1}(\expo/2-1)^{-1}},+\infty) \mapsto \frac{\ln(t)}{t^{\beta(\expo/2-1)}} $ is nonincreasing  and integration by parts to bound the series, \begin{align*}
    \sum_{j=k}^\infty(b_{j} + \theta^\prime(u_{j+1})^{-1}) 
    & = \mathcal{O}\left( \sum_{j=k}^\infty\frac{\ln^{\beta}(j)}{j^{\beta(\expo/2-1)}} \right) 
     =   \mathcal{O}\left( \sum_{j=k}^\infty\frac{\ln(j)}{j^{\beta(\expo/2-1)}} \right) \\
    & =  \mathcal{O}\left( \int_{k-1}^\infty \frac{\ln(t)}{t^{\beta(\expo/2-1)}}dt \right)   =   \mathcal{O}\left( \frac{\ln(k-1)}{(k-1)^{\beta(\expo/2-1)-1}}  \right), 
\end{align*} proving the last claim in (ii), since $\frac{\ln(k-1)}{(k-1)^{\beta(\expo/2-1)-1}} \sim \frac{\ln(k)}{k^{\beta(\expo/2-1)-1}}$ asymptotically.
\end{proofoflemma}

\subsection{Proof of Lemma~\ref{lemma:sequences}} \label{appendix:sequences}

In order to prove Lemma~\ref{lemma:sequences}, we require two preliminary results, namely, Lemmas~\ref{lemma:A2} and~\ref{lemmaA1}.
 
 \begin{lemma}\label{lemma:A2}
Let $c>0$, $p\in\mathbb{R}$, $\expo_1>0$, and $\expo_2<3\expo_1+1$. For $k$ large enough, define
\[
    a_k
    :=
    \left(1+\frac1k\right)^p
    \frac{\ln^{\expo_2}(k)}
    {\ln^{\expo_1}(k+1)\big(c\ln^{\expo_1}(k)-p\big)} .
\]
Then
\[
    \lim_{k\to\infty}
    \frac{k}{c\ln^{\expo_1}(k)-p}
    \left(a_k-a_{k+1}\right)
    =
    0 .
\]
\end{lemma}

\begin{proof}
Define, for $x$ large enough,
\[
    g(x):=
    \left(1+\frac1x\right)^p
    \frac{\ln^{\expo_2}(x)}
    {\ln^{\expo_1}(x+1)\big(c\ln^{\expo_1}(x)-p\big)} .
\]
Then $a_k=g(k)$. Since $\expo_1>0$ and $c>0$, we have
\[
    |g(x)|
    =
    \mathcal{O}\left(\ln^{\expo_2-2\expo_1}(x)\right).
\]
Moreover, differentiating logarithmically gives
\[
    \frac{g'(x)}{g(x)}
    =
    -\frac{p}{x^2\left(1+\frac1x\right)}
    +\frac{\expo_2}{x\ln(x)}
    -\frac{\expo_1}{(x+1)\ln(x+1)}
    -\frac{c\expo_1\ln^{\expo_1-1}(x)}
    {x\big(c\ln^{\expo_1}(x)-p\big)} .
\]
Hence,
\[
    \frac{g'(x)}{g(x)}
    =
    \mathcal{O}\left(\frac1{x\ln(x)}\right),
\]
and consequently
\[
    |g'(x)|
    =
    \mathcal{O}\left(
        \frac{\ln^{\expo_2-2\expo_1-1}(x)}{x}
    \right).
\]
Therefore, there exist constants $c_0>0$ and $x_0>0$ such that
\[
    |g'(x)|
    \leq
    c_0\psi_0(x),
    \qquad x\geq x_0,
\]
where
\[
    \psi_0(x):=
    \frac{\ln^{\expo_2-2\expo_1-1}(x)}{x}.
\]
Since $\psi_0$ is nonincreasing for all sufficiently large $x$, we may enlarge $x_0$ if necessary so that $\psi_0$ is nonincreasing on $[x_0,\infty)$.

By the mean value theorem, for every $k$ large enough, there exists $\xi_k\in(k,k+1)$ such that
\[
    a_k-a_{k+1}
    =
    g(k)-g(k+1)
    =
    -g'(\xi_k).
\]
Thus, since $\xi_k\geq k$ and $\psi_0$ is nonincreasing,
\[
    |a_k-a_{k+1}|
    \leq
    c_0\psi_0(\xi_k)
    \leq
    c_0\psi_0(k)
    =
    c_0
    \frac{\ln^{\expo_2-2\expo_1-1}(k)}{k}.
\]
It follows that
\[
    \left|
    \frac{k}{c\ln^{\expo_1}(k)-p}
    \left(a_k-a_{k+1}\right)
    \right|
    \leq
    \frac{c_0\ln^{\expo_2-2\expo_1-1}(k)}
    {c\ln^{\expo_1}(k)-p}.
\]
Since $c\ln^{\expo_1}(k)-p\sim c\ln^{\expo_1}(k)$, we obtain
\[
    \left|
    \frac{k}{c\ln^{\expo_1}(k)-p}
    \left(a_k-a_{k+1}\right)
    \right|
    =
    \mathcal{O}\left(
        \ln^{\expo_2-3\expo_1-1}(k)
    \right).
\]
Finally, $\expo_2<3\expo_1+1$ implies $\expo_2-3\expo_1-1<0$. Hence
\[
    \ln^{\expo_2-3\expo_1-1}(k)\to 0,
\]
and therefore
\[
    \lim_{k\to\infty}
    \frac{k}{c\ln^{\expo_1}(k)-p}
    \left(a_k-a_{k+1}\right)
    =
    0,
\]
which concludes the proof.
\end{proof}

\begin{lemma}\label{lemmaA1}
Let $(u_k)_{k\geq 2}$ be a nonnegative sequence, let   $\expo_1>0$, and $\expo_2<3\expo_1+1$, and assume that
\begin{equation}\label{eq_lemmaA1_00}
    u_{k+1}
    \leq
    \left(1-c\frac{\ln^{\expo_1}(k)}{k}\right)u_k
    +
    d\frac{\ln^{\expo_2}(k)}{k^{p+1}},
    \qquad d>0,\quad p>0,\quad c>0 .
\end{equation}
Then
\begin{equation}\label{eq_lemmaA1_01}
    u_k
    =
    \mathcal{O}\left(
        \frac{\ln^{\expo_2 - \expo_1}(k)}{k^p}
    \right) + o\left( \frac{\ln^{\expo_1}(k)}{k^p} \right).
\end{equation}
\end{lemma}

\begin{proof}
Assume that, \eqref{eq_lemmaA1_00} holds for every $k\geq 1$.  
Define
\[
    w_k:=\frac{k^p}{\ln^{\expo_1}(k)}u_k,
    \qquad k\geq 2 .
\]
Multiplying the recurrence by $(k+1)^p/\ln^{\expo_1}(k+1)$, we obtain
\begin{align*}
    w_{k+1}
    &\leq
    \left(1+\frac1k\right)^p
    \frac{\ln^{\expo_1}(k)}{\ln^{\expo_1}(k+1)}
    \left(
        1-c\frac{\ln^{\expo_1}(k)}{k}
    \right)w_k \\
    &\quad
    +
    d\left(1+\frac1k\right)^p
    \frac{\ln^{\expo_2}(k)}{k\ln^{\expo_1}(k+1)} .
\end{align*}
Since $\ln(k)/\ln(k+1)\leq 1$,  for large enough $k \in \N$ we have
\[
    \left(1+\frac1k\right)^p
    \frac{\ln^{\expo_1}(k)}{\ln^{\expo_1}(k+1)}
    \left(
        1-c\frac{\ln^{\expo_1}(k)}{k}
    \right)
    \leq
    \left(1+\frac1k\right)^p
    \left(
        1-c\frac{\ln^{\expo_1}(k)}{k}
    \right).
\] 
Moreover,
\[
    \left(1+\frac1k\right)^p
    =
    1+\frac{p}{k}
    +
    \mathcal{O}\left(\frac1{k^2}\right).
\]
Thus,
\[
    w_{k+1}
    \leq
    \left(
        1-\frac{c\ln^{\expo_1}(k)-p}{k}
        +
        \mathcal{O}\left(\frac{\ln^{\expo_1}(k)}{k^2}\right)
    \right)w_k
    +
    d\left(1+\frac1k\right)^p
    \frac{\ln^{\expo_2}(k)}{k\ln^{\expo_1}(k+1)} .
\]

Now define
\[
    v_k:=w_k-a_k,
\]
where
\[
    a_k
    :=
    d\left(1+\frac1k\right)^p
    \frac{\ln^{\expo_2}(k)}
    {\ln^{\expo_1}(k+1)\big(c\ln^{\expo_1}(k)-p\big)} .
\]
For $k$ large enough, $c\ln^{\expo_1}(k)-p>0$, so $a_k$ is well defined. This choice gives
\[
    \frac{c\ln^{\expo_1}(k)-p}{k}a_k
    =
    d\left(1+\frac1k\right)^p
    \frac{\ln^{\expo_2}(k)}{k\ln^{\expo_1}(k+1)} .
\]
Using $w_k=v_k+a_k$, we get
\begin{align*}
    v_{k+1}
    &=w_{k+1}-a_{k+1} \\
    &\leq
    \left(
        1-\frac{c\ln^{\expo_1}(k)-p}{k}
        +
        \mathcal{O}\left(\frac{\ln^{\expo_1}(k)}{k^2}\right)
    \right)(v_k+a_k) \\
    &\quad
    +
    \frac{c\ln^{\expo_1}(k)-p}{k}a_k
    -a_{k+1} \\
    &=
    \left(
        1-\frac{c\ln^{\expo_1}(k)-p}{k}
        +
        \mathcal{O}\left(\frac{\ln^{\expo_1}(k)}{k^2}\right)
    \right)v_k
    +a_k-a_{k+1} \\
    &\quad
    +
    \mathcal{O}\left(\frac{\ln^{\expo_1}(k)}{k^2}\right)a_k .
\end{align*}
Since
\[
    a_k
    =
    \mathcal{O}\left(\ln^{\expo_2-2\expo_1}(k)\right),
\]
we have
\[
    \mathcal{O}\left(\frac{\ln^{\expo_1}(k)}{k^2}\right)a_k
    =
    \mathcal{O}\left(
        \frac{\ln^{\expo_2-\expo_1}(k)}{k^2}
    \right).
\]
Therefore,
\[
    v_{k+1}
    \leq
    \left(
        1-\frac{c\ln^{\expo_1}(k)-p}{k}
        +
        \mathcal{O}\left(\frac{\ln^{\expo_1}(k)}{k^2}\right)
    \right)v_k
    +a_k-a_{k+1}
    +
    \mathcal{O}\left(
        \frac{\ln^{\expo_2-\expo_1}(k)}{k^2}
    \right).
\] 
Finally, by Lemma~\ref{lemma:A2} we have
\[
    \lim_{k\to\infty}
    \frac{k}{c\ln^{\expo_1}(k)-p}
    \left(a_k-a_{k+1}\right)
    =
    0,
\]
 so we can apply  \cite[Lemma 3, p. 45]{PolyakBook} to conclude that 
$$\limsup_{  k \to \infty} v_k \leq 0,  $$
which shows that  \eqref{eq_lemmaA1_01} holds.
\end{proof}

 We now use Lemmas~\ref{lemma:A2} and~\ref{lemmaA1} to prove Lemma~\ref{lemma:sequences}.
 
 \begin{proofoflemma}{lemma:sequences} 
    First, let us simply assume that 
    $$ w_k \leq c_0  \frac{\ln^{p_1}(k)}{k^{p_2}}, \quad \text{ for all }k \geq 2.$$ We prove (i) first. Let $b\in(0,1)$ be arbitrary. Since $q\in(0,1)$ and $\alpha_k \to 0^+$  we have
\[
\alpha_k^{\,q-1}\to +\infty .
\]
Therefore, there exists $k_0\in\mathbb N$ such that for all $k\ge k_0$,
\[
\alpha_{k+1}\le b\left(\alpha_k+\frac{\ln^{p_1}(k)}{k^{p_2}}\right).
\] Fix $k\ge k_0$. Iterating the previous inequality from $k_0$ up to $k-1$, we obtain\begin{align*}
    \alpha_k
&\le
b^{k-k_0}\alpha_{k_0}
+\sum_{j=k_0}^{k-1} b^{k-j}\frac{\ln^{p_1}(j)}{j^p}  = b^{k-k_0}\alpha_{k_0}
+\sum_{i=1}^{k-k_0} b^i \frac{\ln^{p_1}(k-i)}{(k-i)^{p_2}}\\
& \le b^{k-k_0}\alpha_{k_0}
+\sum_{i=1}^{k-1} b^i \frac{\ln^{p_1}(k-i)}{(k-i)^{p_2}},
\end{align*} where in the first equality we set $i=k-j$, and in the second inequality we extend the sum up to $k-1$, since the summands are nonnegative. Now, assume $k\ge 2$ and split the sum as
\begin{equation}\label{sumbi}
\sum_{i=1}^{k-1} b^i \frac{\ln^{p_1}(k-i)}{(k-i)^{p_2}}
=
\sum_{i=1}^{\lfloor k/2\rfloor} b^i \frac{\ln^{p_1}(k-i)}{(k-i)^{p_2}}
+
\sum_{i=\lfloor k/2\rfloor+1}^{k-1}b^i \frac{\ln^{p_1}(k-i)}{(k-i)^{p_2}} ,
\end{equation}
where $\lfloor k/2\rfloor$ stands for the largest integer smaller than $k/2$. For the first term in the right-hand side of~\eqref{sumbi}, since $k-i\ge k/2$ whenever $1\le i\le \lfloor k/2\rfloor$ and the logarithm is strictly increasing, we get
\[
\sum_{i=1}^{\lfloor k/2\rfloor} b^i \frac{\ln^{p_1}(k-i)}{(k-i)^{p_2}}
\le
2^{p_2}  \frac{\ln^{p_1}(k)}{k^{p_2}} \sum_{i=1}^{\infty} b^i
=
\left(\frac{2^{p_2} b}{1-b}\right)\frac{\ln^{p_1}(k)}{k^{p_2}}.
\]
For the second term in the right-hand side of~\eqref{sumbi}, since $(k-i)^{-p_2}\le 1$, we have
\[
\sum_{i=\lfloor k/2\rfloor+1}^{k-1}  b^i \frac{\ln^{p_1}(k-i)}{(k-i)^{p_2}}
\le
\ln^{p_1}(k) \sum_{i=\lfloor k/2\rfloor+1}^{\infty} b^i
=
\frac{b^{\lfloor k/2\rfloor+1}}{1-b} \ln^{p_1}(k).
\]
Consequently,
\[
\alpha_k
\le
b^{k-k_0}\alpha_{k_0}
+
\left(\frac{2^{p_2} b}{1-b}\right)\frac{\ln^{p_1}(k)}{k^{p_2}}
+
\frac{b^{\lfloor k/2\rfloor+1}}{1-b} \ln^{p_1}(k).
\]
Multiplying the above inequality  by $k^{p_2} \ln^{-p_1}(k) $, we obtain
\[
\alpha_k \frac{k^{p_2}}{\ln^{p_1}{(k)}}
\le
\alpha_{k_0} b^{k-k_0}  \frac{k^{p_2}}{\ln^{p_1}{(k)}}
+
\frac{2^{p_2} b}{1-b}
+
\frac{ b^{\lfloor k/2\rfloor+1}}{1-b} k^{p_2}.
\]
Since exponential decay dominates any power, both
\[
 b^{k-k_0}  \frac{k^{p_2}}{\ln^{p_1}{(k)}}\to 0
\qquad\text{and}\qquad
b^{\lfloor k/2\rfloor+1} k^{p_2} \to 0.
\]
Hence
\[
\limsup_{k\to\infty} \alpha_k \frac{k^{p_2}}{\ln^{p_1}{(k)}}
\le
\frac{2^{p_2} b}{1-b}.
\]
Because $b\in(0,1)$ was arbitrary, letting $b\downarrow 0$ yields
\[
\limsup_{k\to\infty}\, \alpha_k \frac{k^{p_2}}{\ln^{p_1}{(k)}} =0,
\]
that is,
\[
\alpha_k=o\left(\frac{\ln^{p_1}(k)}{k^{p_2}}\right),
\]
which proves (i). 

The proof of (ii) is identical, except that in this case the recurrence becomes
\[
\alpha_{k+1}\le \frac{c}{1+c}  \alpha_k+  \left(\frac{c_0}{1+c} \right) \frac{\ln^{p_1}(k)}{k^{p_2}}.
\]
Repeating the same argument gives
\[
\limsup_{k\to\infty}\, \alpha_k \frac{k^{p_2}}{\ln^{p_1}{(k)}} 
\le   c_0 2^{p_2},
\]
which concludes 
$\alpha_k=\mathcal O\bigl(k^{-p_2}\, \ln^{p_1}(k)\bigr)$, as claimed.

For item (iii), since $q>1$, the function $h_q(t)=t^q$ is convex on $[0,+\infty)$. Hence, for every $\epsilon>0$,
\[
\alpha_{k+1}^q
\geq
\epsilon^q+q\epsilon^{q-1}(\alpha_{k+1}-\epsilon),
\]
that is,
\[
\alpha_{k+1}^q-\epsilon^q
\geq
q\epsilon^{q-1}(\alpha_{k+1}-\epsilon).
\]
Combining this estimate with \eqref{eq:recursion-lemma}, we obtain
\begin{align}\label{lemma:sequences:eq00}
(c+q\epsilon^{q-1})\alpha_{k+1}
\leq
c\alpha_k+q\epsilon^q+ c_0\frac{\ln^{p_1} (k) }{k^{p_2}}.
\end{align}Now, take  
\[
\epsilon:= \frac{\ln^{\frac{\expo_1}{q-1}}(k) }{ k^{\frac{1}{q-1}}}, \text{ with } \expo_1 \in (0,1).
\]
Then $\epsilon^q=  k^{-\frac{q}{q-1}}\ln^{\expo_1 \frac{q}{q-1} } (k)$ and $\epsilon^{q-1}= k^{-1}\ln^{\expo_1} (k)$, so  inequality \eqref{lemma:sequences:eq00} becomes
\[
\alpha_{k+1}
\leq
\frac{c}{c+q  k^{-1}\ln^{\expo_1}(k)  }\,\alpha_k
+  \frac{q \ln^{\expo_1 \frac{q}{q-1}}(k)  }{\bigl(c+q  k^{-1}\ln^{\expo_1}(k)  \bigr)k^\frac{q}{q-1}}  + \frac{c_0 \ln^{p_1}(k) }{ \bigl(c+q  k^{-1}\ln^{\expo_1}(k) \bigr)k^{p_2}}.
\]
Since $c+q  k^{-1}\ln^{\expo_1}(k)  \geq c$, we deduce that
\[
\alpha_{k+1}
\leq
\frac{c}{ c+q  k^{-1}\ln^{\expo_1}(k)  }\,\alpha_k
+
\left(\frac{q+c_0}{c} \right)\frac{\bigl(\ln(k)\bigr)^{ \max\{ p_1, \expo_1 \frac{q}{q-1}   \}  } }{k^{\min\{ p_2, \frac{q}{q-1} \}}}.
\]
Moreover,
\[
\frac{c}{c+q  k^{-1}\ln^{\expo_1}(k) }
=
1-q\frac{1 }{\bigl(c+q k^{-1}\ln^{\expo_1}(k) \bigr)} \frac{\ln^{\expo_1}(k)}{ k} .
\]
Since
\[
q\frac{1 }{(c+q \ln^{\expo_1}(k) k^{-1})}   \to \frac{q}{c}, \text{ as } k \to\infty,
\]
there exists $k_0\in\mathbb N$ such that for all $k\geq k_0$,
\[
q\frac{1 }{(c+q \ln^{\expo_1}(k) k^{-1})}  \geq \frac{q  }{2c}.
\] 
Therefore, for all $k\geq k_0$,
\[
\alpha_{k+1}
\leq
\left(1- \frac{q  }{2c} \frac{\ln^{\expo_1}(k)}{k}  \right)\alpha_k
+
\left(\frac{q+c_0}{c} \right)\frac{\bigl(\ln(k)\bigr)^{ \max\{ p_1, \expo_1 \frac{q}{q-1}   \}  } }{k^{\min\{ p_2, \frac{q}{q-1} \}}}.
\]  

Now, considering the particular case where $p_1=q$,   $\expo_1^\ast = q-1$ and $\expo_2^\ast := \max\{ p_1, \expo_1 \frac{q}{q-1}\}=q $, we simply get  the inequality 
\[
\alpha_{k+1}
\leq
\left(1- \frac{q  }{2c} \frac{\ln^{\expo_1^\ast }(k)}{k}  \right)\alpha_k
+
\left(\frac{q+c_0}{c} \right)\frac{\ln^{ \expo_2^\ast      }(k) }{k^{\min\{ p_2, \frac{q}{q-1} \}}}.
\]
Now, noticing that $\expo_2^\ast < 3 \expo_1^\ast +1 \Leftrightarrow q > 1 $, Lemma~\ref{lemmaA1} implies that \[ \alpha_k = \mathcal{ O  } \left( \frac{\ln(k)}{  k^{\min\{ p_2, \frac{q}{q-1} \} -1} }   \right). \] 
This proves item (iii).
    \end{proofoflemma}

\end{document}